\documentclass[11pt,leqno,twoside,a4paper]{article}
\usepackage[english]{babel}
\usepackage[utf8]{inputenc}
\usepackage{authblk}
\usepackage{amsfonts,amssymb,amsmath,amsthm}
\usepackage{enumitem}
\usepackage[a4paper,top=25mm,bottom=30mm,left=30mm,right=30mm]{geometry}
\usepackage{graphicx} 
\providecommand{\widecheck}[1]{\check{#1}}
\usepackage{verbatim}
\usepackage{calc}
\usepackage{color}
\usepackage{xcolor}
\usepackage{lineno}
\usepackage[mathscr]{eucal}
\usepackage{needspace}
\usepackage{lmodern}     
\usepackage[T1]{fontenc}
\usepackage{microtype}
\usepackage{ifthen}
\usepackage[hidelinks,hyperfootnotes=false]{hyperref}
\usepackage{resmes}

\providecommand{\appendixname}{Appendix}

\usepackage{fancyhdr}
\fancypagestyle{plain}{%
	\fancyhf{}
	\fancyhead[LE,LO]{\thepage}

}

\newlength{\abstractleft}
\newlength{\abstractright}
\makeatletter
\renewenvironment{abstract}{%
	\par\small
	\list{}{%
		\setlength{\leftmargin}{\abstractleft}%
		\setlength{\rightmargin}{\abstractright}%
		\setlength{\listparindent}{0pt}%
		\setlength{\itemindent}{0pt}%
		\setlength{\parsep}{0pt}%
	}%
	\item\relax
	\noindent\textsc{Abstract.}\ %
	\ignorespaces
}{%
	\endlist
	\par\normalsize
}
\makeatother	



\newcommand{\ShortAuthors}{\scriptsize L. CARRERO, A. QUAAS, AND A. ZUNIGA} 
\newcommand{\ShortTitle}{\scriptsize
	Anisotropic capillarity and transported ABP rigidity} 

\fancypagestyle{plain}{%
	\fancyhf{}
	\fancyfoot[C]{\thepage}%
}

\usepackage{etoolbox}
\newif\ifinappendix
\pretocmd{\appendix}{%
	\inappendixtrue
	\addtocontents{toc}{\protect\inappendixtrue}%
}{}{}

\makeatletter

\renewcommand{\@seccntformat}[1]{%
	\ifinappendix
	\ifdefstring{#1}{section}{\appendixname~\thesection.\quad}{%
		\ifdefstring{#1}{subsection}{\appendixname~\thesubsection.\quad}{%
			\csname the#1\endcsname\quad
		}%
	}%
	\else
	\csname the#1\endcsname\quad
	\fi
}
\makeatletter

\newcommand{\TOCEntryFont}{}

\renewcommand{\l@section}[2]{%
	\begingroup
	\parskip=0pt
	\let\oldnumberline\numberline
	\ifinappendix
	\renewcommand{\numberline}[1]{\oldnumberline{\appendixname~##1.}}%
	\@dottedtocline{1}{1.5em}{6.8em}{{\TOCEntryFont #1}}{{\TOCEntryFont #2}}%
	\else
	\renewcommand{\numberline}[1]{\oldnumberline{##1.}}%
	\@dottedtocline{1}{1.5em}{3.0em}{{\TOCEntryFont #1}}{{\TOCEntryFont #2}}%
	\fi
	\endgroup
}

\renewcommand{\l@subsection}[2]{%
	\begingroup
	\parskip=0pt
	\let\oldnumberline\numberline
	\ifinappendix
	\renewcommand{\numberline}[1]{\oldnumberline{\appendixname~##1.}}%
	\@dottedtocline{2}{3.8em}{7.4em}{{\TOCEntryFont #1}}{{\TOCEntryFont #2}}%
	\else
	\renewcommand{\numberline}[1]{\oldnumberline{##1.}}%
	\@dottedtocline{2}{3.8em}{3.8em}{{\TOCEntryFont #1}}{{\TOCEntryFont #2}}%
	\fi
	\endgroup
}
\makeatother

\usepackage{titlesec}
\titleformat{\section}
{\normalfont\scshape\large\filcenter}
{\ifinappendix\appendixname~\thesection.\else\thesection.\fi}{0.8em}{}

\titleformat{\subsection}
{\normalfont\scshape\normalsize\filcenter}
{\ifinappendix\appendixname~\thesubsection.\else\thesubsection.\fi}{0.8em}{}

\titlespacing*{\section}{0pt}{1.5ex plus .2ex}{0.8ex}
\titlespacing*{\subsection}{0pt}{1.2ex plus .2ex}{0.6ex}

\titleformat{\subsubsection}
{\normalfont\scshape\normalsize}
{\thesubsubsection.}{0.8em}{}

\newtheorem{theorem}{Theorem}[section]
\newtheorem{lemma}[theorem]{Lemma}
\newtheorem{proposition}[theorem]{Proposition}
\newtheorem{corollary}[theorem]{Corollary}
\newtheorem{definition}[theorem]{Definition}
\newtheorem{remark}[theorem]{Remark}

\newcommand{\tension}{\varphi}
\newcommand{\Deficit}{\Theta}
\renewcommand{\P}{\operatorname{Per}}

\newcommand{\capW}[1][]{%
	\ifthenelse{ \equal{#1}{} }
	{\ensuremath{W^{\lambda}_{\tension}}}
	{{\ensuremath{W^{\lambda}_{\tension}[#1]}}}
}%

\newcommand{\capWalpha}[1][]{%
	\ifthenelse{ \equal{#1}{} }
	{\ensuremath{W^{\alpha}_{\tension}}}
	{{\ensuremath{W^{\alpha}_{\tension}[#1]}}}
}%

\newcommand{\capWv}[1][]{%
	\ifthenelse{ \equal{#1}{} }
	{\ensuremath{W^{\lambda}_{\tension}[v]}}
	{{\ensuremath{W^{\lambda}_{\tension}[#1]}}}
}%

\newcommand{\capWvalpha}[1][]{%
	\ifthenelse{ \equal{#1}{} }
	{\ensuremath{W^{\alpha}_{\tension}[v]}}
	{{\ensuremath{W^{\alpha}_{\tension}[#1]}}}
}%

\newcommand{\calH}{\mathcal{H}}

\newcommand{\C}{\mathbf{C}}

\renewcommand{\a}{\mathfrak{a}}

\newcommand{\equal}{ \ =\ }
\renewcommand{\H}{\mathbf{H}}
\newcommand{\E}{\mathcal{E}}
\newcommand{\EnergyC}{\E_{\tension,\lambda,\C}}

\newcommand{\N}{\mathbb{N}}

\newcommand{\M}{\mathsf{M}}

\newcommand{\PerAni}{\P_{\tension}}
\newcommand{\R}{\mathbb{R}}

\newcommand{\wulff}{W_{\tension}}
\newcommand{\wcap}{W_{\tension,\lambda}}

\newcommand{\x}{\mathbf{x}}

\newcommand{\Bsubdiff}[2]{%
	\ifthenelse{ \equal{#1}{} }
	{\ifthenelse{ \equal{#2}{} }
		{\ensuremath{\mathscr{B}^{\lambda,u_{\Gamma}}_{\a,\tension}}}
		{\ensuremath{\mathscr{B}^{\lambda,#2}_{\a,\tension}}}
	}
	{\ensuremath{\mathscr{B}^{#1,#2}_{\a,\tension}}}
}%
\newcommand{\mres}{\mathbin{\vrule height 1.5ex depth 0ex width 0.13ex\vrule height 0.13ex depth 0ex width 0.7em}}

\DeclareMathOperator{\dist}{dist}

\colorlet{darkred}{red!80!black}
\colorlet{darkblue}{blue!80!black}
\colorlet{darkgreen}{green!60!black}

\providecommand{\keywords}[1]
{
	\small
	\textit{Keywords: #1}
}

\title{\large\bfseries\MakeUppercase{
		Sharp planar anisotropic capillarity outside convex obstacles and
		two-point rigidity of transported boundary images
}} 
\author{\small\MakeUppercase{Lisbeth Carrero, \; Alexander Quaas,\; and\; Andres Zuniga}}
\date{}

\newcommand{\addressblock}[1]{%
	\par\noindent 
	{\scshape #1}\par\medskip
}

\newcommand{\EmailB}[1]{%
	\textit{E-mail address}: %
	{\upshape\ttfamily\bfseries\href{mailto:#1}{#1}}%
}

\makeatletter
\renewcommand{\tableofcontents}{%
	\@starttoc{toc}%
}
\makeatother

\usepackage{tikz}
\usetikzlibrary{arrows.meta,patterns}
\usepackage{caption}
\usepackage{flafter} 
\usepackage[section]{placeins}
\newcommand{\CapFigureScale}{0.84}
\definecolor{capInk}{HTML}{263445}
\definecolor{capBlue}{HTML}{0072B2}
\definecolor{capAmber}{HTML}{E69F00}
\definecolor{capGreen}{HTML}{009E73}
\definecolor{capPurple}{HTML}{AA4499}
\definecolor{capGrey}{HTML}{87929E}
\tikzset{
	cap panel/.style={x=1cm,y=1cm,font=\small,line cap=round,line join=round},
	cap edge/.style={draw=capInk,line width=.8pt},
	cap obstacle/.style={fill=capGrey!12,postaction={pattern=north east lines,pattern color=capGrey!65}},
	cap liquid/.style={fill=capBlue!22},
	cap added/.style={fill=capAmber!30,postaction={pattern=north west lines,pattern color=capAmber!85!black}},
	cap guide/.style={draw=capInk!55,line width=.45pt,densely dashed},
	cap leader/.style={draw=capInk!70,line width=.45pt},
	cap normal/.style={draw=capInk,line width=.65pt,-{Stealth[length=2mm]}},
	cap secondary/.style={fill=capGreen!23},
	cap reference/.style={draw=capGreen!65!black,line width=.75pt,densely dashed},
	cap difference/.style={fill=capPurple!22,postaction={pattern=north west lines,pattern color=capPurple!85!black}},
	cap label/.style={fill=white,inner sep=1.5pt}
}
\providecommand{\CapPanel}[1]{%
	\path[use as bounding box] (0,0) rectangle (7.2,5.1);
	\node[anchor=north west,font=\bfseries] at (.05,5.08) {(#1)};
}
\pgfmathdeclarefunction{cappocket}{1}{%
	\pgfmathparse{(#1>45 && #1<75) ? .30*sin(6*(#1-45))^2 :
		((#1>105 && #1<135) ? .30*sin(6*(#1-105))^2 : 0)}%
}
\pgfmathdeclarefunction{capinner}{1}{%
	\pgfmathparse{1+.26*(max(0,(.3025-sin(#1)^2)/.3025))^2}%
}
\def\CapEps{.05}
\pgfmathdeclarefunction{capwx}{1}{%
	\pgfmathparse{(1+\CapEps*cos(4*#1))*cos(#1)
		+4*\CapEps*sin(4*#1)*sin(#1)}%
}
\pgfmathdeclarefunction{capwy}{1}{%
	\pgfmathparse{(1+\CapEps*cos(4*#1))*sin(#1)
		-4*\CapEps*sin(4*#1)*cos(#1)}%
}
\newcommand{\CapRing}{%
	\path[cap liquid,even odd rule] (0,0) circle (1.6)
	plot[domain=0:360,samples=241] ({capinner(\x)*cos(\x)},
	{capinner(\x)*sin(\x)}) -- cycle;
	\fill[white] (0,0) circle (1);
	\fill[cap obstacle] (0,0) circle (1);
	\draw[cap edge] (0,0) circle (1.6);
	\draw[cap edge] plot[domain=0:360,samples=241]
	({capinner(\x)*cos(\x)},{capinner(\x)*sin(\x)}) -- cycle;
	\draw[cap guide] (0,0) circle (1);
}

\begin{document}
	\maketitle
	\thispagestyle{plain}
	\vspace{-2.25em} 
	
	\begingroup
	\renewcommand{\thefootnote}{}
	\footnotetext{\raggedright%
		\textit{Date:} October 27, 2026.\par 
		\keywords{Anisotropic capillarity; Wulff shapes; Alexandrov--Bakelman--Pucci method; anisotropic perimeter; crystalline anisotropies; convex obstacles.
		}\par
		2020 \textit{Mathematics Subject Classification.} Primary 49Q20; Secondary 49Q10, 52A40, 76B45.\par
		\textit{Affiliations:} Lisbeth Carrero and Andres Zuniga, Instituto de Ciencias de la Ingenier\'ia, Universidad de O'Higgins, Chile.
		Alexander Quaas, Departamento de Matem\'atica, Universidad T\'ecnica Federico Santa Mar\'ia, Chile\par
	}
	\endgroup
	
	\setcounter{footnote}{0}
	
	\begin{abstract}
		For a convex surface tension $\tension$ and the wall-energy density $\sigma(n)=-\lambda\,\tension(-n)$, we prove sharp capillary inequalities outside arbitrary convex planar obstacles for $-1<\lambda\leq0$ and every finite-perimeter competitor of finite positive volume, including non-even and crystalline tensions. Equality holds precisely for translated dilates of Wulff caps in minimum-volume directions, with the entire truncating facet resting on a flat supporting face. A uniform quantitative estimate controls component concentration, shape up to translation, and directional cap-volume excess. The proof combines exact filling, planar completion, and quantitative Wulff stability. For the transport $D\Phi$, $\Phi=\tension^2/2$, the zero-adhesion boundary-image estimate characterizes centered ellipsoids among even anisotropies with $\Phi\in C^3$ away from the origin and positive definite Hessian. Two-point slab tests with arbitrary boundary data suffice: their projection balances yield a cubic identity forcing $\Phi$ to be quadratic. Facets and smooth perturbations provide counterexamples to the transported estimate, including failures at prescribed adhesion levels in $(-1,1)$. Strict finite deficits persist under Hausdorff perturbations of the Wulff body. The classification concerns an unrestricted boundary-data property, stronger than the image estimate needed only for traces of auxiliary Neumann solutions.
	\end{abstract}
	
	\begin{center}
		{\large\textsc{Contents}}
	\end{center}
	\vspace{-.75em}
	\setcounter{tocdepth}{1}
	\tableofcontents
	
	\numberwithin{equation}{section}
	
	\vspace{1em}
	
	\section{Introduction and main results}\label{sec:introduction}
	Capillary energy balances the free surface of a drop against adhesion to its support. Isotropic equilibria on a flat wall are spherical caps~\cite{finn1986equilibrium,wente1980symmetry}. For convex obstacles, the orthogonal-contact comparison of Choe--Ghomi--Ritor\'e~\cite{choe2006total,choe2007relative} was followed by equality and ABP results~\cite{fusco2023total,liu2023relative} and the convex-cylinder comparison~\cite{fusco2025isoperimetriccylinders}. Fusco--Julin--Morini--Pratelli~\cite{fusco2026isoperimetric} proved the isotropic comparison for arbitrary convex obstacles in every dimension throughout the partial-wetting range. Nani--Vincini~\cite{nani2026generalized} developed generalized existence, describing escape to infinity through asymptotic cylinders and profile convergence under obstacle variations.
	
	The anisotropic full-space model is the Wulff body~\cite{wulff1901zur}, whose variational theory was developed by Taylor~\cite{taylor1974existence,taylor1978crystalline}, Fonseca~\cite{fonseca1991wulff}, and Fonseca--M\"uller~\cite{fonseca1991uniqueness}. The ABP proof in~\cite{cabre2016sharp} and quantitative estimates in~\cite{esposito2005quantitative,figalli2010mass,neumayer2016strong} provide comparison and stability tools used below. On a flat wall, the model is the Winterbottom construction~\cite{winterbottom1967equilibrium}, with statistical-mechanical and corner extensions~\cite{kotecky1994equilibrium,zia1988summertop}. Minimality is proved in~\cite{kholmatov2025minimality,kreutz2026winterbottom}, and quantitative stability in~\cite{kreutz2026winterbottom}. For a class of anisotropies admitting symmetrization, Baer~\cite{baer2015minimizers} treats sessile drops with gravity. Related smooth theories include free-boundary regularity and Young's law~\cite{dephilippis2015regularity}, the half-space capillary Alexandrov theorem~\cite{jia2023sharp}, and classifications of stable anisotropic capillary hypersurfaces in wedges~\cite{koiso2023stable} and a half-space~\cite{guo2023stable}. These classifications concern smooth critical points under their respective geometric assumptions.
	
	Equilibrium capillarity provides a mathematical framework for supported crystals and catalyst droplets. Exterior models for nanowire supports~\cite{fonseca2023global}, anisotropic particle models on substrates~\cite{bao2017stable}, and substrate-bound nanoparticle experiments~\cite{lewis2020winterbottom} exhibit related balances between surface and adhesion energies. Discrete crystallization results also recover Winterbottom shapes from particle interactions~\cite{friedrich2026crystallization}. These connections motivate sharp energy comparisons and quantitative control of near-minimizers beyond flat supports. Here the prescribed direction-dependent wall law permits a precise study of how obstacle geometry and anisotropy affect such comparisons. The results concern this equilibrium variational model; they do not describe growth kinetics, strain, or selection among metastable states.
	
	For general anisotropies, cap volume depends on the supporting direction and nonsymmetry affects the wetted-boundary term. Tu~\cite{tu2007anisotropic} proved a zero-adhesion exterior comparison for a Euclidean ball by a linear ABP argument; the flat-wall results above and the isotropic stability estimates in~\cite{pascale2024quantitative,carazzato2025quantitative} concern other support geometries. Here we treat arbitrary convex planar obstacles, including non-even and crystalline tensions, for $-1<\lambda\leq0$. This is the sign range in which the nonnegative contact excess in the filling identity contributes $-\lambda\,\Delta_\tension\geq0$. We combine hyperplane filling with the exterior completion of M\"ader-Baumdicker--Neumayer--Park--Rupflin~\cite[Lemma~5.7, Step~2]{mader2025quantitative}, retaining the reflected anisotropic perimeter. Their stability uses the fixed-obstacle profile under geometric nondegeneracy assumptions; ours uses the universal cap bound and freely translated caps, as distinguished in~\eqref{eq:universal-profile-deficit}.\vspace{-1em}
	
	\paragraph{Contribution relative to previous comparison results.}
	Theorem~\ref{thm:unrestricted-planar} extends the exterior comparison from the isotropic setting to arbitrary planar convex tensions for $-1<\lambda\leq0$, including non-even and crystalline tensions. Theorem~\ref{thm:unrestricted-planar-stability} below transfers quantitative Wulff stability to the universal exterior cap bound, with a constant independent of the obstacle, tension, and adhesion parameter. This comparison concerns shape up to free translation, whereas stability relative to the profile of a fixed obstacle is distinguished in~\eqref{eq:universal-profile-deficit}. Independently, we introduce the transported $\lambda$-ABP property (Definition~\ref{rev:def:abp}) and classify its zero-adhesion case in the smooth even class: Theorem~\ref{thm:two-point-rigidity} shows that validity for two-point data with arbitrary heights forces an ellipsoidal Wulff body. The two parts address the same exterior comparison problem from complementary directions: the first establishes the inequality by a variational construction, while the second identifies the rigidity of one natural boundary-image mechanism. Their logical independence makes precise the distinction between validity of the capillary inequality and validity of that mechanism.
	
	In what follows, we consider $\tension:\R^d\to[0,\infty)$,  a convex and positively one-homogeneous function satisfying
	\[
	c_\tension|q|\leq\tension(q)\leq C_\tension|q|\quad\text{ for every }q\in\R^d,
	\]
	for fixed constants $0<c_\tension\leq C_\tension<\infty$. No symmetry of $\tension$ is assumed unless stated explicitly. Its associated Wulff body is the compact convex set
	\[
	\wulff=\{p\in\R^d:\, p\cdot\nu\leq\tension(\nu)\;\,
	\text{ for every }\nu\in\mathbb{S}^{d-1}\}.
	\]
	Given a nonempty compact convex set $K\subset\R^d$, we denote its support function by 
	\[
	h_K(q):=\max_{p\in K}p\cdot q\quad \text{ for every }\; q\in\R^d.
	\]
	In particular, $0\in\operatorname{int}\wulff$ and $h_{\wulff}=\tension$. Throughout, \emph{an obstacle} is a proper closed convex set $\C\subset\R^d$ with nonempty interior. We reserve bold $\C$ for the obstacle, $C$ (with subscripts when needed) for estimate constants, and $B_r:=\{x\in\R^d:|x|\le r\}$ for the closed Euclidean ball of radius $r>0$ centered at the origin. for the Euclidean ball of radius $r>0$ centered at the origin. The wetted set is $\Gamma$, the lower contact set is $\Omega^+$, and the transported boundary image is denoted by $\mathscr B$ when abbreviated. Throughout, $d\geq 2$ denotes dimension and $\Deficit(E)$ denotes the normalized planar capillary energy deficit defined in~\eqref{eq:normalized-capillary-deficit} below; $\delta$ in the two-point rigidity argument denotes slab thickness.
	
	For an adhesion coefficient $\lambda\in\R$ and a set $E\subset\R^d\setminus\C$ of finite perimeter and finite positive volume, we define its \emph{anisotropic capillary energy} by
	\begin{equation}\label{v12:eq:energy}
		\EnergyC(E)
		:=\PerAni(E;\R^d\setminus\C)-\lambda\int_{\partial\C}\tension(-\nu_{\C})\,T_E\,d\calH^{d-1}.
	\end{equation}
	Here
	\[
	T_E := \operatorname{Tr}^{\mathrm{ext}}_{\partial\C}\chi_E
	\]
	denotes the exterior $BV$-trace of $\chi_E$ on $\partial\C$, taken from $\R^d\setminus\C$; see~\cite[Theorem~3.87]{ambrosio2000functions}, applied locally.
	The $\tension$-anisotropic perimeter functional is given by
	\[
	\PerAni(E;U)=\int_{\partial^*E\cap U}\tension(\nu_E)\,d\calH^{d-1}.
	\]
	Here $\partial^*E$ denotes the reduced boundary of $E$, $\nu_E$ its measure-theoretic outer unit normal, and $\calH^{d-1}$ the $(d-1)$-dimensional Hausdorff measure; see \cite[Chapters~3 and~15]{maggi2012sets}. The outer unit normal $\nu_{\C}$ to the obstacle is defined $\calH^{d-1}$-almost everywhere on $\partial\C$. The first term of~\eqref{v12:eq:energy} measures the anisotropic free surface. The prescribed wall-energy density is $\sigma_{\tension,\lambda}(n)=-\lambda\,\tension(-n)$: $\lambda$ is dimensionless, whereas the wall coefficient depends on the supporting direction. All results below concern this wall law. On the dry part of $\partial\C$, $T_E=0$; on the wetted part, $T_E=1$ and $\nu_E=-\nu_\C$ almost everywhere. For the finite-perimeter framework, see~\cite[Chapters~15 and~20]{maggi2012sets}.
	More generally, for a continuous, positively one-homogeneous function $f:\R^d\to\R$, a set $A$ of finite perimeter, and a Borel set $V\subset\R^d$, we write
	\[
	P_f(A;V):=\int_{\partial^*A\cap V}f(\nu_A)\,d\calH^{d-1},
	\qquad P_f(A):=P_f(A;\R^d).
	\]
	Thus $P_\tension(A;V)=\PerAni(A;V)$. The notation $P_f$ also accommodates integrands that need not be convex or nonnegative; in particular, a support function $h_K$ need not be positive away from the origin when $0\notin\operatorname{int}K$.
	
	If $\tension_b(q)=\tension(q)+b\cdot q$ remains positive away from zero, Gauss--Green for the zero extension of $E$ gives
	\begin{equation}\label{eq:wall-gauge-dependence}
		\E_{\tension_b,\lambda,\C}(E)-\EnergyC(E)
		=(1+\lambda)\int_{\partial\C}b\cdot\nu_\C\,T_E\,d\calH^{d-1}.
	\end{equation}
	The required trace integrability is established in Lemma~\ref{rev:lem:coercivity}. Thus translating the Wulff body preserves full-space anisotropic perimeter but generally changes this capillary energy. Related gauge reformulations absorb prescribed constant wetting terms on a flat wall~\cite{lu2023capillary} or compatible wedge faces~\cite{ma2026wedge}. Those reformulations transform the wall term together with the surface tension, whereas~\eqref{eq:wall-gauge-dependence} keeps the wall law tied to the chosen tension.
	
	The partial-wetting range depends on the direction. For a given anisotropy $\tension$, set
	\begin{equation}\label{v12:eq:range}
		\lambda_*(\tension):=\min_{n\in\mathbb{S}^{d-1}}\frac{\tension(n)}{\tension(-n)}.
	\end{equation}
	We assume $-1<\lambda<\lambda_*(\tension)$ whenever all directions are considered. For a specified direction $n\in\mathbb{S}^{d-1}$, the corresponding condition is $-1<\lambda<\tension(n)/\tension(-n)$. These conditions ensure that the cutting hyperplane meets the interior of $\wulff$. In particular, $\lambda_*(\tension)=1$ for even surface tensions. 
	
	Define the \emph{Wulff cap}, its \emph{directional volume}, and the \emph{minimum cap volume} by
	\begin{align}
		K_{\tension,\lambda}(n)
		&:=\wulff\cap\{p:p\cdot n\geq\lambda\,\tension(-n)\},\label{v12:eq:cap}\\
		m_{\tension,\lambda}(n)&:=|K_{\tension,\lambda}(n)|,\quad
		m_{\tension,\lambda}^{\min}:=\min_{n\in\mathbb{S}^{d-1}}m_{\tension,\lambda}(n),\label{v12:eq:min-cap}
	\end{align}
	respectively. In the uniform partial-wetting range, the minimum $m_{\tension,\lambda}^{\min}$ is positive. Indeed, the cap volume is positive and continuous in the direction $n$. Here $n$ is the outer normal of the supporting wall, whereas the truncating facet of the model cap has outer normal $-n$. The cap occupies the side of increasing $p\cdot n$. Its volume may depend on $n$, which is why the direction-independent comparison uses $m_{\tension,\lambda}^{\min}$.
	
	For $d=2$ and $-1<\lambda\leq0$, we define the \emph{normalized capillary energy deficit} of a finite-perimeter set $E\subset\R^2\setminus\C$ with $0<|E|<\infty$ by
	\begin{equation}\label{eq:normalized-capillary-deficit}
		\Deficit(E):=\frac{\EnergyC(E)}{2\sqrt{m_{\tension,\lambda}^{\min}|E|}}-1.
	\end{equation}
	This measures the relative energy excess above the universal capillary bound and is nonnegative by Theorem~\ref{thm:unrestricted-planar}. We write $\Deficit=\Deficit(E)$ when the competitor is fixed.
	
	For ease of reference, the table below collects the principal notation used in the variational comparison and the transported ABP framework.
	\begin{center}
		\small
		\begin{tabular}{p{.24\textwidth}p{.66\textwidth}}
			\hline
			\textbf{Symbol} & \textbf{Meaning} \\
			\hline
			$d$, $\Deficit(E)$ & Ambient dimension; normalized capillary deficit~\eqref{eq:normalized-capillary-deficit} \\
			$\C$, $U$ & Convex obstacle; exterior $\R^d\setminus\C$ \\
			$T_E$ & Exterior $BV$ trace $\operatorname{Tr}^{\mathrm{ext}}_{\partial\C}\chi_E$, evaluated from $\mathbb R^d\setminus\C$ \\
			$\Gamma$ & Wetted set (or wetted boundary piece) \\
			$K_{\tension,\lambda}(n)$, $\mathcal X$ & Model Wulff cap; compact boundary data set \\
			$\mathcal A_i$, $\mathcal S$ & Slope cells; union $E\cup\C$ in completion \\
			$\mathscr J$, $j_\Phi$ & Jordan curve; Jacobian density $\det D^2\Phi$ \\
			$\mathsf T$, $\mathcal T$ & Gradient transport $D\Phi$; cubic tensor $D^3\Phi$ \\
			$Q_\tension$, $\mathsf B$, $\mathscr R$ & Unit slope region; obstacle matrix; completion energy gain \\
			$\mathcal I$, $C$, $C_{\mathrm{stab}}$ & Adhesion interval; estimate constants; stability constant \\
			\hline
		\end{tabular}
	\end{center}
	
	\paragraph{Main results.} The two main contributions are a variational capillary comparison and a classification of the associated transported boundary-image property.
	\begin{enumerate}[label=$\bullet$,itemsep=0em]
		\item \emph{Sharp planar comparison and uniform stability.}
		For every convex planar obstacle, every surface tension satisfying the standing assumptions, and every exterior finite-perimeter competitor of finite positive volume, Theorem~\ref{thm:unrestricted-planar} gives
		\[
		\EnergyC(E)\geq2\sqrt{m_{\tension,\lambda}^{\min}|E|},
		\qquad -1<\lambda\leq0.
		\]
		Equality holds precisely for translated dilates of Wulff caps in minimizing directions whose entire truncating facet rests on a flat supporting face. Theorem~\ref{thm:unrestricted-planar-stability} gives uniform control of concentration in one indecomposable component, symmetric difference from a freely translated cap, and directional cap-volume excess, by transferring quantitative Wulff stability~\cite{figalli2010mass} through planar completion.
		
		\item \emph{Two-point characterization of ellipsoids.}
		For $d\geq2$, Theorem~\ref{thm:two-point-rigidity} characterizes centered ellipsoids by the transported zero-adhesion estimate among even tensions with $\Phi=\tension^2/2\in C^3(\R^d\setminus\{0\})$ and $D^2\Phi>0$. Two-point slab configurations with arbitrary boundary heights suffice. Their tangential projection balances yield a cubic identity forcing $\Phi$ to be quadratic. Facet obstructions, smooth failures at prescribed adhesion levels, and Hausdorff persistence of strict finite deficits complement this classification (Section~\ref{rev:sec:obstruction} and Appendix~\ref{app:perturbation}).
	\end{enumerate}
	
	The variational mechanism is the exact identity
	\[
	\EnergyC(E)
	=P_{h_K}(E\cup D)+R_K(E)-\lambda\,\Delta_\tension(\Gamma),
	\]
	where $K=K_{\tension,\lambda}(n)$, $h_K$ is its support function, and $P_{h_K}$ is the boundary functional defined above with $f=h_K$, integrated over the entire reduced boundary of $E\cup D$. Here $\Gamma$ is the wetted set and $D\subset\C$ is a hyperplane filling. The remainder $R_K(E)$ and the contact excess $\Delta_\tension(\Gamma)$ defined in Subsection~\ref{subsec:filling_identity} are nonnegative. Planar completion makes this identity applicable to arbitrary wetted traces when $-1<\lambda\leq0$. At zero adhesion, a further comparison treats two parallel wetted caps with arbitrary exterior connections in every dimension.
	
	Together, Theorems~\ref{thm:unrestricted-planar} and~\ref{thm:two-point-rigidity} show that every even nonellipsoidal planar tension in the regularity class of the latter satisfies the sharp capillary inequality at zero adhesion although its transported $0$-ABP property fails. The negative tests use arbitrary boundary data, which need not be traces of Neumann solutions. Proposition~\ref{rev:thm:abp} separates this geometric property from explicit analytic hypotheses; centered ellipsoids satisfy it by affine normalization.
	
	Section~\ref{sec:variat_comp_planar_stab} proves the planar comparison and stability, and a higher-dimensional result for two parallel wetted caps. Section~\ref{rev:sec:transport} develops the transported framework, including its graph measure; Section~\ref{rev:sec:ellipsoids} treats ellipsoids. Section~\ref{rev:sec:obstruction} proves facet obstructions, two-point rigidity, and the limitations of approximation. Section~\ref{sec:questions} formulates open problems. Appendix~\ref{app:perturbation} contains the perturbative construction in every dimension.
	\smallskip
	
	\section{Variational comparison and planar stability}\label{sec:variat_comp_planar_stab}
	For the contact geometries below, adjoining a portion $D$ of the obstacle to $E$ cancels the wetted interface and replaces it by a flat section. The auxiliary set $G=E\cup D$ can then be compared with a truncated Wulff body by the full-space Wulff inequality. Retaining the added volume and boundary remainders gives the exact identity used for equality and stability; $G$ itself need not be an admissible exterior competitor.
	
	For arbitrary planar traces, we first fill the \emph{exterior pockets} of $E$: this descriptive term denotes the finite-measure indecomposable components of $\R^2\setminus(E\cup\C)$, understood modulo Lebesgue-null sets. We then fill the appropriate portion of $\C$. Exterior completion decreases energy for every $\lambda>-1$; nonpositive adhesion is needed only to discard the contact excess in the second step. For $\lambda\leq0$, the inclusion $0\in K_{\tension,\lambda}(n)$ also makes cap dilations nested.
	
	We retain the standing assumptions on $\tension=h_W$ and the obstacle $\C$, and write $U=\R^d\setminus\C$.
	
	For $n\in\mathbb{S}^{d-1}$ and $s\in\R$, we use the notation
	\[
	\begin{aligned}
		& \Pi_s(n) :=\{x\in\R^d:\, x\cdot n=s\},\\
		& \H_s^+(n) :=\{x\in\R^d:\, x\cdot n>s\},
		\quad\text{ and }\quad \H_s^-(n):=\{x\in\R^d:\, x\cdot n<s\}.
	\end{aligned}
	\]
	These are the open half-spaces on the sides of
	increasing and decreasing $x\cdot n$, respectively,
	with common boundary $\Pi_s(n)$. Their closures are denoted by $\overline{\H_s^\pm(n)}$. The outer unit normals to $\H_s^+(n)$ and $\H_s^-(n)$ are $-n$ and $n$, respectively.
	
	In the capillary energy, traces on $\partial\C$ are
	always understood from the exterior of the obstacle.
	By Lemma~\ref{rev:lem:coercivity} below, finite volume and finite relative perimeter imply finite total perimeter, so all boundary integrals appearing below are finite.
	
	\subsection{Trace inequality and coercivity}\label{rev:subsec:coercivity}
	Convexity controls the wetted trace by a free perimeter with reflected tension, accounting for the wall orientation $\nu_E=-\nu_{\C}$. Comparing this tension with $\tension$ leaves a positive fraction of the total perimeter throughout the partial-wetting range.
	
	The estimate is independent of the transported boundary-image hypothesis and justifies extension by zero across the obstacle.
	\begin{lemma}[Exterior trace inequality and coercivity]\label{rev:lem:coercivity}
		Let $\C\subsetneq\R^d$ be closed and convex with nonempty interior, and put $U=\R^d\setminus\C$. Let $E\subset U$ satisfy
		\[
		0<|E|<\infty,
		\qquad
		\chi_E\in BV(U).
		\]
		Define
		\[
		\widecheck\tension(\nu):=\tension(-\nu),
		\qquad
		I_\tension(E;\C)
		:= \int_{\partial\C}\tension(-\nu_{\C})\,T_E\,d\calH^{d-1}.
		\]
		Then the extension of $\chi_E$ by zero to $\C$ belongs to $BV(\R^d)$, and
		\begin{equation}\label{rev:eq:trace}
			I_\tension(E;\C) 
			\leq \P_{\widecheck\tension}(E;U)
			\leq \frac{1}{\lambda_*(\tension)}
			\PerAni(E;U).
		\end{equation}
		In particular, if $k=\lambda_*(\tension)$ and $-1<\lambda<k$, then
		\begin{equation}\label{rev:eq:coercivity}
			\begin{split}
				\EnergyC(E)
				&\geq\frac{(k-\lambda)}{1+k}\PerAni(E;\R^d)\\
				&\geq\frac{(k-\lambda)}{1+k}d|\wulff|^{1/d}|E|^{(d-1)/d}.
			\end{split}
		\end{equation}
	\end{lemma}
	
	\begin{proof}[Proof of Lemma~\ref{rev:lem:coercivity}]
		
		The geometric input is that a flat cut cannot cost more than the boundary it replaces: convexity compares its single resultant normal with the original distribution of normals.
		For a convex one-homogeneous tension $\psi$ and a set $F$ of finite perimeter and finite measure, intersection with a half-space does not increase perimeter. Indeed, for $H=\H_s^-(n)$ and almost every $s$, BV-slicing and Gauss--Green on $F\cap H^c$ give
		\[
		\begin{aligned}
			P_\psi(F\cap H)&=P_\psi(F;H)+\psi(n)\calH^{d-1}(F^{(1)}\cap\partial H),\\
			\int_{\partial^*F\cap H^c}\nu_F\,d\calH^{d-1}
			&=n\calH^{d-1}(F^{(1)}\cap\partial H).
		\end{aligned}
		\]
		By convexity and homogeneity, the boundary term in the first line is bounded by $P_\psi(F;H^c)$, proving $P_\psi(F\cap H)\leq P_\psi(F)$. Exceptional levels follow by approximation from outside and lower semicontinuity. If a bounded convex set $K$ is contained in $F$, successive intersections with a countable family of supporting half-spaces decrease to $K$ in $L^1$; hence $P_\psi(K)\leq P_\psi(F)$, without symmetry assumptions on $\psi$.
		
		For bounded $\C$ and $E$ of finite total perimeter, apply this to $\C\subset\C\cup E$. BV gluing gives
		\[
		P_\psi(\C)\leq P_\psi(\C\cup E)
		=P_\psi(E;U)+P_\psi(\C)-\int_{\partial\C}\psi(\nu_\C)\,T_E\,d\calH^{d-1}.
		\]
		Choosing $\psi=\widecheck\tension$ proves the first trace inequality in this case.
		
		In general, the boundary of a convex obstacle with nonempty interior is locally Lipschitz. Thus, for almost every large $R$, local BV gluing makes $E_R=E\cap B_R$ a set of finite total perimeter. Apply the preceding argument to $\C_{2R}=\C\cap\overline{B_{2R}}$ and $E_R$. Since $E_R$ vanishes near the artificial boundary of $\C_{2R}$,
		\[
		\begin{aligned}
			\int_{\partial\C\cap B_R}\tension(-\nu_\C)\,T_E\,d\calH^{d-1}
			&\leq P_{\widecheck\tension}(E_R;U)\\
			&\leq P_{\widecheck\tension}(E;U\cap B_R)
			+C_\tension\int_{\partial B_R}\chi_E\,d\calH^{d-1}.
		\end{aligned}
		\]
		Finite volume makes the last term tend to zero along a sequence $R\to\infty$. Monotone convergence proves the trace inequality. The bound $\widecheck\tension\leq\tension/k$ proves its second part; trace integrability and BV gluing give finite total perimeter for the zero extension.
		
		Write $P=\PerAni(E;U)$ and $I=I_\tension(E;\C)$, so that $\PerAni(E)=P+I$ and $P-kI\geq0$. Then
		\[
		P-\lambda I-\frac{k-\lambda}{1+k}(P+I)
		=\frac{1+\lambda}{1+k}(P-kI)\geq0.
		\]
		The full-space Wulff inequality completes~\eqref{rev:eq:coercivity}.
	\end{proof}
	
	\subsection{The filling identity}\label{subsec:filling_identity}
	There are three fillings: close a proper wetted cap by its flat section, leave contact in a supporting face unchanged, or fill the interior of a compact obstacle. In each case the wetted interface is replaced by a planar boundary.
	\begin{enumerate}[label=\arabic*.,itemsep=0.25em]
		\item A hyperplane $\Pi_s(n)$ intersects $\operatorname{int}\C$, and $D=\C\cap \H_s^+(n)$ is a bounded convex cap. Set
		\[
		\Gamma:=\partial\C\cap \H_s^+(n),
		\qquad S:=\C\cap\Pi_s(n),
		\]
		and suppose $T_E=\chi_\Gamma$ almost everywhere on $\partial\C$.
		
		\item The wetted set $\Gamma$ is a bounded measurable subset of a supporting face of $\C$ with outer normal $n\in\mathbb{S}^{d-1}$, has positive $\calH^{d-1}$ measure, and $T_E=\chi_\Gamma$ almost everywhere. In this case set $D:=\emptyset$ and $S:=\Gamma$.
		
		\item $\C$ is compact, $\Pi_s(n)$ is a supporting hyperplane with $\C\subset\overline{\H_s^+(n)}$, and $S\subset\C\cap\Pi_s(n)$ is a compact subset of the corresponding face whose relative boundary is $\calH^{d-1}$-null. Set $D:=\operatorname{int}\C$ and $\Gamma:=\overline{\partial\C\setminus S}$, and suppose $T_E=\chi_\Gamma$ almost everywhere on $\partial\C$.
	\end{enumerate}
	In all cases put $\mathsf{A}_S:=\calH^{d-1}(S)$ and $G:=E\cup D$. Boundary sets and fillings are understood modulo the relevant null sets.
	
	Figure~\ref{fig:filling-flat-contact} illustrates the filling construction and the flat-contact equality configuration for a crystalline Wulff body and polyhedral obstacles.
	
	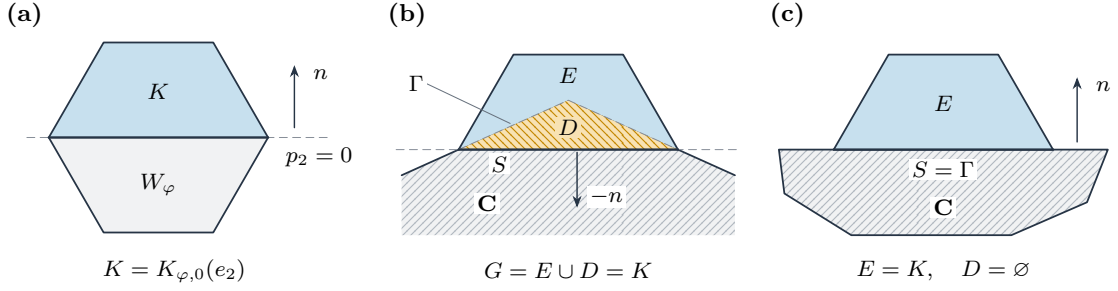
\begin{figure}[!htbp]
		\centering
		\begingroup
		\colorlet{crysInk}{capInk}
		\colorlet{crysBlue}{capBlue!22}
		\colorlet{crysGrey}{capGrey!12}
		\colorlet{crysAmber}{capAmber!30}
		\tikzset{
			crys edge/.style={draw=crysInk,line width=.67pt},
			crys guide/.style={draw=crysInk!55,line width=.38pt,densely dashed},
			crys leader/.style={draw=crysInk!70,line width=.38pt},
			crys normal/.style={draw=crysInk,line width=.54pt,-{Stealth[length=1.6mm]}},
			crys label/.style={fill=white,inner sep=1pt},
			crys obstacle/.style={pattern=north east lines,pattern color=capGrey!65},
			crys added/.style={pattern=north west lines,pattern color=capAmber!85!black}
		}
		\begin{tikzpicture}[x=1cm,y=1cm,font=\fontsize{8.4}{10}\selectfont,
			line cap=round,line join=round]
			\path[use as bounding box] (0,0) rectangle (15,3.85);
			\foreach \x/\letter in {0/a,5.05/b,10.10/c}{
				\node[anchor=north west,font=\fontsize{9.2}{11}\selectfont\bfseries]
				at (\x,3.80) {(\letter)};
			}
			
			\begin{scope}[shift={(2.15,1.88)},scale=1.45]
				\path[fill=crysGrey]
				(-1,0)--(-.5,{sqrt(3)/2})--(.5,{sqrt(3)/2})
				--(1,0)--(.5,{-sqrt(3)/2})--(-.5,{-sqrt(3)/2})--cycle;
				\path[fill=crysBlue]
				(-1,0)--(-.5,{sqrt(3)/2})--(.5,{sqrt(3)/2})--(1,0)--cycle;
				\draw[crys edge]
				(-1,0)--(-.5,{sqrt(3)/2})--(.5,{sqrt(3)/2})
				--(1,0)--(.5,{-sqrt(3)/2})--(-.5,{-sqrt(3)/2})--cycle;
				\draw[crys guide] (-1.20,0)--(1.57,0);
				\draw[crys edge,line width=.90pt] (-1,0)--(1,0);
				\node at (0,.42) {$K$};
				\node at (0,-.43) {$W_\tension$};
				\draw[crys normal] (1.24,.09)--(1.24,.66);
				\node[anchor=west] at (1.32,.58) {$n$};
				\node[anchor=west] at (1.07,-.18) {$p_2=0$};
			\end{scope}
			\node at (2.36,.12) {$K=K_{\tension,0}(e_2)$};
			
			\begin{scope}[shift={(7.57,1.72)},scale=1.45]
				\path[fill=crysGrey]
				(-1.52,-.78)--(-1.52,-.234)--(0,.45)
				--(1.52,-.234)--(1.52,-.78)--cycle;
				\path[crys obstacle]
				(-1.52,-.78)--(-1.52,-.234)--(0,.45)
				--(1.52,-.234)--(1.52,-.78)--cycle;
				\draw[crys edge] (-1.52,-.234)--(0,.45)--(1.52,-.234);
				\path[fill=crysBlue]
				(-1,0)--(-.5,{sqrt(3)/2})--(.5,{sqrt(3)/2})
				--(1,0)--(0,.45)--cycle;
				\path[fill=crysAmber] (-1,0)--(0,.45)--(1,0)--cycle;
				\path[crys added] (-1,0)--(0,.45)--(1,0)--cycle;
				\draw[crys guide] (-1.57,0)--(1.57,0);
				\draw[crys edge]
				(-1,0)--(-.5,{sqrt(3)/2})--(.5,{sqrt(3)/2})--(1,0);
				\draw[crys guide] (-1,0)--(0,.45)--(1,0);
				\draw[crys edge,line width=.90pt] (-1,0)--(1,0);
				\node at (0,.68) {$E$};
				\node[crys label,fill=crysAmber] at (0,.20) {$D$};
				\node[crys label] at (-.74,-.49) {$\C$};
				\node[crys label] at (-.63,-.13) {$S$};
				\draw[crys normal] (.08,-.025)--(.08,-.53);
				\node[crys label,anchor=west] at (.17,-.43) {$-n$};
				\draw[crys leader] (-1.29,.60)--(-.53,.2385);
				\node[anchor=east,crys label] at (-1.29,.60) {$\Gamma$};
			\end{scope}
			\node at (7.57,.12) {$G=E\cup D=K$};
			
			\begin{scope}[shift={(12.53,1.72)},scale=1.45]
				\path[fill=crysGrey]
				(-1.50,0)--(1.50,0)--(1.31,-.48)--(.62,-.78)
				--(-.84,-.78)--(-1.45,-.40)--cycle;
				\path[crys obstacle]
				(-1.50,0)--(1.50,0)--(1.31,-.48)--(.62,-.78)
				--(-.84,-.78)--(-1.45,-.40)--cycle;
				\draw[crys edge]
				(-1.50,0)--(1.50,0)--(1.31,-.48)--(.62,-.78)
				--(-.84,-.78)--(-1.45,-.40)--cycle;
				\path[fill=crysBlue]
				(-1,0)--(-.5,{sqrt(3)/2})--(.5,{sqrt(3)/2})--(1,0)--cycle;
				\draw[crys edge]
				(-1,0)--(-.5,{sqrt(3)/2})--(.5,{sqrt(3)/2})--(1,0);
				\draw[crys edge,line width=.90pt] (-1,0)--(1,0);
				\node at (0,.42) {$E$};
				\node[crys label] at (0,-.53) {$\C$};
				\node[crys label] at (0,-.18) {$S=\Gamma$};
				\draw[crys normal] (1.22,.07)--(1.22,.66);
				\node[anchor=west] at (1.30,.58) {$n$};
			\end{scope}
			\node at (12.53,.12) {$E=K,\quad D=\varnothing$};
		\end{tikzpicture}
		\endgroup
		
		\caption{Crystalline filling and flat contact
			(Lemma~\ref{rev:lem:filling} and Theorem~\ref{rev:thm:filling}),
			with $\lambda=0$ and $n=e_2$.
			(a) A centrally symmetric hexagonal Wulff body and its cap $K$.
			(b) Adding the amber $D\subset\C$ to $E$ removes $\Gamma$ and leaves
			$S$ with outer normal $-n$; the cone is unbounded.
			(c) The same $K$ gives equality on a polygonal obstacle.
			Grey hatching denotes $\C$; all directions minimize $m_{\tension,0}$. }
		\label{fig:filling-flat-contact}
	\end{figure}
	
	The wetted set and flat section have the same normal flux but may have different anisotropic costs. The contact excess measures this difference. It vanishes when one point of $W$ supports all reflected contact normals, which need not coincide for a crystalline $W$.
	
	\begin{lemma}[BV filling and contact excess]\label{rev:lem:filling}
		Under the preceding assumptions,
		\begin{align}
			|G|&=|E|+|D|,\label{v12:eq:hd-volume}\\
			\int_\Gamma \nu_{\C}&\,d\calH^{d-1}=\mathsf{A}_S\,n,\label{v12:eq:hd-flux}\\
			P_f(G)&=P_f(E;U)+\mathsf{A}_S\,f(-n)\label{v12:eq:hd-pasting}
		\end{align}
		for every continuous positively one-homogeneous function $f:\R^d\to\R$. Moreover, the contact excess function on $\Gamma$, $\Delta_\tension(\Gamma)$, is nonnegative:
		\[
		\Delta_\tension(\Gamma)
		:=\int_\Gamma\tension(-\nu_{\C})\,d\calH^{d-1}-\mathsf{A}_S\,\tension(-n)\geq0.
		\]
		Equality holds if and only if there exists $z\in W$ such that
		\begin{equation}\label{v12:eq:hd-calibration}
			z\cdot(-\nu_{\C})=\tension(-\nu_{\C})
			\qquad\calH^{d-1}\text{-almost everywhere on }\Gamma.
		\end{equation}
		Equivalently, $-\nu_{\C}\in N_W(z)$ almost everywhere, where
		$N_W(z):=\{\nu\in\R^d:\,z\cdot\nu=h_W(\nu)\}$ is the normal cone of $W$ at $z$. If $W$ is a polytope, $z$ can be chosen to be a vertex.
	\end{lemma}
	
	\begin{proof}[Proof of Lemma~\ref{rev:lem:filling}]
		In geometry~1, $D$ is a bounded convex body modulo its boundary. Its reduced boundary is $\Gamma\cup S$, with normals $\nu_\C$ and $-n$: the intersection of the cutting hyperplane with $\partial\C$ is the relative boundary of a convex section and is $\calH^{d-1}$-null. In geometry~3 the same boundary description holds for $D=\operatorname{int}\C$. Gauss--Green on $D$ proves~\eqref{v12:eq:hd-flux} in both cases. In geometry~2 it follows from $\nu_\C=n$ on $S=\Gamma$.
		
		Use $D\chi_F=(-\nu_F\calH^{d-1})\mres\partial^*F$ and extend $E$ by zero into $\C$. Lemma~\ref{rev:lem:coercivity} gives finite total perimeter. The BV trace formula~\cite[Section~3.8]{ambrosio2000functions} yields
		\[
		\begin{aligned}
			D\chi_E&=D\chi_E\mres U+(T_E\nu_\C\calH^{d-1})\mres\partial\C,\\
			D\chi_D&=(-\nu_\C\calH^{d-1})\mres\Gamma+(n\calH^{d-1})\mres S.
		\end{aligned}
		\]
		The second identity applies in geometries~1 and~3. Since $\chi_G=\chi_E+\chi_D$ and $T_E=\chi_\Gamma$, cancellation gives
		$D\chi_G=D\chi_E\mres U+(n\calH^{d-1})\mres S$.
		In geometry~2 this follows directly from $S=\Gamma$ and $\nu_\C=n$. The two measures on the right are mutually singular, proving~\eqref{v12:eq:hd-pasting}. Disjointness in volume gives~\eqref{v12:eq:hd-volume}.
		
		Choose $z\in W$ maximizing $z\cdot(-n)$, so that $z\cdot(-n)=h_W(-n)=\tension(-n)$. Taking the scalar product of~\eqref{v12:eq:hd-flux} with $-z$ gives
		\[
		\int_\Gamma z\cdot(-\nu_\C)\,d\calH^{d-1}
		=\mathsf{A}_S\,z\cdot(-n)
		=\mathsf{A}_S\,\tension(-n).
		\]
		Substituting this expression into the definition of $\Delta_\tension(\Gamma)$ yields
		\[
		\Delta_\tension(\Gamma)=\int_\Gamma
		\bigl(\tension(-\nu_\C)-z\cdot(-\nu_\C)\bigr)\,d\calH^{d-1}.
		\]
		The nonnegative integrand proves the excess bound and the necessity of~\eqref{v12:eq:hd-calibration}. Conversely, a calibrating $z$ gives $\int_\Gamma\tension(-\nu_\C)=\mathsf{A}_S\,z\cdot(-n)\leq\mathsf{A}_S\,\tension(-n)$, hence zero excess. The normal-cone formulation is the support equality. For a polytope, express $z$ as a convex combination of vertices. The corresponding nonnegative support deficits have weighted sum zero almost everywhere, so every vertex with positive weight also calibrates $\Gamma$.
	\end{proof}
	
	For a unit vector $n\in\mathbb{S}^{d-1}$ we assume the \emph{directional partial-wetting condition}, namely,
	\begin{equation}\label{v12:eq:hd-partial}
		-1<\lambda<\frac{\tension(n)}{\tension(-n)},
	\end{equation}
	and recall the truncated Wulff body with cutting facet $F_{\tension,\lambda}(n)$ defined by
	\[
	K_{\tension,\lambda}(n)
	=W\cap\overline{\H_{\lambda\,\tension(-n)}^+(n)},
	\qquad
	F_{\tension,\lambda}(n)
	=K_{\tension,\lambda}(n)\cap\Pi_{\lambda\,\tension(-n)}(n).
	\]
	It follows that $K_{\tension,\lambda}(n)$ has positive volume, $h_{K_{\tension,\lambda}(n)}\leq\tension$, and $h_{K_{\tension,\lambda}(n)}(-n)=-\lambda\,\tension(-n)$.
	
	Write $K=K_{\tension,\lambda}(n)$. Its support function gives the required density on the artificial flat boundary. The nonnegative remainder $R_K(E)$ measures the reduction from $\tension$ to $h_K$ on the free boundary.
	
	\begin{proposition}[Exact filling identity]\label{rev:prop:identity}
		Under the hypotheses of Lemma~\ref{rev:lem:filling} and~\eqref{v12:eq:hd-partial}, let $K=K_{\tension,\lambda}(n)$ and
		\[
		R_K(E)=\int_{\partial^*E\cap U}[\tension(\nu_E)-h_K(\nu_E)]\,d\calH^{d-1}\geq 0.
		\]
		Then
		\begin{equation}\label{rev:eq:filling-identity}
			\EnergyC(E)=P_{h_K}(G)+R_K(E)-\lambda\,\Delta_\tension(\Gamma).
		\end{equation}
		In particular,
		\begin{equation}\label{rev:eq:filling-bound}
			\EnergyC(E)
			\geq d|K|^{\frac 1d}(|E|+|D|)^{\frac{d-1}d}+R_K(E)-\lambda\,\Delta_\tension(\Gamma).
		\end{equation}
	\end{proposition}
	
	\begin{proof}[Proof of Proposition~\ref{rev:prop:identity}]
		Equation \eqref{v12:eq:hd-pasting} gives
		$P_{h_K}(G)=P_{h_K}(E;U)-\lambda\,\mathsf{A}_S\,\tension(-n)$.
		Subtracting this identity from the capillary energy proves~\eqref{rev:eq:filling-identity}.
		
		Translation removes any negative support values without changing the full-space perimeter, since the added linear term has zero boundary integral. Choose $z_0\in\operatorname{int}K$. The support function of $K-z_0$ is positive on the sphere and equals $h_K(\xi)-z_0\cdot\xi$. The Gauss--Green theorem gives $\int_{\partial^*G}\nu_G\,d\calH^{d-1}=0$, also for unbounded $G$ of finite volume and perimeter: apply compactly supported cutoffs to a constant vector field and let their radii tend to infinity. Consequently $P_{h_{K-z_0}}(G)=P_{h_K}(G)$. The finite-perimeter Wulff inequality gives
		\begin{equation}\label{v12:eq:hd-wulff}
			P_{h_K}(G)\geq d|K|^{1/d}|G|^{(d-1)/d}.
		\end{equation}
		Equality in~\eqref{v12:eq:hd-wulff} holds precisely when $G$ is a translate and positive dilation of $K$, modulo null sets, by Taylor~\cite{taylor1978crystalline} and Fonseca--M\"uller~\cite{fonseca1991uniqueness}. For the ABP proof of the inequality, see Cabr\'e--Ros-Oton--Serra~\cite[Theorem~1.2]{cabre2016sharp}. Here the identity $P_{h_K}(K)=d|K|$ follows from the Gauss--Green theorem and the support equality $p\cdot\nu_K=h_K(\nu_K)$ almost everywhere. This proves~\eqref{rev:eq:filling-bound}.
	\end{proof}
	
	In the two regimes below, both boundary remainders are nonnegative and the added volume improves the bound. Equality in the volume-only comparison therefore forces zero filling volume and no exposed truncating facet.
	
	\begin{theorem}[Capillary comparison for a hyperplane filling]\label{rev:thm:filling}
		Let $d\geq2$ and use one of the three fillings above. Suppose either $-1<\lambda\leq0$, or $0<\lambda<\tension(n)/\tension(-n)$ and $\Delta_\tension(\Gamma)=0$. With $K=K_{\tension,\lambda}(n)$,
		\begin{equation}\label{v12:eq:hd-comparison}
			\EnergyC(E)\geq d|K|^{1/d}(|E|+|D|)^{(d-1)/d}
			\geq d|K|^{1/d}|E|^{(d-1)/d}.
		\end{equation}
		Equality in the final, volume-only bound holds precisely when, modulo a null set, $E=x_0+tK$ for some $x_0\in\R^d$ and $t>0$, with its entire truncating facet $x_0+tF_{\tension,\lambda}(n)$ in a supporting face of $\C$ of outer normal $n$. In particular,
		\[
		\C\subset\overline{\H_{x_0\cdot n+t\lambda\,\tension(-n)}^-(n)}.
		\]
		The directional constant is optimal uniformly over this class of convex obstacles at every prescribed positive volume. Equality occurs only in geometry~2; geometries~1 and~3 give strict volume-only comparison.
		
		If also $-1<\lambda<\lambda_*(\tension)$, then
		\[
		\EnergyC(E)\geq d(m^{\min}_{\tension,\lambda})^{1/d}|E|^{(d-1)/d},
		\]
		with equality precisely in the preceding configurations in a minimizing direction. Recall that $m^{\min}_{\tension,\lambda}>0$ in this range and $\lambda_*(\tension)=1$ when $\tension$ is even.
	\end{theorem}
	
	\begin{proof}[Proof of Theorem~\ref{rev:thm:filling}]
		Both $R_K(E)$ and $-\lambda\,\Delta_\tension(\Gamma)$ in~\eqref{rev:eq:filling-bound} are nonnegative under either alternative. Equality in the volume-only bound forces $|D|=0$ by strict monotonicity of $v^{(d-1)/d}$, hence geometry~2. Equality in the Wulff inequality gives $E=x_0+tK$ almost everywhere, and $R_K(E)=0$.
		
		The relative interior of its truncating facet cannot meet $\operatorname{int}\C$, since this would force $|E\cap\C|>0$. A free part of positive $\calH^{d-1}$ measure would contribute $\tension(-n)-h_K(-n)=(1+\lambda)\tension(-n)>0$ to $R_K(E)$. Thus the facet is wetted almost everywhere and, by closedness of $\C$, lies entirely in $\partial\C$. If $\C$ had a point strictly on the cap side of its hyperplane, convexity would force its interior into the cap near a relative interior point of the facet, again contradicting disjointness. This proves the stated supporting half-space, without strict convexity of $W$.
		
		Conversely, on the free boundary of $K$ the support equality is $p\cdot\nu_K=\tension(\nu_K)$, while on its facet it is $p\cdot(-n)=-\lambda\,\tension(-n)$. Gauss--Green applied to $p\mapsto p$ therefore gives
		\[
		\PerAni(K;\H_{\lambda\,\tension(-n)}^+(n))
		-\lambda\,\tension(-n)\calH^{d-1}(F_{\tension,\lambda}(n))=d|K|.
		\]
		Dilation proves equality for $x_0+tK$. A sufficiently large supporting face of a compact convex prism realizes it at any volume. Finally, directional cap volume is continuous by dominated convergence off the cutting hyperplane. Compactness of the sphere gives a positive minimum in the uniform partial-wetting range, and equality in the minimized bound requires a minimizing direction.
	\end{proof}
	
	A translated rectangle already distinguishes opposite supporting directions.
	\begin{remark}[Directional cap volumes for a nonsymmetric anisotropy]
		For $a,b>0$ and $0<t<a$, the body $W=[t-a,t+a]\times[-b,b]$ has support function $\tension(\xi)=a|\xi_1|+b|\xi_2|+t\xi_1$. At zero adhesion,
		\[
		m_{\tension,0}(e_1)=2b(a+t),\qquad
		m_{\tension,0}(-e_1)=2b(a-t).
		\]
		Thus opposite directions can have different cap volumes even at
		$\lambda=0$, making minimization over directions essential.
		
		Translation preserves the rectangle's shape but changes the areas on either side of the central cut, information invisible to full-space perimeter.
	\end{remark}
	
	\subsection{Planar completion and sharp comparison}
	We first explain how planar contact geometry relates to the filling hypotheses, before removing the restriction on the wetted trace.
	\begin{remark}[Contact geometry and planar completion]
		The filling hypotheses are geometric restrictions on the wetted trace. In dimensions \(d\geq 3\), a connected wetted patch need not have its rim in a hyperplane, so connectedness alone does not imply the required filling. Conversely, geometry~2 permits disconnected traces in a single supporting face, and geometry~3 can also produce disconnected traces in higher dimensions.
		
		In the plane, every connected wetted trace admits one of the three fillings, as shown next. For nonpositive adhesion, the connectedness restriction can subsequently be removed by filling bounded complementary components outside the obstacle; see Theorem~\ref{thm:unrestricted-planar}. This completion changes the competitor before the filling comparison is applied.
	\end{remark}
	
	For a proper planar wetted arc, the line through its endpoints determines the filling: it either cuts out an obstacle cap or supports a flat face. Complete contact is handled by filling the whole obstacle. This gives the following comparison for connected contact.
	\begin{corollary}[Connected planar contact]\label{cor:connected-planar}
		Let $d=2$, $-1<\lambda\leq 0$, and let $E\subset\R^2\setminus\C$ have finite perimeter and $0<|E|<\infty$.
		
		Suppose that the exterior BV trace admits a connected representative, in the sense that there exists a connected Borel set $\Gamma\subset\partial\C$ such that
		\[
		T_E=\chi_\Gamma,
		\quad \calH^1\text{-almost everywhere
			on }\partial\C.
		\]
		Then
		\[
		\EnergyC(E)
		\geq
		2\sqrt{m_{\tension,\lambda}^{\min}|E|}.
		\]
		Equality holds precisely in the configurations characterized by Theorem~\ref{rev:thm:filling}, with a minimizing direction.
	\end{corollary}
	
	\begin{proof}[Proof of Corollary~\ref{cor:connected-planar}]
		Fix the connected representative $\Gamma$. Lemma~\ref{rev:lem:coercivity} gives $\calH^1(\Gamma)<\infty$. The boundary of a compact planar convex body is a rectifiable Jordan curve; the components of the boundary of an unbounded convex obstacle are proper curves homeomorphic to a line. Thus a connected subset of finite length is bounded. If its length is zero, the full-space Wulff inequality is strictly stronger than the cap bound. If $\C$ is compact and the trace is one almost everywhere on $\partial\C$, use geometry~3 with $S=\emptyset$.
		
		Otherwise, up to endpoints, $\Gamma$ is a bounded proper arc with distinct endpoints $a,b$. Let $\ell$ join them. If $\ell$ meets $\operatorname{int}\C$, its only boundary intersection points are $a,b$. The arc lies on one side of $\ell$ and bounds a convex cap with $[a,b]$. For boundedness, the closure of the arc together with $[a,b]$ is a Jordan curve. Let $B$ be its bounded interior and choose $n$ toward the side containing the arc. The curve lies in $\C$; convexity therefore implies $B\subset\C\cap\H_s^+(n)$. Put $O=\operatorname{int}\C\cap\H_s^+(n)$. This set is convex and connected, and $B$ is a nonempty open subset of $O$. Since $\partial B$ consists of the arc in $\partial\C$ and the chord in $\Pi_s(n)$, one has $\partial B\cap O=\emptyset$. Hence $B$ is also relatively closed in $O$, so $B=O$. Consequently the convex cap is $\overline B$, hence bounded. This is geometry~1.
		
		Otherwise $\ell$ supports $\C$ and $[a,b]\subset\partial\C$. The wetted arc is either this segment (geometry~2), or the alternative bounded boundary arc, possible only for compact $\C$ (geometry~3 with $S=[a,b]$). Theorem~\ref{rev:thm:filling} and $|K_{\tension,\lambda}(n)|\geq m^{\min}_{\tension,\lambda}$ prove the claim and identify equality.
	\end{proof}
	
	We now remove the connectedness assumption in Corollary~\ref{cor:connected-planar}. The argument uses two enlargements: planar completion adds exterior pockets, whereas the hyperplane filling adds a portion of $\C$ itself. Geometrically, the pockets are bounded complementary regions enclosed by the free boundary of $E$ and possibly part of $\partial\C$; they need not touch the obstacle. We first estimate the energy change for such a region under the boundary and trace conditions stated below, then construct the completion for a bounded obstacle and an indecomposable competitor. Componentwise summation and localization yield Theorem~\ref{thm:unrestricted-planar}. The reflected surface tension appears because
	the normal to the removed interface is reversed when viewed from the pocket.
	
	Adding an exterior pocket removes a free interface and may create wall contact. When $\lambda<0$, the trace inequality bounds the cost of that new contact by the removed interface.
	\begin{lemma}[Filling an exterior pocket]\label{lem:exterior-pocket}
		Let \(E,H\subset U\) be disjoint finite-perimeter sets of finite volume. Suppose that
		\[
		\partial^*H\cap U\subset\partial^*E\cap U,
		\qquad \nu_H=-\nu_E \quad \calH^{d-1}\text{-a.e. on }\partial^*H\cap U,
		\]
		where the inclusion is understood modulo \(\calH^{d-1}\)-null sets, and assume that \(T_E\,T_H=0\) almost everywhere on \(\partial\C\). Then
		\begin{equation} \label{eq:pocket-energy-identity}
			\EnergyC(E)-\EnergyC(E\cup H)
			= 
			\P_{\widecheck\tension}(H;U) +\lambda\,I_\tension(H;\C).
		\end{equation}
		In particular, if \(\lambda>-1\), then
		\begin{equation}\label{eq:pocket-energy-gain}
			\EnergyC(E)-\EnergyC(E\cup H)
			\geq 
			\min\{1,1+\lambda\}\P_{\widecheck\tension}(H;U).
		\end{equation}
	\end{lemma}
	
	\begin{proof}[Proof of Lemma~\ref{lem:exterior-pocket}]
		The interface $\partial^*H\cap U$ disappears on adjoining $H$, while the other reduced-boundary portions of $E$ remain unchanged. Since $\nu_E=-\nu_H$ on that interface,
		$\PerAni(E;U)-\PerAni(E\cup H;U)=P_{\widecheck\tension}(H;U)$.
		The disjoint traces satisfy $T_{E\cup H}=T_E+T_H$, so $I_\tension(E\cup H;\C)=I_\tension(E;\C)+I_\tension(H;\C)$. This proves~\eqref{eq:pocket-energy-identity}. Lemma~\ref{rev:lem:coercivity} gives $0\leq I_\tension(H;\C)\leq P_{\widecheck\tension}(H;U)$. Multiplication by $\lambda\leq0$ yields the factor $1+\lambda$; for $\lambda\geq0$ the adhesion contribution is nonnegative.
	\end{proof}
	
	For bounded planar obstacles, exterior completion connects the nonzero wetted trace of an indecomposable competitor. Indeed, disconnected completed contact would permit a cut inside the obstacle, splitting the competitor into two positive-volume parts with additive perimeters, contrary to indecomposability.
	\begin{lemma}[Planar completion]\label{lem:planar-completion}
		Let \(\C\subset\R^2\) be a convex body, and let \(E\subset\R^2\setminus\C\) be an indecomposable finite-perimeter set with
		\[
		0<|E|<\infty,\qquad I_\tension(E;\C)>0.
		\]
		There exists a bounded finite-perimeter set \(E^+\supset E\), contained in \(\R^2\setminus\C\), whose wetted trace has a connected representative.
		
		More precisely, let \(H_j\) be the exterior pockets, namely the finite-measure indecomposable components of \(\R^2\setminus(E\cup\C)\). Then
		\[
		E^+=E\cup\bigcup^{\infty}_{j=1} H_j \quad\text{modulo null sets}.
		\]
		For every \(\lambda>-1\),
		\begin{equation}\label{eq:completion-energy-gain}
			\EnergyC(E) 
			\geq 
			\EnergyC(E^+)+\min\{1,1+\lambda\}\sum^{\infty}_{j=1}\P_{\widecheck\tension}(H_j;U).
		\end{equation}
	\end{lemma}
	See Figure~\ref{fig:planar-completion} for the completion of $E$ and the crosscut configuration excluded in the proof.
	
	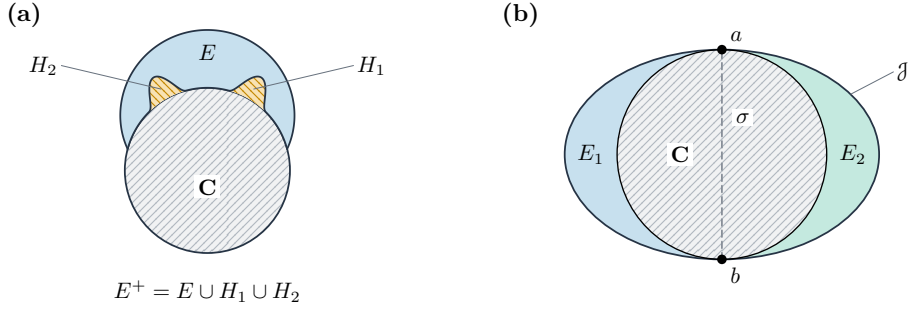
\begin{figure}[!htbp]
		\centering
		\scalebox{\CapFigureScale}{%
			\begin{minipage}{\textwidth}
				\centering
				\begin{tikzpicture}[cap panel]
					\CapPanel{a}
					\begin{scope}[shift={(3.35,2.30)},scale=1.30]
						\fill[white] (0,0) circle (1);
						\fill[cap obstacle] (0,0) circle (1);
						\draw[cap edge] (0,0) circle (1);
						\path[cap liquid]
						plot[domain=15:165,samples=150]
						({(1+.70*sin(1.2*(\x-15)))*cos(\x)},
						{(1+.70*sin(1.2*(\x-15)))*sin(\x)})
						-- plot[domain=165:15,samples=240]
						({(1+cappocket(\x))*cos(\x)},
						{(1+cappocket(\x))*sin(\x)}) -- cycle;
						\foreach \lo/\hi in {45/75,105/135}{
							\path[fill=white]
							plot[domain=\lo:\hi,samples=75]
							({(1+cappocket(\x))*cos(\x)},
							{(1+cappocket(\x))*sin(\x)})
							-- plot[domain=\hi:\lo,samples=50]
							({cos(\x)},{sin(\x)}) -- cycle;
							\path[cap added]
							plot[domain=\lo:\hi,samples=75]
							({(1+cappocket(\x))*cos(\x)},
							{(1+cappocket(\x))*sin(\x)})
							-- plot[domain=\hi:\lo,samples=50]
							({cos(\x)},{sin(\x)}) -- cycle;
						}
						\draw[cap edge]
						plot[domain=15:165,samples=150]
						({(1+.70*sin(1.2*(\x-15)))*cos(\x)},
						{(1+.70*sin(1.2*(\x-15)))*sin(\x)});
						\draw[cap edge]
						plot[domain=15:165,samples=240]
						({(1+cappocket(\x))*cos(\x)},
						{(1+cappocket(\x))*sin(\x)});
						\node at (0,1.43) {$E$};
						\node[cap label] at (0,-.20) {$\C$};
					\end{scope}
					\draw[cap leader] (5.65,3.95) -- (4.08,3.58);
					\node[anchor=west,cap label] at (5.65,3.95) {$H_1$};
					\draw[cap leader] (1.05,3.95) -- (2.62,3.58);
					\node[anchor=east,cap label] at (1.05,3.95) {$H_2$};
					\node at (3.35,.40) {$E^+=E\cup H_1\cup H_2$};
				\end{tikzpicture}\hfill
				\begin{tikzpicture}[cap panel]
					\CapPanel{b}
					\begin{scope}[shift={(3.65,2.55)},scale=1.65]
						\path[cap liquid] (0,0) ellipse (1.5 and 1);
						\begin{scope}
							\clip (0,-1.05) rectangle (1.55,1.05);
							\path[cap secondary] (0,0) ellipse (1.5 and 1);
						\end{scope}
						\fill[white] (0,0) circle (1);
						\fill[cap obstacle] (0,0) circle (1);
						\draw[cap edge] (0,0) ellipse (1.5 and 1);
						\draw[line width=.6pt] (0,0) circle (1);
						\draw[cap guide,line width=.8pt] (0,-1) -- (0,1);
						\fill (0,1) circle (1.3pt) (0,-1) circle (1.3pt);
						\node[above right] at (0,1) {$a$};
						\node[below right] at (0,-1) {$b$};
						\node[cap label] at (-.40,0) {$\C$};
						\node[cap label] at (.20,.35) {$\sigma$};
						\node at (-1.25,0) {$E_1$};
						\node at (1.25,0) {$E_2$};
					\end{scope}
					\draw[cap leader] (6.35,3.85) -- (5.68,3.50);
					\node[anchor=west,cap label] at (6.35,3.85) {$\mathscr J$};
				\end{tikzpicture}
			\end{minipage}%
		}
		\caption{Planar completion (Lemma~\ref{lem:planar-completion}).
			(a) Adding the amber exterior pockets gives $E^+=E\cup H_1\cup H_2$. (b) In the hypothetical configuration, a crosscut $\sigma$ through $\C$ separates $E$ into the blue and green positive-area pieces without adding perimeter. Indecomposability excludes this configuration. }
		\label{fig:planar-completion}
	\end{figure}
	
	\begin{proof}[Proof of Lemma~\ref{lem:planar-completion}]
		We use the saturation theory of Ambrosio, Caselles, Masnou, and Morel~\cite{ambrosio2001connected}: their union, saturation, hole-decomposition, and Jordan-representation results (Propositions 5(i), 6(ii), and 9; Theorem 7) give, respectively, indecomposability of unions with positive common boundary, preservation under saturation, the perimeter-additive hole decomposition, and the planar Jordan representation.
		
		Since $I_\tension(E;\C)>0$, the trace formula gives $\calH^1(\partial^*E\cap\partial^*\C)>0$, with opposite normals there. Thus $\mathcal S:=E\cup\C$ is indecomposable. Since $\C$ is convex it is indecomposable; the positive common reduced boundary permits application of the union result.
		
		Saturation fills all finite-measure complementary components. The saturation theorem makes the resulting set simple, and the planar Jordan-representation theorem for simple sets identifies it, modulo Lebesgue-null sets, with the bounded interior $G$ of a rectifiable Jordan curve $\mathscr J$. In this representation $\mathscr J$ and $\partial^*G$ agree modulo $\calH^1$, and the hole-decomposition identities give
		\[
		G=\mathcal S\cup\bigcup_{j=1}^{\infty}H_j,\qquad
		\P(\mathcal S)=\P(G)+\sum_{j=1}^{\infty}\P(H_j),\qquad
		\partial^*G\subset\partial^*\mathcal S
		\]
		modulo the appropriate null sets. Put $E^+=G\setminus\C$.
		
		\emph{Connectedness of the completed trace.} The Jordan boundary agrees with $\partial^*G$ modulo $\calH^1$. Hence $\calH^1(\mathscr J\cap\operatorname{int}\C)=0$. If $\mathscr J$ met $\operatorname{int}\C$, it would have a nondegenerate subarc in this open set, a contradiction. Since $\operatorname{int}\C$ is connected and lies almost everywhere in $G$, it follows that $\operatorname{int}\C\subset G$ and $\C\subset\overline G$.
		We claim that $\mathscr J\setminus\C$ has at most one component. The contradiction comes from cutting through the obstacle, where the competitor is absent, so that no new interface is created in $E$.
		Otherwise choose $a,b\in \mathscr J\cap\partial\C$ so that both joining subarcs of $\mathscr J$ contain positive-length exterior portions. Choose $c\in\operatorname{int}\C$ off the line through $a,b$, and put $\sigma=[a,c]\cup[c,b]$. For a convex body, every open segment joining a boundary point to an interior point lies in the interior. Thus $\sigma$ is simple and its relative interior lies in $\operatorname{int}\C\subset G$. It is a crosscut of $G$, separating it into Jordan subdomains $G_1,G_2$. Set $E_i=E\cap G_i$.
		
		Each exterior subarc contains a nondegenerate open subarc at positive distance from $\C$ after shortening. Since $\mathscr J=\partial^*G$ modulo $\calH^1$ and $\partial^*G\subset\partial^*(E\cup\C)$ modulo $\calH^1$, almost every point of that subarc belongs to $\partial^*E$. Choose such a point on each side, away from the endpoints. A sufficiently small ball about that point avoids $\C$, the crosscut, and the other boundary subarc; within $G$ it lies on the corresponding side of the crosscut. Since $E\subset G$ up to a null set, the density-$1/2$ property gives $|E_i|>0$ for both sides.
		
		To justify perimeter additivity, first restrict away from balls of radius $\rho$ about $a,b$. On each compact subarc of $\sigma\setminus\{a,b\}$ there is an open neighborhood contained in $\operatorname{int}\C$, where $\chi_E=0$ almost everywhere. Hence both BV traces on the polygonal cut vanish, and the restriction formula adds no jump across that cut. The portions of $\mathscr J$ introduce no additional interface between $E_1$ and $E_2$: they bound $G$, and $\chi_E=0$ almost everywhere outside $G$. The restriction formula, equivalently applied to the Jordan subdomains $G_i$, therefore partitions the original variation measure between the two sides.
		
		The endpoint balls can be removed along radii for which their artificial boundary contributions tend to zero, since $\chi_E$ is bounded and those circles have length $2\pi\rho$; moreover $|D\chi_E|$ has no atoms. Letting $\rho\to0^+$ proves
		\[
		E=E_1\cup E_2,\qquad \P(E)=\P(E_1)+\P(E_2),
		\]
		contradicting indecomposability.
		
		If $\mathscr J\cap\C=\emptyset$, then $E^+$ wets all of $\partial\C$. Otherwise $\mathscr J$ cannot be contained in $\partial\C$: two Jordan curves with this inclusion coincide, which would give $G=\operatorname{int}\C$ and contradict $|E|>0$. The claim now shows that $\mathscr J\cap\partial\C$ is a connected closed arc, possibly a point. Its complement in $\partial\C$ is connected. 
		At $\calH^1$-almost every point of $\mathscr J\cap\partial\C$, both boundaries are reduced boundaries. Since $\operatorname{int}\C\subset G$, their half-space blow-ups are nested and hence coincide; in particular, their outer normals agree. The exterior trace of $\chi_{G\setminus\C}$ is therefore zero on $\mathscr J\cap\partial\C$. On $\partial\C\cap G$ it equals one because $G$ is open. This proves the connected-trace assertion. See also~\cite[Lemma~5.7, Step~2]{mader2025quantitative} for the isotropic construction.
		
		\emph{Anisotropic energy accounting.} The perimeter-additive decomposition gives
		\[
		D\chi_{\mathcal S}=D\chi_G-\sum_{j=1}^{\infty}D\chi_{H_j},\quad\text{ and }\quad
		|D\chi_{\mathcal S}|=|D\chi_G|+\sum_{j=1}^{\infty}|D\chi_{H_j}|.
		\]
		The boundaries $\partial^*G$ and $\partial^*H_j$, $j\geq1$, are mutually disjoint modulo $\calH^1$; the normal of $\mathcal S$ is that of $G$ on $\partial^*G$ and is $-\nu_{H_j}$ on $\partial^*H_j$. Thus
		$\PerAni(\mathcal S)=\PerAni(G)+\sum_{j=1}^{\infty}\P_{\widecheck\tension}(H_j)$.
		The trace identities for $\mathcal S=E\cup\C$ and $G=E^+\cup\C$ read
		\[
		\begin{aligned}
			\PerAni(\mathcal S)&=\PerAni(E;U)+\int_{\partial\C}\tension(\nu_\C)(1-T_E)\,d\calH^1,\\
			\PerAni(G)&=\PerAni(E^+;U)+\int_{\partial\C}\tension(\nu_\C)(1-T_{E^+})\,d\calH^1.
		\end{aligned}
		\]
		Since $T_{E^+}-T_E=\sum_{j=1}^{\infty}T_{H_j}$, the obstacle-boundary terms cancel against those in the full perimeters of the holes. Consequently,
		\[
		\begin{aligned}
			\PerAni(E;U)-\PerAni(E^+;U)&=\sum_{j=1}^{\infty}\P_{\widecheck\tension}(H_j;U),\\
			I_\tension(E^+;\C)-I_\tension(E;\C)&=\sum_{j=1}^{\infty}I_\tension(H_j;\C).
		\end{aligned}
		\]
		These identities also hold for the countable family: $G$ is bounded and $\sum_{j=1}^{\infty}(|H_j|+\P(H_j))<\infty$, so the partial sums of the hole indicators converge in the $BV$-norm, and their traces converge on bounded Lipschitz neighborhoods of $\partial\C$. By Lemma~\ref{rev:lem:coercivity},
		$0\leq\sum_{j=1}^{\infty} I_\tension(H_j;\C) \leq\sum_{j=1}^{\infty}\P_{\widecheck\tension}(H_j;U)<\infty$.
		Subtracting the energies gives the sum of the pocket identities~\eqref{eq:pocket-energy-identity}. For $-1<\lambda\leq0$ each summand is at least $(1+\lambda)\P_{\widecheck\tension}(H_j;U)$; for $\lambda\geq0$ it is at least $\P_{\widecheck\tension}(H_j;U)$. This proves~\eqref{eq:completion-energy-gain} for every $\lambda>-1$.
	\end{proof}
	
	Apply the connected-contact comparison to each completed indecomposable component and sum the bounds. Truncation far from each component removes the boundedness assumption on the obstacle without changing the relevant boundary terms.
	\begin{theorem}[Unrestricted planar comparison]\label{thm:unrestricted-planar}
		Let \(d=2\), let \(\C\subsetneq\R^2\) be closed and convex with nonempty interior, and let \(\tension=h_W\), where \(W\) is a convex body with \(0\in\operatorname{int}W\). For every \(-1<\lambda\leq 0\) and every finite-perimeter set \(E\subset\R^2\setminus\C\) with \(0<|E|<\infty\),
		\begin{equation}\label{eq:unrestricted-planar}
			\EnergyC(E)\geq 2\sqrt{m_{\tension,\lambda}^{\min}|E|}.
		\end{equation}
		Equality holds if and only if, modulo a null set,
		$E=x_0+tK_{\tension,\lambda}(n)$
		for some \(t>0\), where
		$m_{\tension,\lambda}(n)=m_{\tension,\lambda}^{\min}$,
		and the entire truncating facet rests on a supporting face of \(\C\) with outer normal \(n\), as in Theorem~\ref{rev:thm:filling}.
	\end{theorem}
	
	\begin{proof}[Proof of Theorem~\ref{thm:unrestricted-planar}]
		First let $\C$ be bounded and $E$ indecomposable. If its trace vanishes, the full-space Wulff inequality gives
		\[
		\EnergyC(E)=\PerAni(E)\geq2\sqrt{|W||E|}
		>2\sqrt{m_{\tension,\lambda}^{\min}|E|},
		\]
		since the cutting planes meet $\operatorname{int}W$. Otherwise, Lemma~\ref{lem:planar-completion} and Corollary~\ref{cor:connected-planar} yield
		\[
		\EnergyC(E)\geq\EnergyC(E^+)
		\geq2\sqrt{m_{\tension,\lambda}^{\min}|E^+|}
		\geq2\sqrt{m_{\tension,\lambda}^{\min}|E|}.
		\]
		For a general competitor, the perimeter-additive indecomposable decomposition $E=\bigcup_{j=1}^{\infty}E_j$ satisfies
		\[
		D\chi_E=\sum_{j=1}^{\infty}D\chi_{E_j},\quad\text{ and }\quad
		|D\chi_E|=\sum_{j=1}^{\infty}|D\chi_{E_j}|,
		\]
		by~\cite{ambrosio2001connected}. Equality of variations identifies the normals on the corresponding reduced boundaries, giving additivity of the anisotropic free perimeter. The decomposition converges in the $BV$ norm; restricting the derivative to $\partial\C$ in the zero-extension formula gives $T_E=\sum_{j=1}^{\infty}T_{E_j}$ almost everywhere. Thus the wetted integral and energy are additive, and
		\[
		\EnergyC(E)=\sum_{j=1}^{\infty}\EnergyC(E_j)
		\geq2\sqrt{m_{\tension,\lambda}^{\min}}
		\sum_{j=1}^{\infty}\sqrt{|E_j|}
		\geq2\sqrt{m_{\tension,\lambda}^{\min}|E|}.
		\]
		For unbounded $\C$, every planar indecomposable component of finite volume and perimeter is essentially bounded~\cite[Remark~2.9]{dayrens2022connected}. Choose $R_j$ so that $E_j$ is essentially contained in a compact subset of $B_{R_j}$ and $\C\cap\overline{B_{R_j}}$ has nonempty interior. Replacing $\C$ by this convex body leaves both surface terms of $E_j$ unchanged. The bounded-obstacle argument applies to each component and can again be summed.
		
		Equality forces a single positive-volume component. Its trace is nonzero by the strict zero-trace bound; equality in the completion chain gives $|E^+\setminus E|=0$. Corollary~\ref{cor:connected-planar} then identifies a Wulff cap in a minimizing direction, with its facet on a supporting face of the truncated obstacle. To recover support for the original obstacle, choose $R$ with the entire convex cap in $B_{R-1}$. If $\Pi_s(n)$ supports $\C\cap\overline{B_R}$ along its facet, take $x$ in the relative interior of that facet. For any $y\in\C$ with $y\cdot n>s$, the points $(1-t)x+ty$ belong to $\C\cap B_R$ and lie above $\Pi_s(n)$ for small $t>0$, a contradiction. Hence the same plane supports $\C$. Conversely, every stated configuration gives equality by Theorem~\ref{rev:thm:filling}.
	\end{proof}
	
	Added area raises the comparison bound, while removed interfaces consume energy excess. Retaining these contributions controls pocket area and perimeter; the zero-adhesion comparison for the pockets gives the stronger quadratic area bound below.
	\begin{corollary}[Quantitative control of exterior pockets]\label{cor:quantitative-pockets}
		In the setting of Lemma~\ref{lem:planar-completion}, let \(-1<\lambda\le0\), \(v=|E|\), and let $\Deficit=\Deficit(E)$ be the normalized capillary energy deficit in~\eqref{eq:normalized-capillary-deficit}. Then
		\[
		\frac{|E^+\setminus E|}{v}
		\le 2\Deficit+\Deficit^2
		\]
		and
		\[
		(1+\lambda)\sum^{\infty}_{j=1}\P_{\widecheck\tension}(H_j;U)
		\leq 2\sqrt{m_{\tension,\lambda}^{\min}v\,}\,\Deficit.
		\]
		Moreover, the added volume satisfies the quadratic estimate
		\[
		\frac{|E^+\setminus E|}{|E|}
		\leq \frac{m_{\tension,\lambda}^{\min}}{(1+\lambda)^2m_{\tension,0}^{\min}}\,\Deficit^2.
		\]
	\end{corollary}
	
	\begin{proof}[Proof of Corollary~\ref{cor:quantitative-pockets}]
		The completion estimate and Corollary~\ref{cor:connected-planar} give
		\[
		1+\Deficit
		\geq \sqrt{\frac{|E^+|}{v}}
		+\frac{(1+\lambda)\sum^{\infty}_{j=1}\P_{\widecheck\tension}(H_j;U)}
		{2\sqrt{m_{\tension,\lambda}^{\min}v}}.
		\]
		Dropping the second term and squaring proves the volume estimate. Using \(|E^+|\ge v\) instead proves the perimeter estimate.
		
		To obtain the quadratic estimate, apply Theorem~\ref{thm:unrestricted-planar} with zero adhesion and the reflected anisotropy $\widecheck\tension(\nu):=\tension (-\nu)$ to each exterior pocket $H_j$. Since
		$m_{\widecheck\tension,0}^{\min} = m_{\tension,0}^{\min}$,
		we have
		\[
		\begin{aligned}
			\sum_{j=1}^{\infty}\operatorname{Per}_{\widecheck\tension}(H_j;U)
			&\geq 2\sqrt{m_{\tension,0}^{\min}} \sum_{j=1}^{\infty}\sqrt{|H_j|}\\
			&\geq 2\sqrt{m_{\tension,0}^{\min}|E^+\setminus E|}.
		\end{aligned}
		\]
		Combining this bound with the perimeter estimate above yields
		\[
		\frac{|E^+\setminus E|}{|E|} \leq \frac{m_{\tension,\lambda}^{\min}}{(1+\lambda)^2m_{\tension,0}^{\min}}\,\Deficit^2,
		\]
		as claimed.
	\end{proof}
	
	Near equality, three losses must be small: the cost of splitting area between components, the cap-volume excess of a nonminimizing direction, and the Wulff deficit of the filled principal component. The next theorem quantifies these effects.
	
	{\emergencystretch=2em All three estimates use a universal constant, independent of the obstacle, the anisotropy, and the admissible nonpositive adhesion parameter.\par}
	\begin{theorem}[Unrestricted quantitative planar comparison]\label{thm:unrestricted-planar-stability}
		Let $d=2$, let $\C\subsetneq\R^2$ be closed and convex with nonempty interior, and let $\tension=h_{\wulff}$, where $\wulff$ is an arbitrary convex body containing the origin in its interior. Assume that $-1<\lambda\leq 0$ and set
		\[
		m_*:=m_{\tension,\lambda}^{\min}.
		\]
		For every finite-perimeter set $E\subset\R^2\setminus\C$ with
		$0<|E|<\infty$, recall the deficit~\eqref{eq:normalized-capillary-deficit}, which reads
		\[
		\Deficit(E)=
		\frac{\EnergyC(E)}{2\sqrt{m_*|E|}}-1.
		\]
		There exists a universal constant \(C_{\mathrm{stab}}>0\) such that, for every such \(E\), one can find \(n\in\mathbb{S}^1\) and \(x_0\in\R^2\) satisfying
		\[
		0\leq
		\frac{m_{\tension,\lambda}(n)}{m_*}-1
		\leq C_{\mathrm{stab}}\,\Deficit(E)
		\]
		and
		\[
		\frac{\left|E\triangle
			\left(x_0+\sqrt{(|E|/m_{\tension,\lambda}(n))}
			\,K_{\tension,\lambda}(n)\right)\right|}{|E|}
		\leq C_{\mathrm{stab}}\sqrt{\Deficit(E)}.
		\]
		Moreover, if $E_1$ denotes a largest indecomposable component of $E$, then
		\[
		\frac{|E\setminus E_1|}{|E|}
		\leq C_{\mathrm{stab}}\,\Deficit(E)^2.
		\]
		No connectedness assumption on the wetted trace, or smoothness, symmetry, or strict convexity assumption on $\wulff$, is required.
	\end{theorem}
	
	For comparison with the fixed-obstacle problem, define
	\[
	\mathcal I_{\tension,\lambda,\C}(v)
	:=\inf\{\EnergyC(A):\,A\subset\R^2\setminus\C,\ |A|=v,\ \P(A)<\infty\},
	\quad B(v):=2\sqrt{m_*v}.
	\]
	For $v=|E|$, the deficit in Theorem~\ref{thm:unrestricted-planar-stability} decomposes as
	\begin{equation}\label{eq:universal-profile-deficit}
		\Deficit(E)=
		\frac{\EnergyC(E)-\mathcal I_{\tension,\lambda,\C}(v)}{B(v)}
		+\frac{\mathcal I_{\tension,\lambda,\C}(v)-B(v)}{B(v)}.
	\end{equation}
	The terms measure the excess above the fixed-obstacle profile and that profile's gap above the universal bound. A positive gap prevents the universal deficit from vanishing even along a minimizing sequence. The comparison cap need not be admissible for $\C$. The dependence of exterior profiles on obstacle geometry is already substantial in the isotropic zero-adhesion setting; see~\cite{fusco2024rigidity} for rigidity and large-volume asymptotics. Figure~\ref{fig:quantitative-cap-stability} illustrates concentration and comparison in measure.
	
	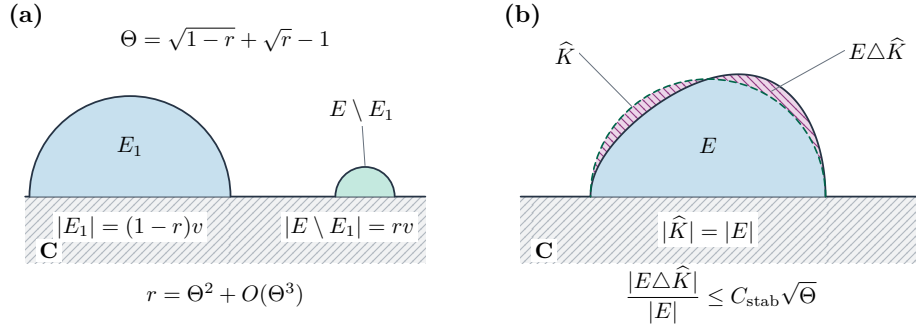
\begin{figure}[!htbp]
		\centering
		\scalebox{\CapFigureScale}{%
			\begin{minipage}{\textwidth}
				\centering
				\begin{tikzpicture}[cap panel]
					\CapPanel{a}
					\fill[cap obstacle] (.45,.85) rectangle (6.75,1.90);
					\draw[cap edge] (.45,1.90) -- (6.75,1.90);
					\begin{scope}[shift={(2.10,1.90)}]
						\path[cap liquid] ({1.65*sqrt(.92)},0)
						arc (0:180:{1.65*sqrt(.92)}) -- cycle;
						\draw[cap edge] ({1.65*sqrt(.92)},0)
						arc (0:180:{1.65*sqrt(.92)});
					\end{scope}
					\begin{scope}[shift={(5.80,1.90)}]
						\path[cap secondary] ({1.65*sqrt(.08)},0)
						arc (0:180:{1.65*sqrt(.08)}) -- cycle;
						\draw[cap edge] ({1.65*sqrt(.08)},0)
						arc (0:180:{1.65*sqrt(.08)});
					\end{scope}
					\node at (2.10,2.69) {$E_1$};
					\draw[cap leader] (5.75,3.10) -- (5.80,2.39);
					\node[cap label] at (5.75,3.25) {$E\setminus E_1$};
					\node[cap label] at (2.10,1.42) {$|E_1|=(1-r)v$};
					\node[cap label] at (5.55,1.42) {$|E\setminus E_1|=rv$};
					\node[cap label] at (.82,1.05) {$\C$};
					\node at (3.60,4.38) {$\Deficit=\sqrt{1-r}+\sqrt r-1$};
					\node at (3.60,.35) {$r=\Deficit^2+O(\Deficit^3)$};
				\end{tikzpicture}\hfill
				\begin{tikzpicture}[cap panel]
					\CapPanel{b}
					\fill[cap obstacle] (.45,.85) rectangle (6.75,1.90);
					\draw[cap edge] (.45,1.90) -- (6.75,1.90);
					\begin{scope}[shift={(3.40,1.90)},scale=1.85]
						\path[cap liquid] plot[domain=0:180,samples=121]
						({cos(\x)},{sin(\x)*(1+.3*cos(\x))}) -- cycle;
						\path[cap difference] plot[domain=0:90,samples=61]
						({cos(\x)},{sin(\x)*(1+.3*cos(\x))}) --
						plot[domain=90:0,samples=61] ({cos(\x)},{sin(\x)}) -- cycle;
						\path[cap difference] plot[domain=90:180,samples=61]
						({cos(\x)},{sin(\x)*(1+.3*cos(\x))}) --
						plot[domain=180:90,samples=61] ({cos(\x)},{sin(\x)}) -- cycle;
						\draw[cap edge] plot[domain=0:180,samples=121]
						({cos(\x)},{sin(\x)*(1+.3*cos(\x))});
						\draw[cap reference] (1,0) arc (0:180:1);
					\end{scope}
					\node at (3.40,2.70) {$E$};
					\draw[cap leader] (1.37,4.16) -- (2.20,3.32);
					\node[cap label,anchor=east] at (1.37,4.16) {$\widehat K$};
					\draw[cap leader] (5.55,4.20) -- (4.40,3.56);
					\node[cap label,anchor=west] at (5.55,4.20) {$E\triangle\widehat K$};
					\node[cap label] at (3.40,1.38) {$|\widehat K|=|E|$};
					\node[cap label] at (.82,1.05) {$\C$};
					\node at (3.60,.35) {$\displaystyle
						\frac{|E\triangle\widehat K|}{|E|}\leq C_{\mathrm{stab}}\sqrt\Deficit$};
				\end{tikzpicture}
			\end{minipage}%
		}
		\caption{Concentration and comparison in measure
			(Theorem~\ref{thm:unrestricted-planar-stability}).
			(a) Two separated equality caps give the displayed exact deficit and the sharp quadratic concentration order of Corollary~\ref{cor:sharp-component-concentration}. (b) A blue competitor and a green dashed model cap of equal area; purple hatching marks their symmetric difference. The comparison cap may be translated freely.}
		\label{fig:quantitative-cap-stability}
	\end{figure}
	\vspace{-.5em}
	\begin{proof}[Proof of Theorem~\ref{thm:unrestricted-planar-stability}]
		
		We select a principal component, control its filling and Wulff deficit, then remove the added regions and restore the total area.
		Write
		\[
		v=|E|,\qquad m_*=m_{\tension,\lambda}^{\min},\qquad
		\Deficit=\frac{\EnergyC(E)}{2\sqrt{m_*v}}-1.
		\]
		By Theorem~\ref{thm:unrestricted-planar}, \(\Deficit\geq0\). Let \(E=\bigcup_{j=1}^{\infty}E_j\) be its perimeter-additive indecomposable decomposition, and put \(v_j=|E_j|\) and \(t_j=v_j/v\). Relabel so that \(t_1=\max_{j\in\N}t_j\); the maximum exists because \(\sum_{j=1}^{\infty}t_j=1\). Fix the universal number \(\Deficit_c=1/10\).
		
		If \(\Deficit\geq\Deficit_c\), choose a minimizing direction \(n_*\) and any translate of \(\sqrt{v/m_*}\,K_{\tension,\lambda}(n_*)\). Its symmetric difference with \(E\) has measure at most \(2v\), the directional excess is zero, and \(|E\setminus E_1|/v\le1\). All conclusions follow after taking \(C_{\mathrm{stab}}\geq\max\{2/\sqrt{\Deficit_c},1/\Deficit_c^2\}\). It remains to consider \(0\leq\Deficit<\Deficit_c\).
		
		\medskip
		\noindent\textit{Concentration and the principal deficit.}
		Energy additivity and Theorem~\ref{thm:unrestricted-planar} give
		$1+\Deficit\ge\sum_{j=1}^{\infty}\sqrt{t_j}$.
		If \(t_1\leq1/2\), the right-hand side is at least \(\sqrt2\), a contradiction. Thus \(t_1>1/2\). For \(r=1-t_1\),
		$\Deficit\geq\sqrt{1-r}+\sqrt r-1\geq\sqrt r-r$,
		and hence
		$r\leq\Deficit^2/(1-2^{-1/2})^2$.
		Set
		$\Deficit_j:=\EnergyC(E_j)/(2\sqrt{m_*v_j})-1\geq 0$.
		Then
		\[
		\sum_{j=1}^{\infty}\sqrt{t_j}\,\Deficit_j
		=1+\Deficit-\sum_{j=1}^{\infty}\sqrt{t_j}\le\Deficit,
		\qquad
		0\le\Deficit_1\le\sqrt2\,\Deficit<1.
		\]

		\noindent\textit{A principal component with positive trace.}
		Suppose first that \(I_\tension(E_1;\C)>0\). If \(\C\) is unbounded, use the essential boundedness of planar indecomposable finite-perimeter sets of finite volume to choose \(R\) with \(E_1\) contained, up to a null set, in a compact subset of \(B_R\), and with \(\C_R=\C\cap\overline{B_R}\) having nonempty interior. The free perimeter and wetted integral of \(E_1\) are unchanged by this replacement. Thus the following completion argument can be carried out with a bounded convex obstacle.
		
		Apply Lemma~\ref{lem:planar-completion} to \(E_1\), obtaining \(F=E_1^+\). The proof of Corollary~\ref{cor:connected-planar}
		provides a filling $D$ in a direction $n$. Put
		$K=K_{\tension,\lambda}(n)$, $m_n=|K|$, $G=F\cup D$.	Here $\C_R$ denotes the original obstacle when it is bounded. Set
		\[
		\eta_K(G):=\frac{P_{h_K}(G)}{2\sqrt{m_n|G|}}-1,\qquad
		\mathscr R:=\mathcal E_{\tension,\lambda,\C_R}(E_1)
		-\mathcal E_{\tension,\lambda,\C_R}(F).
		\]
		Writing $\Gamma_F$ for the wetted trace of $F$ on $\partial\C_R$, Proposition~\ref{rev:prop:identity} gives the exact identity
		\begin{equation}\label{eq:principal-deficit-accounting}
			\begin{aligned}
				1+\Deficit_1={}
				\sqrt{\frac{m_n}{m_*}\frac{|G|}{v_1}}\bigl(1+\eta_K(G)\bigr)+\frac{(\mathscr R+R_K(F)-\lambda\,\Delta_\tension(\Gamma_F))}{2\sqrt{m_*v_1}}.
			\end{aligned}
		\end{equation}
		Every factor is at least one and every additive remainder is nonnegative. Thus a small principal deficit controls direction, added volume, Wulff shape error, and boundary losses separately.
		
		All contact terms refer to $\C_R$. The Wulff inequality gives $\eta_K(G)\geq0$, and the other remainders are nonnegative. Moreover, Lemma~\ref{lem:planar-completion} gives
		$\mathscr R\geq(1+\lambda)\sum_{j=1}^\infty\P_{\widecheck\tension}(H_j;\R^2\setminus\C_R)$,
		where $H_j$ are the completion holes. Since $m_n\geq m_*$ and $|G|\geq v_1$, dropping the remainders in~\eqref{eq:principal-deficit-accounting} yields
		\[
		1\leq\frac{m_n}{m_*}\leq(1+\Deficit_1)^2,\qquad
		1\leq\frac{|G|}{v_1}\leq(1+\Deficit_1)^2.
		\]
		Equation~\eqref{eq:principal-deficit-accounting} also gives $\eta_K(G)\leq\Deficit_1$. The quantitative Wulff inequality of Figalli--Maggi--Pratelli~\cite[Theorem~1.1]{figalli2010mass}, translating \(K\) by an interior point if necessary, gives a point \(x_0\) such that
		\[
		\left|G\triangle\left(x_0+\sqrt{|G|/m_n}\,K\right)\right|
		\leq C|G|\sqrt{\Deficit_1},
		\]
		with a dimensional constant $C$ independent of $K$. Because \(0\in K\), its positive dilations are nested. Moreover,
		$|G\setminus E_1|\le(2\Deficit_1+\Deficit_1^2)v_1$. This linear bound is uniform in $\lambda$, unlike the quadratic hole bound in Corollary~\ref{cor:quantitative-pockets}, whose constant may diverge as $\lambda\to(-1)^+$.
		Using $E_1\subset G$ and the nestedness of the comparison caps, the triangle inequality gives
		\[
		\left|E_1\triangle
		\left(x_0+\sqrt{v_1/m_n}\,K\right)\right|
		\leq
		\left|G\triangle
		\left(x_0+\sqrt{|G|/m_n}\,K\right)\right|+2(|G|-v_1).
		\]
		Indeed, removing $G\setminus E_1$ costs $|G|-v_1$ in symmetric difference, and shrinking the comparison cap from volume $|G|$ to volume $v_1$ costs the same amount, since $0\in K$ makes these dilations nested. Combining this estimate with the preceding bounds, $|G|\le(1+\Theta_1)^2v_1$, and $\Theta_1<1$, we obtain
		\[
		\left|E_1\triangle\left(x_0+\sqrt{v_1/m_n}\,K\right)\right|
		\leq Cv_1\sqrt{\Deficit_1},\qquad
		\frac{m_n}{m_*}-1\le3\Deficit_1.
		\]
		
		\noindent\textit{Uniformity of the constants.}
		The stability constant just used is dimensional in the quantitative Wulff inequality, with the Wulff deficit normalized by $2\sqrt{m_n|G|}$. The filling and completion estimates enter only through nonnegative terms in~\eqref{eq:principal-deficit-accounting}; the proof never divides by $1+\lambda$, a curvature bound for $\C$, or an ellipticity constant for $\tension$. Rescaling changes the comparison set by its volume difference, since $0\in K$ makes the dilations nested. These observations explain uniformity even for crystalline or highly eccentric Wulff bodies and as $\lambda\to(-1)^+$.
		
		\medskip
		\noindent\textit{A principal component with zero trace.}
		The full-space Wulff inequality allows a dry component to have small capillary deficit only when the minimum cap volume is close to $|W|$. We therefore compare first with $W$ and then with a cap of equal area.
		Suppose now that \(I_\tension(E_1;\C)=0\), and put \(w=|\wulff|\). The full-space Wulff deficit is
		\[
		\eta=\frac{\PerAni(E_1;\R^2)}{2\sqrt{wv_1}}-1
		=(1+\Deficit_1)\sqrt{\frac{m_*}{w}}-1.
		\]
		The Wulff inequality implies
		\[
		0\leq\eta\le\Deficit_1,\qquad
		\frac{m_*}{w}\geq\frac{1}{(1+\Deficit_1)^2}.
		\]
		Quantitative Wulff stability gives a translate \(x_0+a\wulff\), where \(a=\sqrt{v_1/w}\), such that
		\[
		|E_1\triangle(x_0+a\wulff)|
		\leq Cv_1\sqrt{\Deficit_1}.
		\]
		Choose a minimizing direction \(n_*\), let \(K_*=K_{\tension,\lambda}(n_*)\), and set \(b=\sqrt{v_1/m_*}\). Since \(a\le b\), \(K_*\subset\wulff\), and \(0\in K_*\), the set \(aK_*\) is contained in both \(a\wulff\) and \(bK_*\). Consequently,
		\[
		\begin{aligned}
			|a\wulff\,\triangle\, bK_*|
			&\leq 2v_1(1-m_*/w)\\
			&\leq 2v_1\bigl(1-1/(1+\Deficit_1)^2\bigr)
			\leq 4v_1\Deficit_1.
		\end{aligned}
		\]
		The same componentwise shape estimate follows, now with \(n=n_*\), \(m_n=m_*\), and zero directional excess.\smallskip
		
		\noindent\textit{Recombination.}
		In either case we have a cap of volume \(v_1\) with componentwise symmetric difference at most \(Cv_1\sqrt{\Deficit_1}\). Enlarge this cap concentrically to volume \(v\). Nestedness and the triangle inequality give
		\[
		\begin{aligned}
			\frac{\left|E\triangle\left(x_0+\sqrt{v/m_n}\,K_{\tension,\lambda}(n)\right)\right|}{v}
			&\leq 2r+C\frac{v_1}{v}\sqrt{\Deficit_1}\\
			&\leq C\sqrt\Deficit.
		\end{aligned}
		\]
		Also \(m_n/m_*-1\le C\Deficit\), and \(|E\setminus E_1|/v=r\le C\Deficit^2\). All constants used here are dimensional or numerical. Increasing \(C_{\mathrm{stab}}\) to include the large-deficit case proves the theorem.
	\end{proof}
	
	Grouping the area outside the largest component into one component minimizes the sum of square roots. Two separated equality caps attain the resulting sharp bound.
	\begin{corollary}[Sharp component concentration]\label{cor:sharp-component-concentration}
		Under the hypotheses of Theorem~\ref{thm:unrestricted-planar-stability}, if $0\leq\Deficit(E)<\sqrt2-1$ and $E_1$ is a largest indecomposable component, then
		\begin{equation}\label{eq:sharp-component-concentration}
			\frac{|E\setminus E_1|}{|E|}
			\leq\frac{1-\sqrt{1-(2\Deficit(E)+\Deficit(E)^2)^2}}{2}.
		\end{equation}
		The right-hand side is $\Deficit(E)^2+O(\Deficit(E)^3)$ as $\Deficit(E)\to0$. For this component-concentration estimate, both the exponent and the leading constant are optimal.
	\end{corollary}
	\begin{proof}[Proof of Corollary~\ref{cor:sharp-component-concentration}]
		Energy additivity and Theorem~\ref{thm:unrestricted-planar} give $\sum_{j=1}^{\infty}\sqrt{|E_j|/|E|}\leq1+\Deficit(E)<\sqrt2$, hence $|E_1|>|E|/2$. Writing $r=|E\setminus E_1|/|E|<1/2$, we obtain $\sqrt{1-r}+\sqrt r\leq1+\Deficit(E)$. Squaring and solving for $r$ proves~\eqref{eq:sharp-component-concentration}.
		For sharpness, fix a minimizing direction and a half-plane obstacle with that outer normal. Place two disjoint equality caps on its boundary, separated tangentially, with areas $(1-r)v$ and $rv$, where $0<r<1/2$. Their deficit is exactly $\sqrt{1-r}+\sqrt r-1$, so equality holds in~\eqref{eq:sharp-component-concentration}.
	\end{proof}
	
	Small cap-volume excess need not imply angular proximity. The quadratic growth assumption below gives this implication; moving the cutting plane then changes the cap only in a thin strip, allowing an exactly minimizing direction.
	
	\begin{corollary}[Quantitative selection of
		minimizing directions]\label{cor:orientation-stability}
		Assume the hypotheses of Theorem~\ref{thm:unrestricted-planar-stability}. Let
		\[
		\mathcal N_{\tension,\lambda}
		:= \left\{n\in\mathbb{S}^1:\, m_{\tension,\lambda}(n) =
		m_{\tension,\lambda}^{\min}\right\}.
		\]
		Suppose that, for some $c_0>0$,
		\[
		m_{\tension,\lambda}(n) - m_{\tension,\lambda}^{\min}
		\geq
		c_0\operatorname{dist}(n,\mathcal N_{\tension,\lambda})^2
		\]
		for every $n\in\mathbb{S}^1$.
		
		Then there exist $\Deficit_0>0$ and $A>0$, depending on $\tension,\lambda,c_0$, such that, for  $\Deficit(E)\leq\Deficit_0$, one can find $n_*\in\mathcal N_{\tension,\lambda}$ and $x_0\in\R^2$ satisfying
		\[
		\frac{\left|E\triangle
			\left(x_0+\sqrt{(|E|/m_{\tension,\lambda}^{\min})}
			K_{\tension,\lambda}(n_*)\right)\right|}{|E|}
		\leq A\sqrt{\Deficit(E)}.
		\]
	\end{corollary}
	
	\begin{proof}[Proof of Corollary~\ref{cor:orientation-stability}]
		By Theorem~\ref{thm:unrestricted-planar-stability},  $\exists n\in\mathbb{S}^1$ so that
		$m_{\tension,\lambda}(n) - m_{\tension,\lambda}^{\min} \leq C\Deficit(E)$.
		The nondegeneracy assumption gives
		$\operatorname{dist}(n,\mathcal N_{\tension,\lambda}) \leq C\sqrt{\Deficit(E)}$.
		Choose $n_*\in\mathcal N_{\tension,\lambda}$ realizing this distance.
		
		Since $\wulff$ is bounded and $\tension$ is Lipschitz on the sphere, the defining cutting half-spaces of $K_{\tension,\lambda}(n)$ and $K_{\tension,\lambda}(n_*)$ differ inside $\wulff$ only in a strip whose width is bounded by $C|n-n_*|$. Consequently,
		$|K_{\tension,\lambda}(n)\,\triangle\, K_{\tension,\lambda}(n_*)| \leq C|n-n_*|$.
		The difference between the volume-normalizing dilation factors is controlled by
		$m_{\tension,\lambda}(n) - m_{\tension,\lambda}^{\min} \leq C\Deficit(E)$.
		
		Combining these estimates with the quantitative comparison of the preceding theorem and the triangle inequality for symmetric difference proves the assertion.
	\end{proof}
	
	\paragraph{A non-even example of direction selection.} Take $W=b+B_1\subset\R^2$, where $0<|b|<1$, so $\tension(q)=|q|+b\cdot q$. For $-1<\lambda\leq0$, writing $p=b+y$ gives
	\[
	m_{\tension,\lambda}(n)=\mathcal V(\tau(n)),\qquad
	\tau(n)=\lambda-(1+\lambda)b\cdot n,\qquad
	\mathcal V(t)=\arccos t-t\sqrt{1-t^2}.
	\]
	The cutting height $\tau(n)$ is largest when the cap area is smallest. Thus the translation $b$ selects a preferred direction through the wall law, although $W$ is a round disk.
	
	The range $I=[\lambda-(1+\lambda)|b|,\lambda+(1+\lambda)|b|]$ lies in $(-1,1)$. Since $\mathcal V'(t)=-2\sqrt{1-t^2}$, the minimizing direction is uniquely $n_*=-b/|b|$. Moreover,
	$\tau(n_*)-\tau(n)=(1+\lambda)|b||n-n_*|^2/2$.
	The mean value theorem therefore verifies the hypothesis of Corollary~\ref{cor:orientation-stability} with
	$c_0=(1+\lambda)|b|\min_{t\in I}\sqrt{1-t^2}>0$.
	This illustrates the dependence of direction selection on the anisotropy and adhesion, consistently with~\eqref{eq:wall-gauge-dependence}, despite the universal constant in Theorem~\ref{thm:unrestricted-planar-stability}.
	
	\subsection{Two parallel wetted caps at zero adhesion}
	When the wetted trace consists of two opposite caps cut by parallel hyperplanes, an intermediate plane separates them regardless of the exterior liquid connection. At zero adhesion its new faces have zero cost for the auxiliary half-Wulff support functions, so the sides can be filled and compared separately. This vanishing concerns the auxiliary perimeters, not the original tension.
	
	\begin{proposition}[Two parallel wetted caps]\label{prop:cc-parallel-caps-zero}
		Let $d\geq 2$, let $W$ be a convex body with $0\in\operatorname{int}W$, set $\tension=h_W$, and let $\C\subset\mathbb R^d$ be a convex body.  Fix $e\in\mathbb{S}^{d-1}$ and $s_-<s_+$ such that both hyperplanes $\Pi_{s_-}(e)$ and $\Pi_{s_+}(e)$ meet $\operatorname{int}\C$. Suppose that
		$E\subset\R^d\setminus\C$ has finite perimeter, $0<|E|<\infty$, and exterior trace satisfying $\calH^{d-1}$-almost everywhere on $\partial\C$
		\[
		T_E=\chi_{\Gamma_+\cup\Gamma_-},
		\qquad
		\Gamma_+=\partial\C\cap\H_{s_+}^+(e),
		\qquad
		\Gamma_-=\partial\C\cap\H_{s_-}^-(e).
		\]
		Set $n_+=e$, $n_-=-e$, $D_+=\C\cap\H_{s_+}^+(e)$, $D_-=\C\cap \H_{s_-}^-(e)$, $b_\pm=|D_\pm|$, $v_D=b_++b_-$, and $\rho_0=d(m_{\tension,0}^{\min})^{1/d}$. Then
		\begin{equation}\label{eq:cc-parallel-caps-volume}    
			\PerAni(E;\R^d\setminus\C)
			\geq\rho_0(|E|+v_D)^{(d-1)/d}>\rho_0|E|^{(d-1)/d}.
		\end{equation}
		No restriction is imposed on the exterior connection between the wetted caps. If $\tension$ is even, then $\rho_0=d(|W|/2)^{1/d}$.
	\end{proposition}
	
	See Figure~\ref{fig:parallel-wetted-caps} for the splitting and filling construction associated with the two parallel wetted caps.
	
	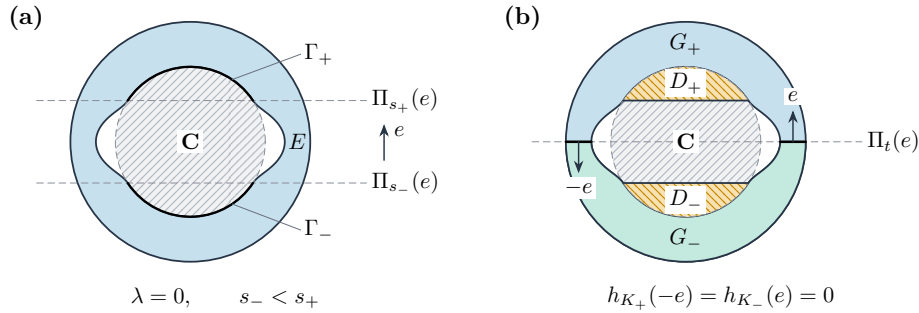
\begin{figure}[!htbp]
		\centering
		\scalebox{\CapFigureScale}{%
			\begin{minipage}{\textwidth}
				\centering
				\begin{tikzpicture}[cap panel]
					\CapPanel{a}
					\begin{scope}[shift={(3.05,2.80)},scale=1.18]
						\CapRing
						\draw[cap guide] (-2.05,.55) -- (2.35,.55);
						\draw[cap guide] (-2.05,-.55) -- (2.35,-.55);
						\draw[line width=1.05pt] ({sqrt(1-.55^2)},.55)
						arc (33.367013:146.632987:1);
						\draw[line width=1.05pt] ({-sqrt(1-.55^2)},-.55)
						arc (213.367013:326.632987:1);
						\node[cap label] at (0,0) {$\C$};
						\node at (1.43,0) {$E$};
					\end{scope}
					\node[cap label,anchor=west] at (5.85,3.45) {$\Pi_{s_+}(e)$};
					\node[cap label,anchor=west] at (5.85,2.15) {$\Pi_{s_-}(e)$};
					\draw[cap leader] (4.79,4.18) -- (3.71,3.78);
					\node[cap label,anchor=west] at (4.79,4.18) {$\Gamma_+$};
					\draw[cap leader] (4.79,1.41) -- (3.71,1.82);
					\node[cap label,anchor=west] at (4.79,1.41) {$\Gamma_-$};
					\draw[cap normal] (6.10,2.51) -- (6.10,3.08);
					\node at (6.34,2.96) {$e$};
					\node at (3.60,.35) {$\lambda=0,\qquad s_-<s_+$};
				\end{tikzpicture}\hfill
				\begin{tikzpicture}[cap panel]
					\CapPanel{b}
					\begin{scope}[shift={(3.05,2.80)},scale=1.18]
						\CapRing
						\begin{scope}
							\clip (-1.61,-1.61) rectangle (1.61,0);
							\path[cap secondary,even odd rule] (0,0) circle (1.6)
							plot[domain=0:360,samples=241]
							({capinner(\x)*cos(\x)},{capinner(\x)*sin(\x)}) -- cycle;
							\draw[cap edge] (0,0) circle (1.6);
							\draw[cap edge] plot[domain=0:360,samples=241]
							({capinner(\x)*cos(\x)},{capinner(\x)*sin(\x)}) -- cycle;
						\end{scope}
						\path[fill=white] ({sqrt(1-.55^2)},.55)
						arc (33.367013:146.632987:1) -- cycle;
						\path[cap added] ({sqrt(1-.55^2)},.55)
						arc (33.367013:146.632987:1) -- cycle;
						\path[fill=white] ({-sqrt(1-.55^2)},-.55)
						arc (213.367013:326.632987:1) -- cycle;
						\path[cap added] ({-sqrt(1-.55^2)},-.55)
						arc (213.367013:326.632987:1) -- cycle;
						\draw[cap guide] ({sqrt(1-.55^2)},.55)
						arc (33.367013:146.632987:1);
						\draw[cap guide] ({-sqrt(1-.55^2)},-.55)
						arc (213.367013:326.632987:1);
						\draw[cap edge] ({-sqrt(1-.55^2)},.55) -- ({sqrt(1-.55^2)},.55);
						\draw[cap edge] ({-sqrt(1-.55^2)},-.55) -- ({sqrt(1-.55^2)},-.55);
						\draw[cap guide] (-2.05,0) -- (2.35,0);
						\draw[line width=1.15pt] (-1.6,0) -- (-1.26,0);
						\draw[line width=1.15pt] (1.26,0) -- (1.6,0);
						\draw[cap normal] (-1.43,0) -- (-1.43,-.42);
						\draw[cap normal] (1.43,0) -- (1.43,.42);
						\node[cap label] at (-1.43,-.61) {$-e$};
						\node[cap label] at (1.43,.61) {$e$};
						\node[cap label] at (0,.78) {$D_+$};
						\node[cap label] at (0,-.78) {$D_-$};
						\node[cap label] at (0,0) {$\C$};
						\node at (0,1.30) {$G_+$};
						\node at (0,-1.30) {$G_-$};
					\end{scope}
					\node[cap label,anchor=west] at (5.85,2.80) {$\Pi_t(e)$};
					\node at (3.60,.35) {$h_{K_+}(-e)=h_{K_-}(e)=0$};
				\end{tikzpicture}
			\end{minipage}%
		}
		\caption{Two parallel wetted caps at $\lambda=0$
			(Proposition~\ref{prop:cc-parallel-caps-zero}).
			(a) The exterior connection is unrestricted. (b) Splitting at $\Pi_t(e)$ and adjoining the amber $D_\pm$ gives $G_\pm=E_t^\pm\cup D_\pm$. The artificial cutting faces have zero cost for $h_{K_\pm}$; the positive added volume makes the comparison strict.}
		\label{fig:parallel-wetted-caps}
	\end{figure}
	
	\begin{proof}[Proof of Proposition~\ref{prop:cc-parallel-caps-zero}]
		Write $U=\R^d\setminus\C$, $K_\pm=K_{\tension,0}(n_\pm)$ and $\rho_\pm=d|K_\pm|^{1/d}$. Since $n_+=e$ and $n_-=-e$,
		\[
		h_{K_+}(-e)=h_{K_-}(e)=0,\qquad h_{K_\pm}\leq\tension.
		\]
		For almost every $t\in(s_-,s_+)$, set $E_t^\pm:=E\cap\H_t^\pm(e)$. These restrictions have finite perimeter,
		$|D\chi_E|(\Pi_t(e))=0$, and the two traces of $\chi_E$ on $\Pi_t(e)$ agree almost everywhere. The BV product rule gives artificial boundaries with normals $-e$ and $e$, respectively. Their contributions to the corresponding support
		perimeters vanish, and therefore
		\[
		P_{h_{K_+}}(E_t^+;U)+P_{h_{K_-}}(E_t^-;U)
		\leq\PerAni(E;U).
		\]
		The remaining integrals are precisely those of $h_{K_+}(\nu_E)$ above the cut and $h_{K_-}(\nu_E)$ below it.
		
		The exterior traces of $E_t^\pm$ on $\partial\C$
		are $\chi_{\Gamma_\pm}$. Indeed, restriction multiplies the obstacle trace by $\chi_{\H_t^\pm(e)}$ away from $\Pi_t(e)$, and $\Pi_t(e)\cap\partial\C$ has zero $\calH^{d-1}$ measure because it is the
		relative boundary of a convex section through
		$\operatorname{int}\C$. Each restriction has positive volume, since its exterior trace is nonzero on a set of positive surface measure. The sets $G_\pm=E_t^\pm\cup D_\pm$ have finite perimeter and finite positive volume. By Lemma~\ref{rev:lem:filling} and the vanishing support values,
		\[
		P_{h_{K_\pm}}(G_\pm)=P_{h_{K_\pm}}(E_t^\pm;U),
		\qquad |G_\pm|=|E_t^\pm|+b_\pm.
		\]
		Applying the full-space Wulff inequality~\eqref{v12:eq:hd-wulff} gives
		\begin{equation}\label{eq:cc-parallel-caps-slice}
			\PerAni(E;U)
			\geq\rho_+(|E_t^+|+b_+)^{(d-1)/d}+\rho_-(|E_t^-|+b_-)^{(d-1)/d}.
		\end{equation}
		Translation of $K_\pm$ by an interior point justifies this application, exactly as in the proof of Proposition~\ref{rev:prop:identity}. Now $\rho_\pm\ge\rho_0$ and
		$x^{(d-1)/d}+y^{(d-1)/d}\ge(x+y)^{(d-1)/d}$ for $x,y\geq0$, proving the first inequality in~\eqref{eq:cc-parallel-caps-volume}. The second is strict because $v_D>0$. If $W=-W$, both half-space sections through the origin have volume $|W|/2$.
	\end{proof}
	
	Splitting comparable volumes creates a definite excess through the power $(d-1)/d$. Small deficit therefore forces almost all volume to one side of every intermediate cut. Continuity as the plane moves prevents that side from switching, yielding concentration in one outer lobe.
	\begin{corollary}[Concentration in one outer lobe]\label{cor:cc-parallel-caps-concentration}
		Under the hypotheses of Proposition~\ref{prop:cc-parallel-caps-zero}, put $v=|E|$, $U=\R^d\setminus\C$, and
		\[
		\Deficit_d(E)=\frac{\PerAni(E;U)}{\rho_0\,v^{(d-1)/d}}-1,
		\qquad \Theta_d:=\Theta_d(E).
		\]
		Define $a_+:=|E\cap \H_{s_+}^+(e)|$, $a_-:=|E\cap \H_{s_-}^-(e)|$, and $V:=|E\cap\{x:s_-<x\cdot e<s_+\}|$. If $\Deficit_d<2^{1/d}-1$, then
		\begin{equation}\label{eq:cc-parallel-caps-concentration}
			\frac{1}{v}(V+\min\{a_+,a_-\})
			\leq\left(\frac{\Deficit_d}{2-2^{(d-1)/d}}\right)^{d/(d-1)}.    
		\end{equation}
	\end{corollary}
	
	\begin{proof}[Proof of Corollary~\ref{cor:cc-parallel-caps-concentration}]
		Put $u(t)=|E_t^+|/v$ and 
		\[
		F(x,y)=x^{(d-1)/d}+y^{(d-1)/d}-(x+y)^{(d-1)/d}.
		\]
		By~\eqref{eq:cc-parallel-caps-slice}, $\Deficit_d\geq F(u(t),1-u(t))$ for almost every intermediate $t$. Monotonicity of $F$ in each variable yields $F(x,y)\geq(2-2^{(d-1)/d})\min\{x,y\}^{(d-1)/d}$. Set $\varepsilon=(\Deficit_d/(2-2^{(d-1)/d}))^{d/(d-1)}<1/2$. By Fubini's theorem $u$ is absolutely continuous, so these estimates extend to $[s_-,s_+]$. Its image lies in $[0,\varepsilon]\cup[1-\varepsilon,1]$ and, by continuity, in just one of these intervals. Since $u(s_-)=1-a_-/v$ and $u(s_+)=a_+/v$, either $a_-\geq(1-\varepsilon)v$ or $a_+\geq(1-\varepsilon)v$. Using $V+a_++a_-=v$ proves the assertion.
	\end{proof}
	
	\section{The transported ABP framework}\label{rev:sec:transport}
	The Alexandrov--Bakelman--Pucci (ABP) method relates a geometric energy to the image of the lower contact set of an auxiliary Neumann solution. Cabr\'e--Ros-Oton--Serra~\cite{cabre2016sharp} developed this approach for sharp anisotropic inequalities in the full space and in convex cones. For exterior problems, Liu--Wang--Weng~\cite{liu2023relative} used restricted normal cones to obtain an ABP proof of the isotropic relative inequality, while Tu~\cite{tu2007anisotropic} treated anisotropic perimeter outside a Euclidean ball at zero adhesion.
	
	Fusco--Julin--Morini~\cite[Definition~2.6]{fusco2025isoperimetriccylinders} introduced the isotropic $\lambda$-ABP property and used it to treat capillarity outside convex cylinders. Fusco--Julin--Morini--Pratelli~\cite{fusco2026isoperimetric} then obtained comparison for arbitrary convex obstacles and all contact angles in $(0,\pi)$. Their geometric estimate bounds the volume of the boundary subdifferential selected by adhesion below by that of a spherical cap; see~\cite[Theorem~3.1]{fusco2026isoperimetric}.
	
	We ask whether this mechanism extends to anisotropic capillarity under the transport $p=D\Phi(q)$, where $\Phi=\tension^2/2$. The slope $q$ determines affine contact, while $p$ determines the Neumann flux and capillary cut; these variables coincide in the Euclidean case. The chosen transport also makes the flux affine on model Wulff caps.
	
	The analytic argument bounds the transported contact-image volume from above; the transported $\lambda$-ABP property supplies a geometric lower bound. Proposition~\ref{rev:thm:abp} combines them. Sections~\ref{rev:sec:ellipsoids} and~\ref{rev:sec:obstruction} examine this property, independently of the variational comparison in Section~\ref{sec:variat_comp_planar_stab}.

	\subsection{Wulff geometry}\label{v12:sec:geometry}
	The slope region is the unit sublevel of the support function; the target is the Wulff body, where volume and cap conditions are measured. The transport identifies these generally different sets.
	
	Set $Q_\tension=\{q:\,\tension(q)<1\}$ and retain $\Phi=\tension^2/2$. When $\tension$ is smooth away from the origin, we assume that $D^2\Phi$ is uniformly positive definite there. Then
	$\Phi\in C^{1,1}(\R^d)\cap C^2(\R^d\setminus\{0\})$.
	In general, $\Phi$ is not twice differentiable at the origin. Its gradient satisfies
	\[
	D\Phi(q)\cdot q=\tension(q)^2,
	\qquad \tension^\circ(D\Phi(q))=\tension(q),
	\qquad D\Phi(Q_\tension)=\operatorname{int}\wulff.
	\]
	For $q\neq0$, $D\tension(q)$ is the support point of $W$ in direction $q$. The identity $D\Phi(q)=\tension(q)D\tension(q)$ places its image on the radial level $\tension(q)$.
	
	Here $\tension^\circ$ is the polar gauge, equivalently the gauge of the Wulff body:
	\[
	\tension^\circ(p)
	:= \sup_{\tension(q)\le1}p\cdot q
	= \inf\{t>0:p\in t\,\wulff\}.
	\]
	Since directional cap volumes may differ, the reference cap must be chosen in a minimizing direction. Choose such a direction in~\eqref{v12:eq:min-cap} and denote it by $e_d$ after a simultaneous orthogonal change of coordinates. Thus
	$m_{\tension,\lambda}(e_d)=m_{\tension,\lambda}^{\min}$.
	This is a coordinate convention for the pair $(\tension,\C)$, not an
	assumption of rotational invariance. Write
	\[
	\wcap
	:=\wulff\cap\H_{\lambda\,\tension(-e_d)}^+(e_d),
	\qquad
	\H_\lambda
	:=\overline{\H_{\lambda\,\tension(-e_d)}^-(e_d)}.
	\]
	The closed cap $K_{\tension,\lambda}(e_d)$ and $\wcap$ agree up to a set of Lebesgue measure zero.
	\medskip
	
	\subsection{The quasilinear Neumann problem}\label{v12:sec:neumann}
	The auxiliary equation is chosen so that its integrated boundary flux is exactly the capillary energy. Its constant right-hand side therefore records energy per unit volume; this is the quantity that will enter the Jacobian estimate.
	
	For a bounded open set $\Omega\subset\R^d\setminus\C$ (its regularity is specified below), write
	\[
	\Sigma=\partial\Omega\cap(\R^d\setminus\C),\qquad
	\Gamma=\partial\Omega\cap\partial\C,\qquad
	\nu_\Gamma=-\nu_{\C}.
	\]
	With the outer normal of $\Omega$ on each boundary piece, we prescribe
	\begin{equation}\label{v12:eq:neumann}
		\begin{cases}
			\operatorname{div}(D\Phi(\nabla u))=\kappa,&\text{ in }\Omega,\\[0.1em]
			D\Phi(\nabla u)\cdot\nu=\tension(\nu),&\text{ on }\Sigma,\\[0.1em]
			D\Phi(\nabla u)\cdot\nu=-\lambda\,\tension(\nu),&\text{ on }\Gamma.
		\end{cases}
	\end{equation}
	Integrating the equation and using the boundary fluxes gives the compatibility condition
	\begin{equation}\label{v12:eq:compatibility}
		\kappa|\Omega|=\EnergyC(\Omega).
	\end{equation}
	
	Compatibility fixes total flux. Boundary contact additionally requires facewise flux tests and their appropriate combination at an edge, under the following regularity and geometric assumptions. Fix $\vartheta\in(0,1)$. Let $\C$ have $C^1$ boundary, and let $\Omega$ be a bounded connected Lipschitz domain whose free and wetted boundary pieces are $C^{2,\vartheta}$ away from their common edge $\gamma=\overline\Sigma\cap\overline\Gamma$. Assume that their unit normals extend continuously to $\gamma$ and that $\calH^{d-1}(\gamma)=0$.
	
	We call $\Omega$ a \emph{$\Phi$-regular set} if the compatible Neumann problem admits a solution
	\[
	u\in W^{1,2}(\Omega)
	\cap C^{1,1}_{\mathrm{loc}}(\Omega)
	\cap C^1(\overline\Omega)
	\]
	with the following properties. The flux conditions hold classically on the open boundary pieces, and the weak formulation is
	\[
	\begin{split}
		\int_\Omega D\Phi(\nabla u)\cdot\nabla\eta\,dx
		={}& \int_\Sigma\tension(\nu_\Sigma)\eta\,d\calH^{d-1}
		-\lambda\int_\Gamma\tension(\nu_\Gamma)\eta\,d\calH^{d-1}
		-\kappa\int_\Omega\eta\,dx
	\end{split}
	\]
	for every $\eta\in W^{1,2}(\Omega)$.
	
	At every edge point, assume all the edge hypotheses of Lemma~\ref{lem:touch}, including the local $C^1$ wedge representation and accumulation from both open boundary pieces. The lemma yields the boundary-touching condition: if an affine function $\ell$ of slope $q$ satisfies $\ell\leq u$ on $\overline\Omega$ and $\ell(x)=u(x)$ at $x\in\partial\Omega$, then
	\[
	\begin{cases}
		D\Phi(q)\cdot\nu_\Sigma(x)-\tension(\nu_\Sigma(x))
		\geq 0,
		&x\in\Sigma\setminus\gamma,\\[.25em]
		D\Phi(q)\cdot\nu_\Gamma(x)+\lambda\,\tension(\nu_\Gamma(x))
		\geq0,
		&x\in\Gamma\setminus\gamma,\\[.25em]
		\max\bigl\{
		D\Phi(q)\cdot\nu_\Sigma(x)-\tension(\nu_\Sigma(x)),
		\, D\Phi(q)\cdot\nu_\Gamma(x)+\lambda\,\tension(\nu_\Gamma(x))
		\bigr\} \geq0,
		&x\in\gamma.
	\end{cases}
	\]
	These regularity and edge requirements are the analytic hypothesis of Proposition~\ref{rev:thm:abp}; weak solvability alone does not supply them, and no general existence theorem with these properties is asserted.
	
	At a boundary minimum, the supporting slope differs from $\nabla u$ by an outward normal correction, or a nonnegative combination of the two normals at an edge. Flux monotonicity makes the corresponding weighted sum of boundary tests nonnegative, so both tests cannot be strictly negative.
	\begin{lemma}[Boundary touching]\label{lem:touch}
		Let $u$ be as above, with the flux conditions holding classically on $\Sigma\setminus\gamma$ and $\Gamma\setminus\gamma$. Assume that every point of $\gamma$ lies in the closure of both $\Sigma\setminus\gamma$ and $\Gamma\setminus\gamma$, and that near each $x\in\gamma$ there is a $C^1$-diffeomorphism $\Psi$ with $\Psi(x)=x$, $D\Psi(x)=I_d$, mapping $\overline{\Omega}$ locally onto the wedge $x+T_x$, where 
		\[
		T_x=\{w:\ w\cdot\nu_\Sigma(x)\leq 0,\ w\cdot\nu_\Gamma(x)\leq 0\}.
		\]    
		Then the boundary-touching condition holds.
	\end{lemma}
	
	\begin{proof}[Proof of Lemma~\ref{lem:touch}]
		By the assumed local wedge representation, $T_x$ is a convex wedge, reducing to a half-space if $\nu_\Sigma=\nu_\Gamma$. The case $\nu_\Sigma=-\nu_\Gamma$ would give $T_x$ empty interior, contradicting its local identification with the full-dimensional domain $\Omega$. The continuity of $\nabla u$ and of the normals ensures that both flux identities hold at $x\in\gamma$. If $\nu_\Sigma(x)=\nu_\Gamma(x)$ they would give $(1+\lambda)\tension(\nu_\Gamma)=0$, which is impossible; so contact is transversal.
		
		Let $\ell(y)=q\cdot y+b\leq u$ on $\overline\Omega$ with equality at $x\in\partial\Omega$. Since $u-\ell\in C^1$ up to the boundary and is minimal at $x$, $(\nabla u(x)-q)\cdot w\geq0$ for every $w$ in the tangent cone. Hence $q-\nabla u(x)$ lies in the polar cone: $q-\nabla u(x)=s\,\nu_\Sigma$ or $s\nu_\Gamma$ (for $s\geq 0$) on the open pieces, and $a\,\nu_\Sigma+b\,\nu_\Gamma$ (where $a,b\geq 0$) on $\gamma$. Monotonicity of $D\Phi$ and the flux identities give
		\[
		\begin{aligned}
			&a\big(D\Phi(q)\cdot\nu_\Sigma-\tension(\nu_\Sigma)\big)
			+b\big(D\Phi(q)\cdot\nu_\Gamma+\lambda\,\tension(\nu_\Gamma)\big)\\
			&\qquad=\big(D\Phi(q)-D\Phi(\nabla u(x))\big)\cdot\big(q-\nabla u(x)\big)\geq0,
		\end{aligned}
		\]
		and the analogous one-term inequalities on the open pieces. If all coefficients vanish, both tests equal $0$. Otherwise some test is nonnegative.
	\end{proof}
	
	For isotropic variational solutions, boundary touching follows from~\cite[Proposition~2.2]{fusco2026isoperimetric}. See also~\cite{dupuis1991oblique} for oblique problems with corners. The H\"older estimates for linear problems in~\cite{nittka2011regularity} do not supply the $C^1$ regularity assumed here.
	
	On model caps, the dual potential reverses the transport, producing a dilation with equal expansion in every direction. This saturates the determinant estimate and supplies model equality configurations in the analytic class.
	\begin{remark}[Model caps are $\Phi$-regular]\label{rem:modelcap}
		Let $\tension$ be smooth away from the origin, with $D^2\Phi$ uniformly positive definite there, and let $\C$ be a $C^1$-convex obstacle having a flat face with outer normal $n$, where $m_{\tension,\lambda}(n)=m^{\min}_{\tension,\lambda}$. Let $K=x_0+tK_{\tension,\lambda}(n)$ have its facet in that face, and write $\Phi^*=(\tension^\circ)^2/2$ for the Legendre transform of $\Phi$. Then $u(x)=t\Phi^*((x-x_0)/t)$ satisfies $D\Phi(\nabla u)=(x-x_0)/t$ and solves~\eqref{v12:eq:neumann} with $\kappa=d/t$. Moreover, $u\in C^{1,1}(\R^d)$ because $D\Phi^*=(D\Phi)^{-1}$ is Lipschitz, $D\Phi$ being strongly monotone. The smoothness and ellipticity assumptions give the required regularity of the free boundary. The edge is transversal and Lemma~\ref{lem:touch} applies, since $-1<\lambda<\tension(n)/\tension(-n)$. Thus the interior of $K$ is $\Phi$-regular. This holds for every such $\lambda$, including $\lambda\leq0$, and~\eqref{v12:eq:smooth-comparison} is attained in the class.
	\end{remark}
	
	\subsection{The transported ABP estimate}
	At points where an affine function supports $u$ from below, the Hessian is nonnegative. The equation controls the trace of the derivative of the transported gradient, and the arithmetic--geometric mean inequality then controls its determinant. This gives an upper bound for the image volume in terms of the capillary energy.
	
	Let $u$ be the auxiliary solution. By Lemma~\ref{rev:lem:coercivity} and compatibility, $\kappa>0$. Since $D\Phi$ is Lipschitz and $u\in C^{1,1}_{\mathrm{loc}}(\Omega)$, the map
	$T_u:=D\Phi(\nabla u)$
	is locally Lipschitz.
	
	Since $D^2\Phi(0)$ may not exist, first exclude zero gradients. Set $Z=\{x\in\Omega:\nabla u(x)=0\}$. On $Z$ we have $T_u=0$, and locality of weak derivatives gives $DT_u=0$ almost everywhere on $Z$. The equation $\operatorname{div}T_u=\kappa>0$ therefore implies $|Z|=0$. Consequently,
	\[
	DT_u=A D^2u,
	\qquad
	A=D^2\Phi(\nabla u),
	\]
	almost everywhere in $\Omega$.
	
	Define the lower contact set by
	\[
	\Omega^+
	=
	\left\{x\in\Omega:\, u(y)\ge u(x)+\nabla u(x)\cdot(y-x)\; \text{ for every }y\in\overline\Omega\right\}.
	\]
	At almost every point of $\Omega^+$, the Hessian $D^2u$ exists and is nonnegative. Hence
	\[
	S=A^{1/2}(D^2u)A^{1/2}\ge0,
	\qquad
	\det DT_u=\det S,
	\qquad
	\operatorname{tr}S=\kappa.
	\]
	The symmetric nonnegative matrix $S$ has the same trace and determinant as $DT_u$. Its eigenvalue product is maximal at equal eigenvalues. The area formula, by exhaustion, and the arithmetic--geometric mean inequality give
	\begin{equation}\label{rev:eq:abp-bound}
		|T_u(\Omega^+)|
		\leq\int_{\Omega^+}\det S\,dx
		\leq\left(\frac{\kappa}{d}\right)^d|\Omega^+|
		\leq\left(\frac{\kappa}{d}\right)^d|\Omega|.
	\end{equation}
	The Jacobian integral counts multiplicity, whereas image volume counts each point once; injectivity of $T_u$ is unnecessary. On a model cap, dilation attains both the image-volume and determinant bounds.
	
	\subsection{Restricted subdifferentials and the transported boundary image}\label{v12:sec:boundary-image}
	Restricted subdifferentials record slopes supporting the boundary data. The wall condition selects transported images whose slopes, by boundary touching, must be attained at interior contact points.
	
	For a nonempty compact set $\mathcal X\subset\partial\C$ and $\psi\in C(\mathcal X)$, define
	\[
	\partial_{\mathcal X}\psi(x)
	:=\{q\in\R^d:\,\psi(y)\geq\psi(x)+q\cdot(y-x)
	\;\text{ for every }y\in \mathcal X\}.
	\]
	Thus $q\in\partial_{\mathcal X}\psi(x)$ means that $x$ minimizes $\psi(y)-q\cdot y$ on $\mathcal X$. If $n\in\mathbb{S}^{d-1}$ is a supporting unit normal to $\C$ at $x$, then $n\cdot(y-x)\leq0$ on $\mathcal X$, and hence
	\begin{equation}\label{v12:eq:propagation}
		\partial_{\mathcal X}\psi(x)+[0,\infty)\,n\subset\partial_{\mathcal X}\psi(x).
	\end{equation}
	For finite data, a slope selects a minimum of $\psi(x_i)-q\cdot x_i$; changing the heights shifts the cell interfaces. When $\mathcal X=F=\{x_1,\ldots,x_N\}$, the cells $\mathcal A_i=\partial_F\psi(x_i)$ cover $\R^d$, with intersections contained in finitely many affine hyperplanes. This assertion concerns Lebesgue measure in slope space; it does not imply that the intersections have zero Alexandrov measure for a crystalline potential.
	Each cell carries its boundary point's supporting normal. Capture requires $p=D\Phi(q)$ to pass the cap test for an available cell, not merely to belong to some directional cap. For smooth $\tension$ and a specified continuous normal field $\mathbf n$ on $\mathcal X$, set
	\begin{equation}\label{v12:eq:boundary-image}
		\begin{aligned}    
			\mathscr B_{\tension,\lambda}^{\psi}&(\mathcal X,\mathbf n)\\
			&:=\bigcup_{x\in \mathcal X}
			\{D\Phi(q):q\in\partial_{\mathcal X}\psi(x),\,
			D\Phi(q)\in\operatorname{int}\wulff,\,
			D\Phi(q)\cdot\mathbf n(x)>\lambda\,\tension(-\mathbf n(x))\}.
		\end{aligned}
	\end{equation}
	For a $C^1$-obstacle we take $\mathbf n=\nu_{\C}$. For finite $F$ on a nonsmooth obstacle, each $n_i$ is a chosen supporting unit normal. The compact-set image is defined by~\eqref{v12:eq:boundary-image} itself, rather than by an unspecified discretization limit. For a general convex surface tension, replace the singleton $\{D\Phi(q)\}$ in~\eqref{v12:eq:boundary-image} by $\partial\Phi(q)$, retaining precisely those $p\in\partial\Phi(q)$ such that $p\in\operatorname{int}\wulff$ and $p\cdot\mathbf n(x)>\lambda\,\tension(-\mathbf n(x))$.
	
	A compact exhaustion with positive margins in the strict inequalities proves Borel measurability. Indeed, continuity of $\psi$ and convexity of $\Phi$ make
	\[
	\mathcal G=\{(x,q,p):\, x\in \mathcal X,\; q\in\partial_{\mathcal X}\psi(x),\; p\in\partial\Phi(q)\}
	\]
	closed. For positive integers $j,k$, intersect $\mathcal G$ with $|q|\leq j$, $p\in\wulff$, $\dist(p,\partial\wulff)\geq1/k$, and $p\cdot\mathbf n(x)-\lambda\,\tension(-\mathbf n(x))\geq1/k$. Continuity of the prescribed normal field makes this set compact. The transported image is the countable union of its compact $p$-projections, hence $F_\sigma$. The same argument applies to finite sets with the discrete topology. The graph formulation below gives the same finite image up to $\partial\wulff$.
	
	\paragraph{Nonsmooth transport and its graph measure.}\label{v12:sec:graph}
	For nonsmooth anisotropies, one slope can have several transported images, and only some of them may satisfy the cap test. We therefore record slope and image together. The resulting graph measure keeps track of the volume selected by each labeled cell while counting overlaps only once.
	
	Let $W=\wulff$ be a convex body and $\Phi=\tension^2/2$. The gauge of $W$ is $\rho_W=\tension^\circ$, so the Legendre transform is
	$\Phi^*(p)=(\tension^\circ(p))^2/2=\rho_W(p)^2/2$.
	Indeed, maximizing first on each level set of $\tension$ reduces the Legendre transform to $\sup_{r\geq0}\{r\rho_W(p)-r^2/2\}$. Convex duality gives $p\in\partial\Phi(q)$ if and only if $q\in\partial\Phi^*(p)$. Although one slope may have many images, almost every image has a unique inverse slope: local Lipschitz continuity of $\Phi^*$ gives $G_W=D\Phi^*$ almost everywhere. Consequently,
	\begin{equation}\label{v12:eq:inverse-image}
		|\partial\Phi(A)\cap W|
		=|\{p\in W:G_W(p)\in A\}|
	\end{equation}
	for every Borel set $A$, with completed Lebesgue measure understood.
	
	The Alexandrov measure is $\mu_\Phi(A)=|\partial\Phi(A)|$. For the cap restriction, retain both transport variables by defining the finite graph measure
	\begin{equation}\label{v12:eq:graph-measure}
		\gamma_\Phi
		:=(G_W,\operatorname{id})_\#(\mathcal L^d\mres W).
	\end{equation}
	Here $\#$ denotes pushforward: the measure of a set of pairs is the volume of those $p\in W$ for which $(G_W(p),p)$ belongs to that set. Its second marginal is Lebesgue measure on $W$, and its first is $A\mapsto|\partial\Phi(A)\cap W|$; retaining the pair allows the cap condition to be tested in the image variable. In particular, if
	$\mathcal C_{n,\lambda} :=W\cap \H_{\lambda\varphi(-n)}^+(n)$,
	then for any finite labeled family $\{(\mathcal A_i,n_i)\}_{i=1}^N$,
	\begin{equation}\label{v12:eq:labeled-graph}
		\left|\bigcup^N_{i=1}\bigl(\partial\Phi(\mathcal A_i)
		\cap\mathcal C_{n_i,\lambda}\bigr)\right|
		=\gamma_\Phi\left(\bigcup^N_{i=1} \mathcal A_i\times\mathcal C_{n_i,\lambda}\right).
	\end{equation}
	This is a union, so overlap is counted only once. No assumption on the Alexandrov measure of cell interfaces is needed.
	
	For $\mathcal A_i=\partial_F\psi(x_i)$, define
	\[
	\mathscr B_{\tension,\lambda}^{\psi}(F,\mathbf n)
	:=\bigcup^N_{i=1}\bigl(\partial\Phi(\mathcal A_i)\cap\mathcal C_{n_i,\lambda}\bigr).
	\]
	Equation \eqref{v12:eq:labeled-graph} gives its volume. The unrestricted finite inequality to be tested is
	\begin{equation}\label{v12:eq:crystalline-target}
		\left|\bigcup^N_{i=1}\bigl(\partial\Phi(\mathcal A_i)\cap\mathcal C_{n_i,\lambda}\bigr)\right|
		\geq m_{\tension,\lambda}^{\min}.
	\end{equation}
	The half-line property~\eqref{v12:eq:propagation} still holds in slope space. For smooth elliptic tensions, monotonicity of $D\Phi$ makes $D\Phi(q+tn)\cdot n$ increase along such a ray, but supplies no comparison of the volumes selected from different cells. The required lower bound therefore remains a separate geometric assertion.
	
	Existence of an admissible image does not justify counting a slope's entire Alexandrov mass. For $W=[-1,1]^2$ and $\tension(q)=|q_1|+|q_2|$, one has $\partial\Phi(r,0)=\{(r,s):|s|\leq r\}$ for $0<r<1$. Each slope has images on both sides of $p_2=0$. The whole segment of slopes has image area $1$, but only area $1/2$ satisfies $p_2>0$. Formula~\eqref{v12:eq:labeled-graph} retains this pointwise restriction.
	
	\subsection{The transported ABP property}\label{v12:sec:property}
	We introduce the \emph{transported $\lambda$-ABP property} as an anisotropic counterpart of the isotropic notion above. It requires the selected boundary image to have volume at least that of a smallest directional Wulff cap, uniformly over obstacles and arbitrary boundary data. This formulation allows the geometric assertion to be studied independently of the auxiliary equation.
	\begin{definition}\label{rev:def:abp}
		Let $-1<\lambda<\lambda_*(\tension)$. The Wulff body $\wulff$ satisfies the~\emph{transported $\lambda$-ABP property} if
		\begin{equation}\label{v12:eq:transported-property}
			|\mathscr B_{\tension,\lambda}^{\psi}(\mathcal X,\mathbf n)|
			\geq m_{\tension,\lambda}^{\min}
		\end{equation}
		for every $C^1$-convex obstacle, every nonempty compact $\mathcal X\subset\partial\C$, and every $\psi\in C(\mathcal X)$, with $\mathbf n=\nu_{\C}$, and also for every finite set on any convex obstacle with any choice of supporting unit normals.
	\end{definition}
	
	The compact-data estimate enters Proposition~\ref{rev:thm:abp} with $\mathcal X=\Gamma$ and $\psi=u|_\Gamma$, whereas finite configurations suffice for the rigidity and obstruction arguments of Section~\ref{rev:sec:obstruction}. Arbitrary boundary heights need not be traces of auxiliary Neumann solutions. Failure of this unrestricted property therefore does not, by itself, disprove either an estimate restricted to such traces or the capillary inequality.
	
	For finite configurations on nonsmooth obstacles, the interiors of the cells $\partial_F\psi(x_i)$ are disjoint open polyhedra covering $\R^d$ up to a null set and satisfying the half-line property~\eqref{v12:eq:propagation}.  This is the setting of the remark at the end of~\cite[Section~1.1]{fusco2026isoperimetric} and of the proof of~\cite[Theorem~3.1] {fusco2026isoperimetric}.
	
	The minimum-direction constant is necessary. For a singleton $\mathcal X=\{x\}$, the restricted subdifferential is all of $\R^d$, so its image is exactly the directional cap associated with $\mathbf n(x)$. Consequently, an unrestricted estimate with a fixed reference direction $n_0$ can hold only if $m_{\tension,\lambda}(n_0)=m_{\tension,\lambda}^{\min}$.
	
	\paragraph{Geometry behind validity and failure.}
	For a centered ellipsoidal Wulff body, $\Phi$ is quadratic and $D\Phi$ is linear. Dual affine changes of the slope and image variables then normalize the cells and the cap cuts simultaneously, recovering the Euclidean estimate (Proposition~\ref{v12:prop:ellipsoid-transport}). This compatibility explains why the cap test is imposed after transport; imposing it directly on slopes leads to the failure described in Remark~\ref{rem:4.3}.
	
	A facet gives a different mechanism. Its pyramid with vertex at the origin has positive volume, but $G_W=D\Phi^*$ maps its interior to the ray generated by the facet normal. Thus a negligible set of slopes can represent a substantial volume of image points. In the two-point configurations of Theorem~\ref{rev:thm:facet}, with points $0,v$, normals $n,-n$, equal heights and $\lambda=0$, capture away from the shared cell interface is equivalent, up to the cap hyperplane, to $(G_W(p)\cdot v)(p\cdot n)<0$. The sign of the first factor assigns the cell normal; the second decides whether the image passes that normal's cap test. Small positive tilts can assign the opposite facet pyramids to the wrong normals, excluding a positive volume. Setting the tilt equal to zero can instead put the inverse slopes on the shared interface, where both labels are available. The argument therefore keeps positive tilts while taking volume limits, as detailed in Section~\ref{rev:sec:obstruction}. A strict deficit is realized on a smooth strictly convex obstacle.
	
	For smooth elliptic tensions, the transport is invertible, but it changes the relative position of cell interfaces and cap cuts and weights slope volume by $\det D^2\Phi$. These effects need not balance.
	
	For smooth even tensions at zero adhesion, varying the height difference of two boundary points moves the interface between their slope cells. In the thin-slab limit, a configuration and its reflection select complementary parts of the Wulff body, so both can meet the half-volume bound only by exact balance. Requiring this for every height and direction forces balance on individual slope fibers, with the transport density above, and ultimately $D^3\Phi=0$ (Theorem~\ref{thm:two-point-rigidity}).
	
	The quartic perturbations in Appendix~\ref{app:perturbation} exhibit the imbalance explicitly at prescribed adhesion levels: after a suitable rotation, the selected volume decreases relative to the minimum cap volume. Figure~\ref{fig:smooth-transport-failure} illustrates this loss of captured volume despite smoothness and uniform ellipticity.
	
	\subsection{Conditional transported comparison}\label{v12:sec:abp-theorem}
	Slide a selected affine function vertically until it touches $u$ on $\overline\Omega$. The slope's image in $\operatorname{int}W$ excludes free-boundary contact; its strict cap condition excludes wetted contact, including the edge. Contact is therefore interior. The geometric lower bound and~\eqref{rev:eq:abp-bound} then give the comparison for the $\Phi$-regular class. For ellipsoids, Corollary~\ref{v12:cor:ellipsoids} separately treats all finite-perimeter competitors.
	
	Write $m=m^{\min}_{\tension,\lambda}=|\wcap|$. For $v>0$, set $t_v=(v/m)^{1/d}$, $\wcap[v]=t_v\wcap$, and $\H_{\lambda,v}=t_v\H_\lambda$. Then $|\wcap[v]|=v$, its facet lies on $\partial\H_{\lambda,v}$, and the model-cap computation gives
	\begin{equation}\label{v12:eq:model-scaling}
		\E_{\tension,\lambda,\H_{\lambda,v}}(\wcap[v])=d m^{1/d}v^{(d-1)/d}.
	\end{equation}
	The simultaneous dilation in~\eqref{v12:eq:model-scaling} leaves the adhesion coefficient fixed.
	\begin{proposition}[Conditional transported comparison]\label{rev:thm:abp}
		Assume $-1<\lambda<\lambda_*(\tension)$, that $\tension$ is smooth away from the origin, $D^2(\tension^2/2)$ is uniformly positive definite there, and $\wulff$ satisfies the transported $\lambda$-ABP property. Let $\C$ be a $C^1$-convex obstacle. Every $\Phi$-regular competitor $\Omega\subset\R^d\setminus\C$ satisfies
		\begin{equation}\label{v12:eq:smooth-comparison}
			\EnergyC(\Omega)
			\geq d(m_{\tension,\lambda}^{\min})^{1/d}|\Omega|^{(d-1)/d}.
		\end{equation}
	\end{proposition}
	
	\begin{proof}[Proof of Proposition~\ref{rev:thm:abp}]
		If \(\calH^{d-1}(\Gamma)=0\), the energy is \(\PerAni(\Omega)\), and the Wulff inequality gives the result, since \(m^{\min}_{\tension,\lambda}\le |\wulff|\). Otherwise, let $u$ be the auxiliary solution and set
		$\mathscr B=\mathscr B_{\tension,\lambda}^{u|_\Gamma}(\Gamma,\nu_{\C})$.
		We first show that $\mathscr B\subset T_u(\Omega^+)$. Choose $p=D\Phi(q)\in \mathscr B$ and a point $x_0\in\Gamma$ for which $q\in\partial_\Gamma(u|_\Gamma)(x_0)$ and $p\cdot\nu_{\C}(x_0)>\lambda\,\tension(-\nu_{\C}(x_0))$. The function $u-q\cdot x$ has its minimum on $\Gamma$ at $x_0$. This restricted minimum is why the normal attached to the selected cell can be used: if the global minimum on $\overline\Omega$ were attained anywhere on $\Gamma$, it would also be attained at $x_0$. At this point the wet-boundary test is strictly negative:
		$p\cdot\nu_\Gamma(x_0)+\lambda\,\tension(\nu_\Gamma(x_0))<0$.
		If $x_0$ lies on the edge, the free-boundary test is also strictly negative, since $p\in\operatorname{int}\wulff$. Both possibilities contradict the boundary-touching condition. At any point of $\Sigma\setminus\gamma$, the free-boundary test is strictly negative for the same reason. Thus the global minimum is attained in $\Omega$. There $q=\nabla u$ and the point belongs to $\Omega^+$, proving the inclusion.
		
		The transported property and~\eqref{rev:eq:abp-bound} give
		\[
		0<m
		\leq |\mathscr B|
		\leq |T_u(\Omega^+)|
		\leq \left(\frac{\kappa}{d}\right)^d|\Omega|.
		\]
		Taking a $d$-th root and using $\kappa|\Omega|=\EnergyC(\Omega)$ proves the asserted comparison.
	\end{proof}
	
	\section{Ellipsoids}\label{rev:sec:ellipsoids}
	For ellipsoids, a linear change of spatial variables converts both surface energies into their Euclidean counterparts, giving the comparison for all finite-perimeter competitors without auxiliary Neumann regularity. Dual changes in slope and image variables separately verify the transported boundary-image property.
	
	Let $\M$ be a symmetric and positive definite matrix, put $L=\M^{1/2}$, and set
	\[
	W_\M=\{p:p\cdot \M^{-1}p\leq1\}=L\,\overline{B_1},
	\qquad \tension_\M(\nu)=\sqrt{\nu\cdot\M\nu}.
	\]
	Then $\Phi_\M(q)=\tfrac12q\cdot\M q$ and $D\Phi_\M(q)=\M q$. The following comparison follows from~\cite[Theorem~1.1]{fusco2026isoperimetric}, including unbounded convex obstacles and finite-perimeter competitors.
	
	\begin{corollary}[Affine capillary comparison]\label{v12:cor:ellipsoids}
		Let $-1<\lambda<1$, let $\C$ be a convex obstacle, and let $E\subset\R^d\setminus\C$ have finite perimeter and $0<|E|<\infty$. Then
		\[
		\E_{\tension_\M,\lambda,\C}(E)
		\geq d\bigl((\det L)|B_1\cap\{y:\,y_d>\lambda\}|\bigr)^{1/d}
		|E|^{(d-1)/d}.
		\]
		Equality holds precisely when, up to a null set, $E=x_0+tK_{\tension_\M,\lambda}(n)$ for some $t>0$ and $n\in\mathbb{S}^{d-1}$, with the entire truncating facet contained in a supporting face of $\C$ with outer normal $n$.
	\end{corollary}
	
	\begin{proof}[Proof of Corollary~\ref{v12:cor:ellipsoids}]
		Set $\widetilde E=L^{-1}E$ and $\widetilde\C=L^{-1}\C$. The linear area formula gives
		\[
		\nu_E(Ly)=\frac{L^{-\mathsf T}\nu_{\widetilde E}(y)}
		{|L^{-\mathsf T}\nu_{\widetilde E}(y)|},
		\qquad d\calH^{d-1}(Ly)
		=(\det L)\,|L^{-\mathsf T}\nu_{\widetilde E}(y)|\,d\calH^{d-1}(y).
		\]
		Since $\tension_\M(\xi)=|L\xi|$ and $L=L^{\mathsf T}$, the density cancels the normal-dependent area factor. Both surface terms acquire $\det L$, leaving $\lambda$ unchanged:
		\[
		\E_{\tension_\M,\lambda,\C}(E)
		=(\det L)\,\E_{|\cdot|,\lambda,\widetilde\C}(\widetilde E),
		\qquad |E|=(\det L)\,|\widetilde E|.
		\]
		The Euclidean inequality proves the bound. Its equality case gives a spherical cap on a flat supporting face. Under $x=Ly$, a cap with unit cutting normal $\omega\in\mathbb{S}^{d-1}$ becomes a cap of $W_\M$ with
		$n=L^{-\mathsf T}\omega/|L^{-\mathsf T}\omega|$.
		Indeed, $\tension_\M(-n)=|Ln|=1/|L^{-\mathsf T}\omega|$. This transforms both the truncation level and the supporting face as asserted. Conversely, every such ellipsoidal cap transforms to a Euclidean equality configuration.
	\end{proof}
	
	To verify the transported property, we now normalize the boundary cells and apply the following Euclidean estimate.
	
	\paragraph{Euclidean finite partition estimate.}
	We use the following precise form of the geometric estimate in Fusco--Julin--Morini--Pratelli~\cite[proof of Theorem~3.1, (3.13)--(3.15)]{fusco2026isoperimetric}, in arXiv version~3. Let $U_1,\ldots,U_N$ be disjoint open sets covering $\R^d$ up to a null set, and let $n_i\in\mathbb S^{d-1}$. Assume
	\[
	\calH^{d-1}(\partial U_i\cap\partial B_r)=0\quad(r>0),
	\qquad U_i+[0,\infty)\,n_i\subset U_i.
	\]
	For $-1<\lambda<1$, their partition estimate gives
	\[
	\left|\bigcup_{i=1}^N\{p\in U_i\cap B_1:p\cdot n_i>\lambda\}\right|
	\geq |B_1\cap\{p:p_d>\lambda\}|.
	\]
	At $\lambda=0$, reflection $p\mapsto p-2(p\cdot n_i)n_i$ moves a negative-side point forward along an allowed ray, remaining in its cell and preserving the ball and volume. Thus each cell contributes at least half its volume, and summation gives the half-ball bound. For nonzero adhesion we use the cited partition estimate.
	
	Slopes transform dually to positions, preserving affine pairings; images transform with the Wulff body. The quadratic potential identifies the two normalized variables, converting cell inequalities and cap cuts simultaneously to their Euclidean forms.
	\begin{proposition}[Affine transport identity]\label{v12:prop:ellipsoid-transport}
		The ellipsoid $W_\M$ satisfies Definition~\ref{rev:def:abp} for every $-1<\lambda<1$.
	\end{proposition}
	
	\begin{proof}[Proof of Proposition~\ref{v12:prop:ellipsoid-transport}]
		Use $x=L\widetilde x$, $p=L\eta$, and $q=L^{-1}\eta$. Then $q\cdot x=\eta\cdot\widetilde x$ and $p=\M q$, so $\eta$ describes both the normalized slope and image. For $\widetilde{\mathcal X}=L^{-1}\mathcal X$ and $\widetilde\psi(\widetilde x)=\psi(L\widetilde x)$, the change of coordinates acts so that 
		\[
		q\in\partial_{\mathcal X}\psi(x)
		\quad\Longleftrightarrow\quad
		\eta\in\partial_{\widetilde{\mathcal X}}\widetilde\psi(\widetilde x).
		\]
		A supporting normal $n$ transforms to $\widetilde n=Ln/|Ln|$, and
		\[
		p\cdot n>\lambda\,\tension_\M(-n)
		\quad\Longleftrightarrow\quad
		\eta\cdot\widetilde n>\lambda.
		\]
		Thus the transported image is the linear image by $L$ of the Euclidean restricted image. The Euclidean geometric estimate of Fusco--Julin--Morini--Pratelli,~\cite[Theorem~3.1]{fusco2026isoperimetric}, gives the compact case (together with~\cite[Lemma~3.3]{fusco2026isoperimetric}).
		
		For finite data on an arbitrary convex obstacle, put $U_i=\operatorname{int}\partial_{\widetilde{\mathcal X}}\widetilde\psi(\widetilde x_i)$ and use the chosen supporting normals $\widetilde n_i$. These polyhedral interiors are disjoint and cover slope space up to finitely many hyperplanes. Their boundaries meet every centered sphere in a null set. For $q\in U_i$, the cell inequalities are strict; adding $t\widetilde n_i$ preserves them because $\widetilde n_i\cdot(\widetilde x_k-\widetilde x_i)\leq0$.
		
		Thus the stated partition estimate applies. Cell boundaries have zero Lebesgue measure, so replacing the interiors by closed cells does not change the Euclidean image volume. In particular, at zero adhesion this finite step follows from the direct reflection proof above. Finally,
		\[
		m_{\tension_\M,\lambda}(n)
		=(\det L)\,|B_1\cap\{\eta:\,\eta\cdot(Ln/|Ln|)>\lambda\}|
		=(\det L)\,|B_1\cap\{\eta:\,\eta_d>\lambda\}|
		\]
		is independent of $n$. This proves the required lower bound.
	\end{proof}
	
	A direct slope cut transforms both hyperplane normals in the same way. Originally orthogonal directions may become oblique, unbalancing the selected volumes even for an ellipsoid.
	\begin{remark}[A linear ABP estimate fails already for ellipsoids]\label{rem:4.3}
		The linear ABP route (\cite{cabre2016sharp}; Tu~\cite{tu2007anisotropic} for a ball obstacle at $\lambda=0$) imposes the cap condition on the slope:
		\[
		\mathcal B^{\rm lin}=\bigcup_{x\in \mathcal X}\{q\in\partial_{\mathcal X}\psi(x):\, q\in\operatorname{int} W,\; q\cdot\mathbf n(x)>\lambda\,\tension(-\mathbf n(x))\}.
		\]
		Over all smooth convex obstacles, this unrestricted estimate fails for every non-round ellipsoid $W=LB_1$, already at $\lambda=0$.
		\begin{enumerate}[label=$\bullet$,itemsep=0em]
			\item \emph{Directions.} Choose $n$ that is not an eigenvector of $L^2$, and $w\perp n$ with $n\cdot L^2w\neq0$.
			\item \emph{Configuration.} Take the two-point configuration of Theorem~\ref{rev:thm:facet}: points $0,v$ on an ellipsoidal obstacle (distinct from $\wulff$) with normals $\pm n$, $n\cdot v<0$, $\psi=0$, and $v=\pm w-\varepsilon n$, $\varepsilon\to0^+$. By central symmetry $|\mathcal B^{\rm lin}|=2|W\cap\{q\cdot v<0\}\cap\{q\cdot n>0\}|$, which tends to $2|W\cap\{\pm q\cdot w<0\}\cap\{q\cdot n>0\}|$.
			\item \emph{Conclusion.} Under $q=L\eta$ these are two wedges of the ball between hyperplanes with normals $Ln$ and $\pm Lw$. They have equal volume if and only if $n\cdot L^2w=0$, so one limit is $<|W|/2=m^{\min}_{\tension,0}$.
		\end{enumerate}
		
		For a concrete planar example, take $L=\operatorname{diag}(2,1)$, $n=(1,1)/\sqrt2$ and $w=(-1,1)/\sqrt2$. With $v=-w-\varepsilon n$, the limiting selected set consists of the two regions of the ellipse where $|q_2|>|q_1|$. In unit-disk coordinates each has angular width $\arccos(3/5)$, so the total selected area tends to $2\arccos(3/5)<\pi=|W|/2$. The deficit persists for sufficiently small positive $\varepsilon$, when the two points and normals are realized by a smooth strictly convex obstacle as above.
		
		On a ball obstacle, opposite normals force $v\parallel n$, consistently with~\cite{tu2007anisotropic}. In the transported formulation the cap is cut in $p=\M q$, and Proposition~\ref{v12:prop:ellipsoid-transport} restores the estimate. Indeed, in the common variable $\eta$, the limiting cell and cap cuts have normals $L^{-1}w$ and $Ln$, whose scalar product is $w\cdot n=0$. The dual transformation preserves the orthogonality that the direct slope cut loses.
		
		Affine covariance holds more generally: if $\tension_L(q)=\tension(L^Tq)$, then every differentiable profile $F$ satisfies
		$D[F(\tension_L)](q)=L\,D[F(\tension)](L^Tq)$.
		The quadratic profile additionally provides the two-homogeneity used in Theorem~\ref{thm:two-point-rigidity}.
	\end{remark}
	
	\section{Obstructions and rigidity of transported boundary images}\label{rev:sec:obstruction}
	The ellipsoidal case in Section~\ref{rev:sec:ellipsoids} identifies a setting in which the transported boundary-image estimate retains its Euclidean form. This raises the question of whether the estimate extends to more general anisotropies. Two developments motivate this question: the anisotropic ABP inequalities in convex cones of Cabr\'e--Ros-Oton--Serra~\cite{cabre2016sharp}, and the isotropic capillary inequality outside arbitrary convex obstacles of Fusco--Julin--Morini--Pratelli~\cite{fusco2026isoperimetric}. In the latter work, a geometric estimate for boundary subdifferentials supplies the lower bound on image volume needed by the ABP argument. The issue here is whether that geometric mechanism survives anisotropic transport.
	
	We examine the unrestricted property of Definition~\ref{rev:def:abp}, which quantifies over arbitrary boundary values. All counterexamples below use finite boundary data with too little selected image volume. Such data need not arise from an auxiliary Neumann solution, so these examples do not by themselves disprove the capillary inequality or its required image estimate for actual Neumann traces.
	
	We first exhibit the loss of image volume caused by facets. We then show that, within the smooth even class of Theorem~\ref{thm:two-point-rigidity}, validity of all two-point tests at zero adhesion forces the potential to be quadratic. Finally, persistence of strict deficits connects these mechanisms to approximation, and explicit smooth perturbations yield failures at prescribed adhesion levels.
	
	Throughout this section, $h_W(q)=\max_{p\in W}p\cdot q$ denotes the support function of $W$, and $\rho_W(p)=\inf\{t>0:p\in tW\}$ its gauge (Minkowski functional). We write
	\[
	\Phi=\frac12h_W^2,
	\qquad
	G_W=D\Phi^*=D\left(\frac12\rho_W^2\right)
	\]
	wherever the gradient exists. The obstacle $\C$ determines the boundary points and their supporting normals; the Wulff body $W$ determines the transport and the reference cap volumes.
	
	For finite data, the following pointwise description will guide every construction. Given almost any $p\in W$, find its inverse slope $q=G_W(p)$ and the cells $\mathcal A_i$ containing $q$. Each such cell assigns its normal $n_i$ to $p$, and
	\[
	p\in\mathscr B
	\quad\Longleftrightarrow\quad
	\text{there is an }i\text{ such that }
	G_W(p)\in\mathcal A_i
	\ \text{and}\ 
	p\cdot n_i>\lambda\, h_W(-n_i).
	\]
	The cells cover slope space, so almost every point of $W$ receives a label before the cap conditions are imposed. A deficit arises when too many points fail the cut for every label available to them. Thus, \emph{lost volume} means volume excluded by this selection, even when the transport is a diffeomorphism. If an inverse slope belongs to several closed cells, any one of their normals can qualify the point; the union still counts it only once.
	
	\subsection{Facet obstructions}
	A facet makes a whole pyramid of image points share one inverse slope direction. The construction below assigns almost all of this pyramid, in a small-tilt limit, to a normal whose half-space rejects it. Central symmetry fixes the reference cap volume at $|W|/2$, so the resulting deficit can be detected with two equal-height boundary points.
	
	\begin{theorem}[Failure in the presence of a facet]\label{rev:thm:facet}
		Let $d\geq 2$, and let $W=-W\subset\R^d$ be a full-dimensional convex body with a facet of positive $\calH^{d-1}$-measure. There exist a bounded smooth strictly convex ellipsoid $\C$, two points $x_1,x_2\in\partial\C$ with opposite outer unit normals $n_1=-n_2$, and boundary values $\psi(x_1)=\psi(x_2)=0$, such that
		\[
		\left|\mathscr B_{\tension,0}^{\psi}\bigl(\{x_1,x_2\},(n_1,n_2)\bigr)\right|
		<\tfrac{1}{2}|W|
		=\min_{n\in\mathbb{S}^{d-1}}|W\cap\{p:p\cdot n>0\}|.
		\]
		The cells
		$\mathcal A_i=\partial_{\{x_1,x_2\}}\psi(x_i)$
		satisfy $\mu_\Phi(\mathcal A_1\cap \mathcal A_2)=0$.
	\end{theorem}
	
	\begin{proof}[Proof of Theorem~\ref{rev:thm:facet}]
		Let $F$ be a facet with outer unit normal $u$, and put
		\[
		h=h_W(u),
		\qquad P_+=\operatorname{co}(0,F),
		\qquad P_-=-P_+,
		\qquad b=|P_+|=|P_-|>0.
		\]
		On the interior of $P_+$,
		\[
		\rho_W(p)=\frac{(p\cdot u)}{h},
		\qquad G_W(p)=\frac{(p\cdot u)}{h^2}\,u,
		\]
		and the corresponding reflected formula holds on $P_-$. Thus the normal ray represents the entire interior of a pyramid, not merely its facet. Conversely, at a differentiability point of $\Phi^*$, if $G_W(p)$ is parallel to $u$, then convex duality places $p$ in $P_+\cup P_-$, up to a null set.
		
		Outside these pyramids we need a sign test that vanishes only on a null set. We may choose a unit vector $t\in u^\perp$ such that
		\begin{equation}\label{rev:eq:generic-tangent}
			\left|\left\{p\in W\setminus(P_+\cup P_-):\, G_W(p)\cdot t=0\right\}\right|=0.
		\end{equation}
		Indeed, for each vector not parallel to $u$, the unit vectors in $u^\perp$ orthogonal to it form a set of spherical measure zero; Fubini gives the claim. For $d=2$, either unit vector in $u^\perp$ has the required property.
		
		We use both tangential orientations: they will capture complementary parts of the image outside the two pyramids in the limit. For $\sigma\in\{-1,1\}$ and $\varepsilon>0$, set
		\[
		v_\sigma=\sigma t+\varepsilon u,
		\qquad n_\sigma
		=\frac{u-2\sigma\varepsilon t}{\sqrt{1+4\varepsilon^2}}.
		\]
		Then
		$n_\sigma\cdot v_\sigma = -\varepsilon/\sqrt{1+4\varepsilon^2}<0$.
		The equal-height cells associated with the points $0,v_\sigma$ are
		\[
		\mathcal A_1=\{q:\,q\cdot v_\sigma\leq 0\},
		\qquad
		\mathcal A_2=\{q:\,q\cdot v_\sigma\geq 0\}.
		\]
		
		For almost every $\varepsilon>0$,
		\begin{equation}\label{rev:eq:null-interface}
			|\{p\in W:G_W(p)\cdot v_\sigma=0\}|=0.
		\end{equation}
		To see this, use \eqref{rev:eq:generic-tangent} outside $P_+\cup P_-$: for almost every fixed $p$, the equation
		$\sigma\,G_W(p)\cdot t+\varepsilon\,G_W(p)\cdot u=0$
		has at most one solution in $\varepsilon$. On $P_+\cup P_-$ it has no positive solution. Another application of Fubini proves~\eqref{rev:eq:null-interface}. By homogeneity, this also gives $\mu_\Phi(\mathcal A_1\cap \mathcal A_2)=0$.
		
		For values of $\varepsilon$ satisfying~\eqref{rev:eq:null-interface}, the captured image agrees up to a null set with
		\[
		\mathscr B_\sigma(\varepsilon)
		=\left\{p\in W:\,\bigl(G_W(p)\cdot v_\sigma\bigr)\bigl(p\cdot n_\sigma\bigr)<0
		\right\}.
		\]
		The labels explain the opposite-sign requirement. A negative first factor assigns $p$ to $\mathcal A_1$, whose normal is $n_\sigma$, and capture requires $p\cdot n_\sigma>0$. A positive first factor assigns it to $\mathcal A_2$, whose normal is $-n_\sigma$, and capture requires $p\cdot n_\sigma<0$. On $P_+$, however, $G_W(p)\cdot v_\sigma=\varepsilon(p\cdot u)/h^2>0$: the point receives the second label, although it lies on the positive side of the cap cut for small $\varepsilon$. The reflected pyramid receives the first label and fails its cut in the same way.
		At each fixed interior point of $P_+$ both factors are positive for all sufficiently small $\varepsilon>0$; on $P_-$ both are negative. Thus these pyramids do not contribute in the limit. Outside $P_+\cup P_-$, the limiting signs of the first factor are opposite for $\sigma=1$ and $\sigma=-1$, whereas the second factor converges to $p\cdot u$.
		The exceptional sets in~\eqref{rev:eq:generic-tangent} and $\{p:p\cdot u=0\}$
		are null. Dominated convergence therefore yields
		\[
		\lim_{\varepsilon\to0^+}|\mathscr B_+(\varepsilon)|
		+ \lim_{\varepsilon\to0^+}|\mathscr B_-(\varepsilon)|
		= |W|-2b.
		\]
		The two limiting tests together capture exactly the volume outside the pyramids. Hence at least one captures at most half of that remaining volume, namely $|W|/2-b$. Fix that sign and choose a sufficiently small $\varepsilon>0$ satisfying~\eqref{rev:eq:null-interface}, so that
		$|\mathscr B_\sigma(\varepsilon)|\leq |W|/2-b/2$.
		
		The volume deficit has now been obtained from the two cells and their labels. The condition $n_\sigma\cdot v_\sigma<0$ allows us to realize exactly these data on a bounded smooth strictly convex obstacle. Write $v=v_\sigma$ and $n=n_\sigma$. For any $\beta>0$, the matrix
		\[
		\mathsf B=-\frac{(n\otimes n)}{n\cdot v}+\beta\left(I_d-\frac{(v\otimes v)}{|v|^2}\right)
		\]
		is symmetric positive definite and satisfies $\mathsf Bv=-n$. The ellipsoid
		\[
		\C=\left\{x:\,\bigl(x-\frac v2\bigr)\cdot \mathsf B\bigl(x-\frac v2\bigr)\leq \frac14\,v\cdot\mathsf Bv\right\}
		\]
		has outer unit normals $n$ and $-n$ at $0$ and $v$, respectively. The cells and captured image are unchanged. Finally, central symmetry gives $|W\cap\{p:p\cdot\omega>0\}|=|W|/2$ for every $\omega\in\mathbb{S}^{d-1}$.
	\end{proof}
	
	The construction applies to every origin-symmetric polytope and to symmetric nonpolyhedral bodies with a facet. The obstacle is smooth and strictly convex, and the transported cell interface has zero volume. The deficit is therefore already present with two unambiguous labels.
	
	The positive tilt is essential to this interpretation. Formally setting $\varepsilon=0$ puts both normal rays on the common boundary of the closed cells. Points in the facet pyramids can then use either label and pass the appropriate cap cut. This formal zero-tilt image need not have the limiting captured volume. The argument instead takes limits of the volumes at positive tilt and then fixes a positive value for which the deficit is strict.
	
	\paragraph{A square model.}
	For $W=[-1,1]^2$, the inverse map is, away from the diagonals,
	\[
	G_W(p)=
	\begin{cases}
		(p_1,0),& |p_1|>|p_2|,\\
		(0,p_2),& |p_2|>|p_1|.
	\end{cases}
	\]
	Take $u=e_2$, $t=e_1$, $\sigma=1$, and $0<\varepsilon<1/2$. Then $v=(1,\varepsilon)$ and $n=(-2\varepsilon,1)/\sqrt{1+4\varepsilon^2}$. The upper and lower triangles $|p_2|>|p_1|$ are entirely rejected. On the right triangle, the inverse slope belongs to $\mathcal A_2$, so the accepted portion is
	$0<p_1<1$, $-p_1<p_2<2\varepsilon p_1$.
	The left portion is its reflection. Consequently,
	\[
	|\mathscr B|
	=2\int_0^1(1+2\varepsilon)r\,dr
	=1+2\varepsilon<2=|W|/2.
	\]
	The preceding ellipsoid construction realizes these same points and normals. This example exhibits a strict loss of area without any appeal to an asymptotic expansion of the transport.
	
	Without central symmetry, the two orientations still divide the volume outside the facet pyramids in the limit. What changes is the comparison threshold: a directional half-cap need no longer have volume $|W|/2$. The construction therefore yields a failure whenever half of the remaining volume is smaller than even the smallest reference cap. This gives the following sufficient criterion.
	
	\begin{theorem}[Facet obstruction without symmetry]\label{thm:6.1b}
		Let \(W\) be a full-dimensional convex body with \(0\in\operatorname{int}W\). For \(\xi\in\mathbb{S}^{d-1}\), denote its exposed face by \(F_W(\xi):=\{p\in W:p\cdot\xi=h_W(\xi)\}\). Let \(F\) be a facet with outer normal \(u\), and set \(P_u:=\operatorname{co}(0,F)\) and \(P_{-u}:=\operatorname{co}(0,F_W(-u))\), where \(|P_{-u}|=0\) if \(F_W(-u)\) is not a facet. If
		\[
		\big(|W|-|P_u|-|P_{-u}|\big)/2
		< m^{\min}_{\tension,0},
		\]
		then $W$ fails the transported $0$-ABP property, through the two-point configuration of Theorem~\ref{rev:thm:facet} on a smooth strictly convex ellipsoid. For $W=-W$ the condition always holds, which recovers Theorem~\ref{rev:thm:facet}.
	\end{theorem}
	For a triangle with centroid at the origin, the pyramid over an edge has $|P_u|=|W|/3$. The opposite exposed face is a vertex, so its pyramid is a segment and $|P_{-u}|=0$. Removing the first pyramid leaves two-thirds of the area, which the two limiting orientations divide between them. One limiting captured volume is therefore at most one-third, whereas Gr\"unbaum's theorem~\cite{grunbaum1960partitions} gives $m^{\min}_{\tension,0}=4|W|/9$. Thus the criterion reads $|W|/3<4|W|/9$: the strict gap persists for a sufficiently small positive tilt.
	
	More generally, for a $d$-simplex with centroid at the origin,
	\[
	|P_u|=\frac{|W|}{d+1},\qquad |P_{-u}|=0,\qquad
	m^{\min}_{\tension,0}=\left(\frac d{d+1}\right)^d|W|.
	\]
	The first volume uses the pyramid between a facet and the centroid. The reference cap is a different region: a centroidal cut parallel to that facet leaves, on the opposite-vertex side, a similar simplex with linear scale $d/(d+1)$. Its volume attains Gr\"unbaum's lower bound and gives the last equality. The criterion is $1/2<(d/(d+1))^{d-1}$, which holds exactly for $2\leq d\leq4$. Indeed, $(d-1)\log(1+1/d)$ is strictly increasing for $d\geq2$, and the inequality reverses between $d=4$ and $d=5$. The change of sign limits this sufficient criterion; it gives no conclusion about validity of the transported estimate for these simplices in higher dimensions.
	
	\begin{proof}[Proof of Theorem~\ref{thm:6.1b}]
		In the proof of Theorem~\ref{rev:thm:facet}, central symmetry enters only through $P_-=-P_+$ (here replaced by $P_{-u}$, a null set when $F_W(-u)$ is not a facet) and the final identity $m^{\min}_{\tension,0}=|W|/2$:
		\begin{enumerate}[label=$\bullet$,itemsep=0em]
			\item at differentiability points, $G_W(p)\parallel u$ forces $p\in P_u\cup P_{-u}$ by convex duality;
			\item~\eqref{rev:eq:generic-tangent} and~\eqref{rev:eq:null-interface} follow by the same Fubini argument;
			\item on $P_u$ and $P_{-u}$ both factors keep a fixed common sign;
			\item outside them the first factor has opposite limits for $\sigma=\pm1$.
		\end{enumerate}
		Hence $\lim|\mathscr B_+|+\lim|\mathscr B_-|=|W|-|P_u|-|P_{-u}|$, one sign gives at most half of it, and the ellipsoid realization is unchanged.
	\end{proof}
	
	\subsection{Rigidity of the transported estimate}\label{sec:6.3}
	For a smooth strictly convex potential, the inverse map no longer collapses a pyramid onto a ray. The obstruction must instead be read from the distribution of image volume among the cells. Evenness guarantees that every central cap has half the total volume, but it does not guarantee half-volume capture after points are assigned to cells.
	
	The distinction becomes local in the proof below. For $0\neq z\in n^\perp$, central symmetry pairs the fiber $z+\R\,n$ with the different fiber $-z+\R\,n$; it need not balance the two branches within either one. Thin slabs with opposite normals test this stronger balance. Their tangential displacement chooses a projection direction, and their height difference moves the threshold of that projection. Requiring all such tests to succeed forces the transport to be linear.
	\begin{theorem}[Two-point rigidity of transported
		boundary images]\label{thm:two-point-rigidity}
		Let $d\geq 2$, let $\tension$ be even, and suppose that
		\[
		\Phi:=\frac12\tension^2\in C^3(\R^d\setminus\{0\}),
		\qquad D^2\Phi(q)>0 \quad(q\neq 0).
		\]
		The following statements are equivalent.
		\begin{enumerate}[itemsep=0em]
			\item[(i)] $\wulff$ satisfies the transported $0$-ABP property of Definition~\ref{rev:def:abp}.
			
			\item[(ii)] The transported zero-adhesion image estimate holds for every two-point configuration of the following form:
			\[
			\begin{aligned}  
				& \C_\delta = \{x\in\R^d:\,-\delta\leq x\cdot n\leq0\},\\
				& F=\{0,a-\delta n\},
				\qquad a\in n^\perp\setminus\{0\},
			\end{aligned}
			\]
			with supporting normals $n$ and $-n$, respectively, and arbitrary boundary values
			\[
			\psi(0)=0,
			\qquad \psi(a-\delta n)=s,
			\quad s\in\R.
			\]
			Here $n\in\mathbb{S}^{d-1}$ and $\delta>0$ are arbitrary.
			
			\item[(iii)] There exists a symmetric positive definite matrix $\M$ such that
			\[
			\Phi(q)=\frac12 q\cdot \M q
			\qquad\text{for every }q\in\R^d.
			\]
		\end{enumerate}
		In particular, validity of the transported zero-adhesion estimate on this restricted family of two-point slab configurations already characterizes centered ellipsoids.
	\end{theorem}
	
	The proof starts with a test and its reflection. In the zero-thickness limit their selected images partition the Wulff body, so two lower bounds by half its volume force equality in both. Varying the height $s$ and the tangential direction then forces balance on each fiber, with volume measured by the Jacobian density. A local expansion shows that this requirement constrains all tangential third derivatives.
	Arbitrary heights $s$ are used to recover every tangential half-space projection; the classification from equal-height tests alone remains open. Figure~\ref{fig:slab-fiber-rigidity} illustrates the slab test and the two weighted branches of a fiber.
	
	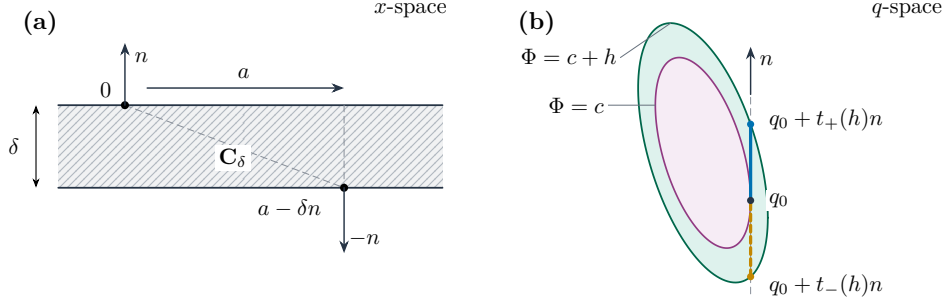
\begin{figure}[!htbp]
		\centering
		\scalebox{\CapFigureScale}{%
			\begin{minipage}{\textwidth}
				\centering
				\begin{tikzpicture}[cap panel]
					\CapPanel{a}
					\fill[cap obstacle] (.75,2.15) rectangle (6.80,3.45);
					\draw[cap edge] (.75,2.15) -- (6.80,2.15);
					\draw[cap edge] (.75,3.45) -- (6.80,3.45);
					\draw[cap guide] (1.80,3.45) -- (5.25,2.15);
					\draw[cap guide] (5.25,3.45) -- (5.25,2.15);
					\fill (1.80,3.45) circle (1.8pt) (5.25,2.15) circle (1.8pt);
					\draw[cap normal] (1.80,3.45) -- (1.80,4.43);
					\draw[cap normal] (5.25,2.15) -- (5.25,1.12);
					\node[above left] at (1.72,3.45) {$0$};
					\node at (2.06,4.20) {$n$};
					\node at (5.57,1.35) {$-n$};
					\node[anchor=east,cap label] at (4.95,1.83) {$a-\delta n$};
					\draw[cap normal] (2.15,3.72) -- (5.25,3.72);
					\node[above] at (3.67,3.74) {$a$};
					\draw[{Stealth[length=1.5mm]}-{Stealth[length=1.5mm]},line width=.5pt]
					(.40,2.15) -- (.40,3.45);
					\node[anchor=east] at (.27,2.80) {$\delta$};
					\node[cap label] at (3.50,2.62) {$\C_\delta$};
					\node[anchor=east,font=\small] at (7.00,4.95) {$x$-space};
				\end{tikzpicture}\hfill
				\begin{tikzpicture}[cap panel]
					\CapPanel{b}
					\begin{scope}[shift={(3.10,2.70)},scale=3]
						\path[fill=capGreen!13]
						plot[domain=0:360,samples=140]
						({sqrt(.3475/3)*cos(\x)},
						{sqrt(.3475)*sin(\x)-sqrt(.3475/3)*cos(\x)}) -- cycle;
						\draw[cap edge,draw=capGreen!65!black]
						plot[domain=0:360,samples=140]
						({sqrt(.3475/3)*cos(\x)},
						{sqrt(.3475)*sin(\x)-sqrt(.3475/3)*cos(\x)});
						\path[fill=capPurple!10]
						plot[domain=0:360,samples=140]
						({.25*cos(\x)},{sqrt(.1875)*sin(\x)-.25*cos(\x)}) -- cycle;
						\draw[draw=capPurple!85!black,line width=.75pt]
						plot[domain=0:360,samples=140]
						({.25*cos(\x)},{sqrt(.1875)*sin(\x)-.25*cos(\x)});
					\end{scope}
					\draw[cap guide] (3.85,.48) -- (3.85,3.60);
					\draw[draw=capBlue,line width=1.45pt] (3.85,1.95) -- (3.85,3.15);
					\draw[draw=capAmber!85!black,line width=1.45pt,densely dashed] (3.85,.75) -- (3.85,1.95);
					\fill[capAmber!85!black] (3.85,.75) circle (1.6pt);
					\fill[capInk] (3.85,1.95) circle (1.6pt);
					\fill[capBlue] (3.85,3.15) circle (1.6pt);
					\node[anchor=west,cap label] at (4.08,3.23) {$q_0+t_+(h)n$};
					\node[anchor=west,cap label] at (4.08,1.95) {$q_0$};
					\node[anchor=west,cap label] at (4.08,.62) {$q_0+t_-(h)n$};
					\draw[cap normal] (3.85,3.62) -- (3.85,4.37);
					\node at (4.10,4.16) {$n$};
					\draw[cap leader] (1.80,4.20) -- (2.59,4.74);
					\node[anchor=east,cap label] at (1.80,4.20) {$\Phi=c+h$};
					\draw[cap leader] (1.60,3.45) -- (2.35,3.45);
					\node[anchor=east,cap label] at (1.60,3.45) {$\Phi=c$};
					\node[anchor=east,font=\small] at (7.00,4.95) {$q$-space};
				\end{tikzpicture}
			\end{minipage}%
		}
		\caption{Two-point rigidity (Theorem~\ref{thm:two-point-rigidity}).
			(a) Opposite supporting normals on a slab, with arbitrary boundary values $0,s$. (b) The fiber $q_0+tn$ is tangent to level $c=\Phi(q_0)$ and meets level $c+h$ at $t_\pm(h)$. Hypothesis~(ii) yields equal $j_\Phi\,dt$-mass on its blue and amber branches, with $j_\Phi=\det D^2\Phi$. The ellipsoidal model illustrates the geometry.}
		\label{fig:slab-fiber-rigidity}
	\end{figure}
	
	For comparison, if $\Phi(q)=q\cdot\M q/2$ and $D\Phi(q_0)\cdot n=0$, then
	\[
	\Phi(q_0+tn)=\Phi(q_0)+\frac{t^2}{2}\,n\cdot\M n,
	\qquad j_\Phi=\det D^2\Phi=\det\M.
	\]
	The two branches in Figure~\ref{fig:slab-fiber-rigidity} have equal lengths and constant density, so their weighted masses agree. For a general potential, both the branch lengths and the density may vary. The proof tests whether these two changes can compensate in every direction.
	
	\begin{proof}[Proof of Theorem~\ref{thm:two-point-rigidity}]
		Proposition~\ref{v12:prop:ellipsoid-transport} gives (iii)$\Rightarrow$(i), and (i)$\Rightarrow$(ii) is immediate. We prove (ii)$\Rightarrow$(iii). Put $\mathsf T=D\Phi$, and $j_\Phi=\det D^2\Phi$ away from zero. Strict convexity, two-homogeneity, and coercivity make $\mathsf T$ a diffeomorphism off the origin, with $\mathsf T(Q_\tension)=\operatorname{int}\wulff$ by Legendre duality. Thus change of variables uses the density $j_\Phi$. Evenness gives $m^{\min}_{\tension,0}=|\wulff|/2$.
		
		\emph{Step 1: signed balance.} Fix $n\in\mathbb S^{d-1}$, $a\in n^\perp\setminus\{0\}$, and $s\in\R$. On the slab $\C_\delta$, the data in (ii) give cells $\{q:q\cdot(a-\delta n)\leq s\}$ and $\{q:q\cdot(a-\delta n)\geq s\}$. Letting $\delta\to0^+$ gives
		\[
		I_1:=\int_{Q_\tension\cap\{q\cdot a\leq s\}}\chi_{\{\mathsf T\cdot n>0\}}\,j_\Phi\,dq
		+\int_{Q_\tension\cap\{q\cdot a\geq s\}}\chi_{\{\mathsf T\cdot n<0\}}\,j_\Phi\,dq
		\geq\frac{|\wulff|}{2}.
		\]
		Dominated convergence applies: $\{q\cdot a=s\}$ is null and $\{\mathsf T\cdot n=0\}$ is $j_\Phi\,dq$-null, being the inverse image of a hyperplane under the diffeomorphism.
		
		The mirrored slab, obtained by replacing $n$ with $-n$, gives $I_2\geq|\wulff|/2$ with the signs exchanged. Since $I_1+I_2=|\wulff|$, equality holds in both. The total positive and negative $\mathsf T\cdot n$ masses are each $|\wulff|/2$, so subtracting them gives
		\begin{equation}\label{eq:rigidity-signed-identity}
			\int_{Q_\tension\cap\{q\cdot a\leq s\}}\operatorname{sgn}(\mathsf T(q)\cdot n)j_\Phi(q)\,dq=0
			\quad(a\perp n,\ a\neq0,\ s\in\R).
		\end{equation}
		
		The integral in~\eqref{eq:rigidity-signed-identity} is precisely $I_1-|\wulff|/2=(I_1-I_2)/2$. A nonzero value would therefore make one reflected test fall strictly below the required volume. Allowing all heights $s$ tests every threshold, rather than just half-spaces through the origin.
		
		\emph{Step 2: fiber balance.} The half-space identities in Step~1 contain more information than a single global balance: the tangential direction and the offset are both arbitrary. We use this freedom to recover balance on each line parallel to $n$. 
		Let $\pi_n(q)=q-(q\cdot n)n$ be the orthogonal projection onto $n^\perp$, and define a finite signed measure $\mu$ on $n^\perp$ by
		\[
		\mu(A):=
		\int_{Q_\tension\cap\pi_n^{-1}(A)}
		\operatorname{sgn}(\mathsf T(q)\cdot n)\,j_\Phi(q)\,dq
		\]
		for every Borel set $A\subset n^\perp$. For each unit vector $a\in n^\perp$, let $\mu_a:=(z\mapsto a\cdot z)_\#\mu$ be its one-dimensional pushforward. Since $a\cdot\pi_n(q)=a\cdot q$,~\eqref{eq:rigidity-signed-identity} gives
		\[
		\mu_a((-\infty,s])
		=
		\int_{Q_\tension\cap\{q\cdot a\le s\}}
		\operatorname{sgn}(\mathsf T(q)\cdot n)\,j_\Phi(q)\,dq
		=0
		\quad\text{ for }\;s\in\R.
		\]
		Taking differences shows that $\mu_a$ vanishes on every bounded half-open interval; hence $\mu_a=0$ as a signed Borel measure on $\R$. These one-dimensional identities determine $\mu$. Indeed, for any $0\neq\xi\in n^\perp$, choose $a=\xi/|\xi|$. By the definition of pushforward,
		\[
		\widehat\mu(\xi)
		:= \int_{n^\perp}e^{-i\xi\cdot z}\,d\mu(z)
		= \int_{\R}e^{-i|\xi|t}\,d\mu_a(t)
		=0.
		\]
		Also $\widehat\mu(0)=\mu(n^\perp)=\mu_a(\R)=0$. Thus $\widehat\mu$ vanishes identically, and uniqueness of the Fourier transform for finite signed measures yields $\mu=0$. Equivalently, Fourier inversion for smooth compactly supported test functions, followed by Fubini, gives $\int_{n^\perp}\eta\,d\mu=0$ for every $\eta\in C_c^\infty(n^\perp)$, which implies the same conclusion.
		
		We now express this vanishing projection as a fiber integral. For $z$ in the interior of the projection of $Q_\tension$, write $\{t:z+tn\in Q_\tension\}=(t_-(z),t_+(z))$. Fubini gives
		\[
		I(z):=\int_{t_-(z)}^{t_+(z)}
		\operatorname{sgn}(\mathsf T(z+tn)\cdot n)\,j_\Phi(z+tn)\,dt=0
		\]
		for almost every $z$. This holds at every such $z\neq0$: the convex and concave endpoint functions are continuous, nearby fibers avoid the origin, and $\mathsf T,j_\Phi$ are continuous with $j_\Phi$ locally bounded there. Moreover, $t\mapsto \mathsf T(z+tn)\cdot n$ is strictly increasing, so it has at most one zero. Extend the integrand by zero outside the fiber and apply dominated convergence to obtain continuity of $I$.
		
		The local calculation in the next step requires the same balance on other sublevels. This follows from homogeneity: $\{\Phi<c\}=\sqrt{2c}\,Q_\tension$, $\mathsf T$ is one-homogeneous, and $j_\Phi$ is zero-homogeneous. Scaling transfers the balance to every positive sublevel set on each line avoiding zero.
		
		\emph{Step 3: an infinitesimal identity.} On a fiber tangent to a level set, a slightly higher sublevel cuts out an interval with one branch on each side of the point of tangency. Their lengths may differ, but Step~2 requires their Jacobian-weighted masses to agree. We compare the first terms that can distinguish the two branches. Fix $q_0\neq0$ and a unit $n$ with $D\Phi(q_0)\cdot n=0$. Euler's identity implies $q_0\notin\R n$. Set
		\[
		a_0=D^2\Phi(q_0)[n,n]>0,\quad B_3=D^3\Phi(q_0)[n,n,n],
		\quad j_0=j_\Phi(q_0),\quad j_1=Dj_\Phi(q_0)[n].
		\]
		For the two solutions $t_-(h)<0<t_+(h)$ of $\Phi(q_0+tn)=\Phi(q_0)+h$, Taylor expansion gives, with $r=(2h/a_0)^{1/2}$,
		\[
		\Phi(q_0+tn)=\Phi(q_0)+\frac{a_0}{2}t^2+\frac{B_3}{6}t^3+o(|t|^3),
		\qquad t_\pm(h)=\pm r-\frac{B_3}{6a_0}r^2+o(r^2).
		\]
		Writing $t_\pm=\pm r+c_2r^2+o(r^2)$, the latter follows by substitution; the coefficient of $r^3$ gives $a_0c_2+B_3/6=0$. Step~2 equates the $j_\Phi\,dt$ masses on $(t_-,0)$ and $(0,t_+)$. Since $j_\Phi(q_0+tn)=j_0+j_1t+o(|t|)$,
		\[
		0=j_0(t_++t_-)+\frac{j_1}{2}(t_+^2+t_-^2)+o(r^2)
		=\left(j_1-\frac{j_0B_3}{3a_0}\right)r^2+o(r^2).
		\]
		The term involving $B_3$ measures the unequal branch lengths, while $j_1$ measures the change of density along the fiber. They can compensate on a single fiber; the displayed equality expresses exactly the compensation required. Divide by $r^2$ and let $h\to0^+$. Homogeneity in the direction gives
		\begin{equation}\label{eq:rigidity-cubic-identity}
			D^3\Phi(q)[v,v,v]
			=3D^2\Phi(q)[v,v]D(\log\det D^2\Phi)(q)[v]
		\end{equation}
		whenever $q\neq0$ and $D\Phi(q)\cdot v=0$.
		
		\emph{Step 4: vanishing of the cubic tensor.} The identity from Step~3 controls third derivatives in directions tangent to a level set. Homogeneity supplies the radial components, and Jacobi's formula relates the remaining components to their own trace. At fixed $q\neq0$, write $H=D^2\Phi(q)$, $\mathcal T=D^3\Phi(q)$, $\ell=D\log j_\Phi(q)$, and $V=D\Phi(q)^\perp$. Homogeneity gives $Hq=D\Phi(q)$ and $\mathcal T(q,\cdot,\cdot)=0$, so $V$ is the $H$-orthogonal complement of $q$.  The right-hand side of the identity below defines a symmetric trilinear form on $V$, with diagonal value $3H(v,v)\,\ell(v)$. By~\eqref{eq:rigidity-cubic-identity}, this agrees with $\mathcal T(v,v,v)$ for every $v\in V$. Since a symmetric trilinear form is determined by its diagonal values, polarization gives 
		\begin{equation}\label{eq:rigidity-polarized-identity}
			\mathcal T(v,w,z) = H(v,w)\,\ell(z)+H(v,z)\,\ell(w)+H(w,z)\,\ell(v),
		\end{equation}
		for all $v,w,z\in V$.
		
		Choose an $H$-orthonormal basis $e_1,\ldots,e_{d-1}$ of $V$, completed by $e_d=q/(q\cdot Hq)^{1/2}$. Since $e_d$ is parallel to $q$ and $\mathcal T(q,\cdot,\cdot)=0$, we have $\mathcal T(v,e_d,e_d)=0$, so the radial term in the trace vanishes. Moreover, for every $v\in V$, $H$-orthonormality gives
		\[
		v=\sum_{i=1}^{d-1}H(v,e_i)\,e_i,
		\qquad
		\sum_{i=1}^{d-1}H(v,e_i)\,\ell(e_i)=\ell(v).
		\]
		Jacobi's formula and~\eqref{eq:rigidity-polarized-identity}
		therefore give
		\[
		\begin{aligned}
			\ell(v)&=\operatorname{tr}\bigl(H^{-1}\mathcal T(v,\cdot,\cdot)\bigr)
			=\sum_{i=1}^{d-1}\mathcal T(v,e_i,e_i)\\
			&=(d-1)\,\ell(v)+2\sum_{i=1}^{d-1}H(v,e_i)\,\ell(e_i)
			=(d+1)\,\ell(v).
		\end{aligned}
		\]
		The trace calculation gives both $\ell(v)$ and $(d+1)\ell(v)$ for the same quantity. Thus the compensation in every tangential direction forces $\ell|_V=0$ and $\mathcal T|_{V^3}=0$. Since $\R^d=V\oplus\R q$ and every component involving $q$ vanishes, $D^3\Phi(q)=0$. Connectedness of $\R^d\setminus\{0\}$ for $d\geq2$ makes the Hessian a constant positive definite matrix $\M$. Two-homogeneity then gives $\Phi(q)=(q\cdot \M q)/2$.
	\end{proof}
	This conclusion has a classical counterpart in Deicke's theorem. In Brickell's $C^4$-formulation~\cite{brickell1965deicke}, constancy of the determinant of $D^2(\tension^2/2)$ forces the Hessian itself to be constant for a positive one-homogeneous $\tension$. Here the hypothesis comes instead from two-point transported boundary-image tests. They yield the full cubic identity, so the contraction above proves $D^3\Phi=0$ directly under the stated $C^3$-assumption. The characterization therefore follows without appealing to the additional regularity in that formulation of Deicke's theorem.
	
	\paragraph{Bounded ellipsoidal tests.}
	The argument also explains how a nonquadratic potential produces a counterexample. If every fiber had the required balance, the proof would force $D^3\Phi=0$. Hence some fiber is unbalanced; continuity and the projection argument detect this through a tangential half-space for which the integral appearing in~\eqref{eq:rigidity-signed-identity} is nonzero. One of the two reflected limiting tests then has a strict deficit. Convergence as $\delta\to0^+$ preserves that strict inequality for a sufficiently small positive slab thickness.
	
	The resulting points, normals, and heights can be kept exactly fixed while the slab is replaced by a bounded smooth strictly convex ellipsoid: $n\cdot(a-\delta n)<0$ allows the realization in Theorem~\ref{rev:thm:facet}. The cells and captured image are unchanged. Conversely, quadratic potentials satisfy these tests by Proposition~\ref{v12:prop:ellipsoid-transport}. Thus every nonquadratic potential in the stated class has a strict two-point deficit on a bounded ellipsoidal obstacle.
	
	\paragraph{Hausdorff persistence of finite deficits.}
	We now fix one of the preceding configurations with a strict deficit and perturb the Wulff body. The order matters: the boundary points, normals, and heights are fixed before any smoothing. Their slope cells therefore remain fixed. Although the inverse transport and the cap cuts vary, closedness of the cells bounds the limsup of the nearby captured volumes by the captured volume for $W$. Continuity of the minimum reference cap then preserves the gap.
	\begin{lemma}[Upper semicontinuity of finite boundary images]\label{lem:finite-image-usc}
		Let $\{W_j\}_{j=1}^{\infty},W\subset\R^d$ be convex bodies containing the origin in their interiors, with $d_H(W_j,W)\to0$, where $d_H$ is Hausdorff distance. Set $\tension_j=h_{W_j}$, $\tension=h_W$, $\Phi_j=\tension_j^2/2$, and $\Phi=\tension^2/2$. Fix $N$ distinct points $F=\{x_1,\ldots,x_N\}$ on a convex obstacle, supporting unit normals $n_i$, and values $\psi_i$. Their closed cells are
		\[
		\mathcal A_i=\{q:\, q\cdot(x_k-x_i)\leq\psi_k-\psi_i,\;\;\text{for every }k\}.
		\]
		Suppose $\lambda_j\to\lambda$ as $j\to\infty$, where $-1<\lambda<\lambda_*(\tension)$ and $-1<\lambda_j<\lambda_*(\tension_j)$ for all sufficiently large $j$. Define
		\[
		\mathscr B_j=\bigcup_{i=1}^N
		\left(\partial\Phi_j(\mathcal A_i)\cap W_j\cap
		\{p:\, p\cdot n_i>\lambda_j\,\tension_j(-n_i)\}\right),
		\]
		and define $\mathscr B$ analogously with $W,\Phi,\lambda$. Then
		\begin{equation}\label{eq:finite-image-usc}
			\limsup_{j\to\infty}|\mathscr B_j|\leq|\mathscr B|,
			\qquad
			m_{\tension_j,\lambda_j}^{\min}\xrightarrow[j\to\infty]{} m_{\tension,\lambda}^{\min}.
		\end{equation}
		No smoothness, strict convexity, or symmetry is assumed.
	\end{lemma}
	\begin{proof}[Proof of Lemma~\ref{lem:finite-image-usc}]
		We work in image space and use the inverse maps to test cell membership. This keeps the comparison in a fixed ambient space even when the forward subdifferentials are multivalued. Hausdorff convergence and positivity give $rB_1\subset W_j,W\subset RB_1$ for fixed $0<r<R$ and all large $j$. The gauges converge locally uniformly, hence so do the convex dual potentials
		$\Phi_j^*=\rho_{W_j}^2/2\longrightarrow\rho_W^2/2=\Phi^*$, as $j\to\infty$.
		Outside a common null set all gradients $G_j=D\Phi_j^*$ and $G=D\Phi^*$ exist, and $G_j(p)\to G(p)$ as $j\to\infty$. Indeed, local uniform convergence bounds the subgradients locally, and every cluster point belongs to the singleton $\partial\Phi^*(p)$.
		
		By~\eqref{v12:eq:inverse-image}--\eqref{v12:eq:labeled-graph}, membership in $\mathscr B_j$ is, almost everywhere, equivalent to $p\in W_j$ and, for some $i$, $G_j(p)\in \mathcal A_i$ and $p\cdot n_i>\lambda_j\,\tension_j(-n_i)$. Exclude also $\partial W$ and the finitely many hyperplanes $p\cdot n_i=\lambda\,\tension(-n_i)$. If a remaining point lies in $\mathscr B_j$ infinitely often, a subsequence uses the same label $i$. Closedness of $\mathcal A_i$ gives $G(p)\in \mathcal A_i$, while Hausdorff convergence gives $p\in W$ and the limiting cap inequality is strict. Thus
		\[
		\limsup_{j\to\infty}\chi_{\mathscr B_j}(p)\leq\chi_{\mathscr B}(p)
		\quad\text{for almost every }p.
		\]
		All images lie in $RB_1$. Applying Fatou's lemma to their complements in this ball proves the first assertion.
		
		If $n_j\to n$ as $j\to\infty$, the indicators of $K_{\tension_j,\lambda_j}(n_j)$ converge almost everywhere to that of $K_{\tension,\lambda}(n)$, outside the limiting Wulff boundary and cutting hyperplane. Dominated convergence gives convergence of the cap volumes. Applying this both to convergent subsequences of minimizing directions and to a fixed limiting minimizer proves convergence of their minima.
	\end{proof}
	
	A limiting inverse slope may reach a cell boundary and acquire an additional label, even if the nearby slopes did not have that label. This explains why the limiting image may gain volume and why the conclusion is upper semicontinuity. If $G(p)$ avoids all cell boundaries for almost every $p\in W$, the same argument gives full continuity. A strict deficit needs only the one-sided bound: nearby captured volumes cannot rise by enough to close a fixed positive gap while the reference cap volume converges.
	\begin{corollary}[Hausdorff persistence of strict finite deficits]\label{cor:strict-deficit-persistence}
		Let $W$ be a convex body containing the origin in its interior, let $-1<\lambda<\lambda_*(h_W)$, and suppose that a fixed finite configuration has
		\[
		\eta:=m_{h_W,\lambda}^{\min}-|\mathscr B|>0.
		\]
		There exists $\varepsilon>0$ such that every convex body $\widetilde W$ containing the origin in its interior and every $\widetilde\lambda\in(-1,\lambda_*(h_{\widetilde W}))$ with $d_H(\widetilde W,W)+|\widetilde\lambda-\lambda|<\varepsilon$
		have, for the same configuration,
		\[
		m_{h_{\widetilde W},\widetilde\lambda}^{\min}-|\widetilde{\mathscr B}|\geq\eta/2.
		\]
		The perturbed bodies may be nonsmooth or non-even.
		In particular, every even nonellipsoidal anisotropy in the class of Theorem~\ref{thm:two-point-rigidity} fails the transported $\lambda$-ABP property for every $|\lambda|<\eta_\tension$, for some $\eta_\tension>0$.
	\end{corollary}
	\begin{proof}[Proof of Corollary~\ref{cor:strict-deficit-persistence}]
		Lemma~\ref{lem:finite-image-usc} gives a lower limit at least $\eta$ for the deficits along every convergent sequence. More explicitly, eventually the reference volume is at least $m_{h_W,\lambda}^{\min}-\eta/4$, and the captured volume is at most $|\mathscr B|+\eta/4$. Their difference is therefore at least $\eta/2$. A contradiction argument yields the stated neighborhood. The last assertion follows from the strict finite deficit supplied at zero adhesion by Theorem~\ref{thm:two-point-rigidity}, keeping the configuration fixed and varying only $\lambda$.
	\end{proof}
	
	For the square model above, fix $\varepsilon=1/4$. Its captured area is $3/2$, whereas its minimum half-cap area is $2$, so $\eta=1/2$. Every sufficiently close smooth Wulff body has, for exactly the same points, normals, and heights, a deficit of at least $1/4$. The corners and facets may disappear under smoothing; the adverse cell assignment still removes enough volume.
	
	\paragraph{Approximation and its obstructions.}\label{rev:sec:approximation}
	A natural route to nonsmooth anisotropies would be to prove the transported estimate for smooth approximants and then pass to the limit in the capillary inequality. The persistence result identifies a restriction on this route: a strict finite deficit survives ordinary smoothing, so an approximating sequence cannot also satisfy the unrestricted transported property.
	
	To state this distinction precisely, for $-1<\lambda<\lambda_*(\tension)$, a \emph{smooth transported-ABP approximation} is a sequence of positive convex one-homogeneous tensions $\tension_j$, smooth away from zero, converging uniformly to $\tension$ on $\mathbb S^{d-1}$, with $D^2\Phi_j>0$, $\Phi_j=\tension_j^2/2$, $\lambda<\lambda_*(\tension_j)$, and each $W_{\tension_j}$ satisfying Definition~\ref{rev:def:abp}.
	
	If such an approximation exists, Proposition~\ref{rev:thm:abp} passes to the limit for a fixed smooth competitor outside a $C^1$ convex obstacle, provided the competitor is $\Phi_j$-regular for every $j$. The finite free and wetted boundary measures give energy convergence, and Lemma~\ref{lem:finite-image-usc} gives convergence of the minimum cap volumes. The obstruction is therefore the availability of approximants with the required geometric property. The next corollary excludes them in the indicated classes.
	
	\begin{corollary}[Obstructions to transported-ABP approximation]\label{cor:6.7}\label{rev:cor:approx-obstruction}
		\leavevmode\begin{enumerate}[label=(\roman*),itemsep=0.2em,leftmargin=*]
			\item An even nonellipsoidal surface tension admits no smooth transported-ABP approximation at zero adhesion by even surface tensions in the class of Theorem~\ref{thm:two-point-rigidity}. If the limiting tension itself belongs to that class, the same conclusion holds for all $|\lambda|<\eta_\tension$, for some $\eta_\tension>0$, even allowing non-even approximants.
			\item If $W=-W$ has a facet, there exists $\varepsilon_W>0$ such that $h_W$ admits no smooth transported-ABP approximation at any $|\lambda|<\varepsilon_W$, even allowing non-even approximants. This holds already for finite tests and for either the minimum-direction constant or a fixed reference direction.
		\end{enumerate}
	\end{corollary}
	\begin{proof}[Proof of Corollary~\ref{rev:cor:approx-obstruction}]
		For (i) at zero adhesion, Theorem~\ref{thm:two-point-rigidity} makes every even approximant ellipsoidal. Their squared support functions are positive definite quadratic forms; uniform convergence and positivity of the limit preserve this property. For a regular even nonellipsoidal limit, the theorem supplies a strict two-point deficit. Corollary~\ref{cor:strict-deficit-persistence} excludes nearby approximants, even non-even ones, for all sufficiently small $|\lambda|$.
		
		For (ii), apply the same persistence result to the strict deficit supplied by Theorem~\ref{rev:thm:facet}. This uses a fixed negative test, and consequently does not require the approximants to be even. A fixed reference cap has volume at least the minimum directional cap volume, so it cannot remove the deficit.
	\end{proof}
	
	\subsection{Smooth counterexamples}\label{subsec:smooth-counterexamples}
	Rigidity gives a negative test at zero adhesion for every nonquadratic potential in the stated even class, and persistence extends that test to a neighborhood of zero. The explicit construction below has a different purpose: it produces failures at any prescribed adhesion level, arbitrarily close to the Euclidean norm and with uniformly controlled ellipticity.
	
	Proposition~\ref{prop:quartic-obstruction} in Appendix~\ref{app:perturbation} constructs the even support functions
	\[
	h_\varepsilon(u)=1+\varepsilon(u_1^4-6u_1^2u_2^2+u_2^4)
	\]
	of smooth uniformly convex bodies $W_\varepsilon$ in every dimension. 
	Let $\varphi_\varepsilon$ denote the positively one-homogeneous extension of $h_\varepsilon$ to $\mathbb R^d$, with $\varphi_\varepsilon(0)=0$, and set $\Phi_\varepsilon=\varphi_\varepsilon^2/2$. For each compact $\mathcal I\subset\subset(-1,1)$, all sufficiently small positive $\varepsilon$ give failures at every $\lambda\in\mathcal I$. In the plane $h_\varepsilon(\theta)=1+\varepsilon\cos4\theta$.
	
	\paragraph{The moving interface and the points it excludes.}
	In the first coordinate plane take the orthogonal unit vectors
	\[
	n_\beta=(\cos\beta,\sin\beta,0,\ldots,0),
	\qquad a_\beta=(-\sin\beta,\cos\beta,0,\ldots,0).
	\]
	The two equal-height points are $0$ and $v_\varepsilon=a_\beta-\varepsilon^2n_\beta$, with normals $n_\beta$ and $-n_\beta$. Their cells are $q\cdot v_\varepsilon\leq0$ and $q\cdot v_\varepsilon\geq0$. The small normal correction ensures $n_\beta\cdot v_\varepsilon=-\varepsilon^2<0$, so these data are realized on a smooth ellipsoidal obstacle. It changes the captured volume only at order $\varepsilon^2$, leaving the first-order effect of the anisotropy visible. The obstacle may depend on both $\varepsilon$ and $\lambda$.
	
	Use image coordinates $p=xn_\beta+ya_\beta+z$, with $z\perp n_\beta,a_\beta$. The transported cell interface has the graph $y=Y_\varepsilon(x,z)$, whose expansion away from the origin is
	\[
	Y_\varepsilon(x,z)
	=-4\varepsilon\sin(4\beta)\frac{x^3}{x^2+|z|^2}
	+o(\varepsilon).
	\]
	The two labels select, up to null sets,
	\[
	\begin{aligned}
		\mathscr B_\varepsilon(\beta,\lambda)
		={}W_\varepsilon\cap
		\bigl(\{y<Y_\varepsilon,\ x>t_\varepsilon\}
		\cup\{y>Y_\varepsilon,\ x<-t_\varepsilon\}\bigr),
		\quad t_\varepsilon=\lambda h_\varepsilon(n_\beta).
	\end{aligned}
	\]
	For $\sin4\beta>0$, the interface moves down where $x>0$ and up where $x<0$, to first order. Where only one cap condition holds, this removes a strip from each selected region: the points still belong to $W_\varepsilon$, but their new cell assigns the normal whose cut they fail. In the central strip $|x|<|t_\varepsilon|$, both cap conditions fail if $\lambda>0$ and both hold if $\lambda<0$. Changing the label there has no effect. The loss from moving the interface therefore comes from the outer regions $|x|>|t_\varepsilon|$.
	
	Since $\Phi_\varepsilon$ is even, $D\Phi_\varepsilon$ is odd and one-homogeneous. In the plane it therefore maps the separating slope line to an exact line through the origin. Its first-order slope in these coordinates is $-4\varepsilon\sin4\beta$. At $\lambda=0$, its rotation removes two wedges from the selected portions, as illustrated in Figure~\ref{fig:smooth-transport-failure}. To establish a deficit, we must also account for the moving Wulff boundary and compare with the smallest directional cap.
	
	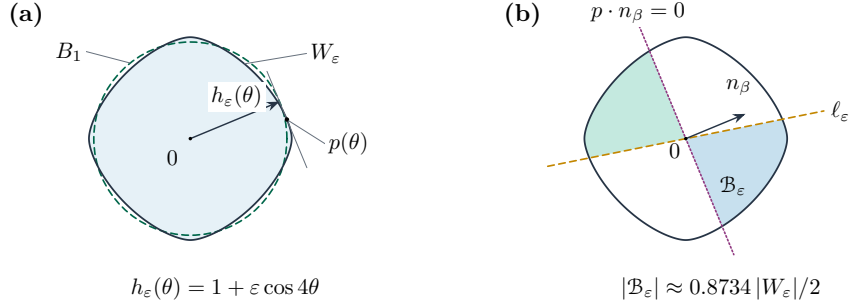
\begin{figure}[!htbp]
		\centering
		\scalebox{\CapFigureScale}{%
			\begin{minipage}{\textwidth}
				\centering
				\begin{tikzpicture}[cap panel]
					\CapPanel{a}
					\begin{scope}[shift={(3.05,2.80)},scale=1.52]
						\path[fill=capBlue!10] plot[domain=0:360,samples=241]
						({capwx(\x)},{capwy(\x)}) -- cycle;
						\draw[cap reference] (0,0) circle (1);
						\draw[cap edge] plot[domain=0:360,samples=241]
						({capwx(\x)},{capwy(\x)}) -- cycle;
						\coordinate (p) at ({capwx(22.5)},{capwy(22.5)});
						\draw[cap leader] ({capwx(22.5)-.54*sin(22.5)},
						{capwy(22.5)+.54*cos(22.5)}) --
						({capwx(22.5)+.54*sin(22.5)},{capwy(22.5)-.54*cos(22.5)});
						\draw[cap normal] (0,0) -- ({cos(22.5)},{sin(22.5)});
						\fill (p) circle (.023);
						\fill (0,0) circle (.018);
						\node[anchor=north east] at (-.04,-.04) {$0$};
					\end{scope}
					\node[cap label] at (3.78,3.46) {$h_\varepsilon(\theta)$};
					\draw[cap leader] (5.17,2.72) -- (4.60,3.08);
					\node[cap label,anchor=west] at (5.17,2.72) {$p(\theta)$};
					\draw[cap leader] (1.41,4.17) -- (2.10,3.99);
					\node[cap label,anchor=east] at (1.41,4.17) {$B_1$};
					\draw[cap leader] (4.88,4.14) -- (3.902,3.977);
					\node[cap label,anchor=west] at (4.88,4.14) {$W_\varepsilon$};
					\node at (3.60,.45) {$h_\varepsilon(\theta)=1+\varepsilon\cos4\theta$};
				\end{tikzpicture}\hfill
				\begin{tikzpicture}[cap panel]
					\CapPanel{b}
					\begin{scope}[shift={(3.05,2.80)},scale=1.52]
						\path[cap liquid] (0,0) --
						plot[domain=-58.1802928206:22.6432391504,samples=90]
						({capwx(\x)},{capwy(\x)}) -- cycle;
						\path[cap secondary] (0,0) --
						plot[domain=121.8197071794:202.6432391504,samples=90]
						({capwx(\x)},{capwy(\x)}) -- cycle;
						\draw[cap edge] plot[domain=0:360,samples=241]
						({capwx(\x)},{capwy(\x)}) -- cycle;
						\draw[draw=capPurple!85!black,line width=.75pt,densely dotted]
						({-1.30*sin(22.5)},{1.30*cos(22.5)}) --
						({1.30*sin(22.5)},{-1.30*cos(22.5)});
						\draw[cap guide,draw=capAmber!85!black,line width=.75pt] (-1.43,-.2864758) -- (1.43,.2864758);
						\draw[cap normal] (0,0) -- ({.67*cos(22.5)},{.67*sin(22.5)});
						\node[anchor=south] at (.55,.36) {$n_\beta$};
						\node at (.49,-.49) {$\mathscr B_\varepsilon$};
						\fill (0,0) circle (.018);
						\node[anchor=north east,inner sep=1pt] at (-.03,-.03) {$0$};
					\end{scope}
					\node[cap label] at (2.30,4.78) {$p\cdot n_\beta=0$};
					\node[cap label,anchor=west] at (5.28,3.23) {$\ell_\varepsilon$};
					\node at (3.60,.45) {$|\mathscr B_\varepsilon|\approx0.8734\,|W_\varepsilon|/2$};
				\end{tikzpicture}
			\end{minipage}%
		}
		\caption{A planar instance of
			Proposition~\ref{prop:quartic-obstruction}, with $\lambda=0$,
			$\varepsilon=0.05$ and $\beta=\pi/8$.
			(a) The support-function perturbation; green dashed: $B_1$. The point labelled $p(\theta)$ is the boundary point of $W_\varepsilon$ with outer unit normal $n_\theta$, given by $p(\theta)=h_\varepsilon(\theta)n_\theta +h_\varepsilon'(\theta)a_\theta$, where $n_\theta=(\cos\theta,\sin\theta)$, $a_\theta=(-\sin\theta,\cos\theta)$.
			(b) The blue and green portions form the captured image $\mathscr B_\varepsilon$, cut by the dotted cap line and the dashed interface
			$\ell_\varepsilon=D\Phi_\varepsilon(\{q:q\cdot v_\varepsilon=0\})$, where $v_\varepsilon=a_\beta-\varepsilon^2n_\beta$. Its area is below $|W_\varepsilon|/2$. The displayed ratio is computed by integrating the area density $h_\varepsilon(h_\varepsilon+h_\varepsilon'')/2$
			over the captured normal angles. The script for reproducing the numerical ratio is available from the authors upon request.}
		\label{fig:smooth-transport-failure}
	\end{figure}
	\vspace{-1.5em}
	
	\paragraph{Comparison with the smallest cap.}
	Let $m_0(\lambda)=|B_1\cap\{p:p_1>\lambda\}|$ and write $M_\varepsilon(\beta,\lambda)$ for the captured volume above. The appendix obtains
	\[
	\begin{aligned}
		m^{\min}_{\tension_\varepsilon,\lambda}
		&=m_0(\lambda)-\varepsilon|a_d(\lambda)|+o(\varepsilon),\\
		M_\varepsilon(\beta,\lambda)
		&=m_0(\lambda)-\varepsilon
		\bigl(a_d(\lambda)\cos4\beta+b_d(\lambda)\sin4\beta\bigr)
		+o(\varepsilon),
	\end{aligned}
	\]
	where $b_d(\lambda)>0$. The coefficient $a_d$ describes the directional cap variation. The additional coefficient $b_d$ is the net contribution of the interface motion together with the remaining outer-boundary variation; it is not the interface term alone. Choosing the orientation so that $(\cos4\beta,\sin4\beta)$ points along $(a_d,b_d)$ makes the captured volume decrease at rate $\sqrt{a_d^2+b_d^2}$, strictly faster than the rate $|a_d|$ of the minimum cap. This comparison includes the minimization over all cap directions.
	
	In the planar case, $a_2(\lambda)=4\lambda(1-\lambda^2)^{3/2}$ and $b_2(\lambda)=4(1-\lambda^2)^2$ are the coefficients in the proof of Proposition~\ref{prop:quartic-obstruction}. For this $M_\varepsilon(\beta,\lambda)$, choosing $\cos4\beta=\lambda$ and $\sin4\beta=\sqrt{1-\lambda^2}$ gives
	\[
	M_\varepsilon(\beta,\lambda)-m^{\min}_{\tension_\varepsilon,\lambda}
	=-4(1-|\lambda|)(1-\lambda^2)^{3/2}\varepsilon
	+o(\varepsilon).
	\]
	
	In the planar case at zero adhesion, $\beta=\pi/8$ and the captured area falls below the reference area by $4\varepsilon+o(\varepsilon)$. In fact, $|W_\varepsilon|=\pi(1-15\varepsilon^2/2)$, so the reference half-area has no first-order change, while the selected area does. This makes the loss particularly visible. The numerical ratio in Figure~\ref{fig:smooth-transport-failure} uses the exact area integral in its caption; at $\varepsilon=0.05$ the curvature factor $h_\varepsilon+h_\varepsilon''$ is still at least $1/4$.
	
	For each fixed $\lambda\in(-1,1)$ the coefficient is strictly negative. The appendix controls the remainders uniformly on compact adhesion intervals, so one sufficiently small positive $\varepsilon$ works throughout such an interval. No uniform deficit is asserted up to the endpoints, where the coefficient tends to zero. These examples use equal boundary heights; the general classification based solely on such tests remains open. \vspace{-.7em}
	
	\paragraph{Perturbations of prescribed Wulff bodies.}
	An invertible linear change carries the example near any centered ellipsoid and multiplies both compared volumes by the same factor, preserving their strict ordering. Non-even examples then follow by first fixing a perturbation and its deficit and only afterwards translating its Wulff body by a sufficiently small amount. This second step uses persistence of the fixed test; translation of the Wulff body is not invoked as an invariance of the transported estimate.
	
	For an arbitrary origin-symmetric convex body, the construction instead begins with a nearby symmetric polytope, fixes one of its facet tests, and smooths the polytope while preserving that deficit. These are the two routes in the following corollary. The first gives $C^k$ approximation with uniform ellipticity bounds; the second gives Hausdorff approximation with ellipticity constants that may depend on the desired accuracy. Along the first family the deficits tend to zero, consistently with validity of the estimate at the limiting ellipsoid itself.
	\smallskip
	
	\begin{corollary}[Smooth counterexamples near prescribed Wulff bodies]\label{cor:smooth-failures-near-ellipsoids}\label{rev:cor:smooth-failures}\vspace{-.5em}
		\leavevmode\begin{enumerate}[label=(\roman*),itemsep=0em,leftmargin=*]
			\item
			Fix \(d\geq 2\), \(\lambda\in(-1,1)\), and a centered ellipsoid \(W_0\). For every \(k\in\N_0\) and every \(\eta>0\), there exist both even and non-even smooth surface tensions \(\tension\) such that
			\[
			\|\tension-h_{W_0}\|_{C^k(\mathbb{S}^{d-1})}<\eta,
			\]
			\(D^2(\tension^2/2)\) is uniformly positive definite, and \(\wulff\) fails the transported \(\lambda\)-ABP property. Ellipticity constants can be chosen uniformly along the approximating families.
			\item Let $W_0=-W_0\subset\R^d$ be any full-dimensional convex body and $\eta>0$. There is an even smooth surface tension $\tension$ with $D^2(\tension^2/2)>0$ away from zero such that $d_H(\wulff,W_0)<\eta$ and $\wulff$ fails the transported $0$-ABP property. Ellipticity constants may depend on $\eta$.
		\end{enumerate}
	\end{corollary}
	
	\begin{proof}[Proof of Corollary~\ref{cor:smooth-failures-near-ellipsoids}]
		Write $W_0=LB_1$, where $L=L^T>0$, and let $\tension_\varepsilon$
		be given by Proposition~\ref{prop:quartic-obstruction}. Set $\tension_{\varepsilon,L}(q)=\tension_\varepsilon(L^Tq)$. The changes
		$x=L\widetilde x$, $q=L^{-T}\widetilde q$, $p=L\widetilde p$, and
		$n=L^{-T}\widetilde n/|L^{-T}\widetilde n|$ identify the cells and cap conditions, since $D(\tension_{\varepsilon,L}^2/2)(q)=L D(\tension_\varepsilon^2/2)(L^Tq)$.
		All captured and directional cap volumes acquire the factor
		$\det L$. Thus the strict deficit is preserved, and
		$\tension_{\varepsilon,L}\to h_{W_0}$ in every $C^k(\mathbb{S}^{d-1})$.
		
		For fixed $b\neq 0$, add $t\,b\cdot q$. This translates the Wulff body to $LW_\varepsilon+tb$ and makes its support function non-even for $t\ne0$. For sufficiently small $|t|$, positivity, the prescribed partial-wetting condition and ellipticity persist, as does the strict deficit by Corollary~\ref{cor:strict-deficit-persistence}. Taking $\varepsilon$ and then $|t|$ small gives the required $C^k$-proximity. Convergence in $C^2$ to $h_{W_0}$ yields ellipticity bounds uniform along both families.
		
		For (ii), we first introduce a facet obstruction at the polyhedral level and then preserve its strict deficit while smoothing. Choose an origin-symmetric polytope $P$ with $d_H(P,W_0)<\eta/2$. Theorem~\ref{rev:thm:facet} provides a strict two-point deficit for $P$. Approximate $h_P$ uniformly on the sphere by even smooth support functions, and add a positive multiple tending to zero of the Euclidean norm. This gives even surface tensions $\tension_j$ for which $D^2(\tension_j^2/2)$ is uniformly positive definite away from zero for each $j$.
		
		For completeness, the smoothing can be obtained by averaging $h_P(Rq)$ against a smooth approximate identity on the rotation group. Such averaging preserves convexity, one-homogeneity, and evenness; addition of $\varepsilon_j|q|$ makes the tangential Hessian strictly positive.
		
		Corollary~\ref{cor:strict-deficit-persistence} preserves the strict deficit for all sufficiently large $j$. Choosing such a $j$ with $d_H(W_{\tension_j},P)<\eta/2$ proves the claim.
	\end{proof}
	
	\section{Scope and remaining questions}\label{sec:questions}
	The preceding sections separate the capillary comparison from one proposed mechanism for proving it. In the plane, at zero adhesion, a smooth even nonellipsoidal tension in the class of Theorem~\ref{thm:two-point-rigidity} satisfies the sharp variational inequality although its unrestricted transported ABP property fails. The remaining questions therefore concern both extensions of the comparison itself and estimates that use more information than arbitrary boundary data provide. A further issue is how a comparison with model caps relates to minimization outside a fixed obstacle.
	\vspace{-.9em}
	
	\paragraph{Scope of the comparison and classification.}
	The sharp comparison constant depends on the tension and adhesion parameter but is uniform over convex obstacles. The stability theorem controls three different quantities: the mass outside a principal component, the symmetric difference from a freely translated cap, and the excess of its directional cap volume above the minimum. Corollary~\ref{cor:sharp-component-concentration} proves that the quadratic order of the first estimate is optimal: on a half-plane wall, two separated equality caps with area fractions $1-r$ and $r$ have deficit $\sqrt{1-r}+\sqrt r-1\sim\sqrt r$ as $r\to0^+$. This concentration example does not establish optimality of the square-root order in the shape estimate. Moreover, the comparison cap need not be admissible for the prescribed obstacle; this distinction becomes relevant to localization below.
	
	The quantifiers in the transported property also matter. Definition~\ref{rev:def:abp} allows arbitrary boundary heights, whereas Proposition~\ref{rev:thm:abp} uses only $\psi=u|_\Gamma$ for an auxiliary Neumann solution. Such a trace is constrained by an equation and boundary conditions that are absent from the finite geometric tests. We have not shown that the obstructing configurations arise from these solutions. Even agreement with prescribed heights at two points would not determine the image associated with the full wetted trace: the remaining boundary points can change the slope cells and contribute additional image. Thus an estimate restricted to actual Neumann traces remains open and would require analytic information beyond the unrestricted geometry. Compact data supply the input to the conditional comparison, while finite data detect failure and force rigidity; no equivalence between these classes at a fixed adhesion level is asserted.
	\vspace{-1em}
	
	\paragraph{Positive adhesion in the plane.}
	The first variational question is whether $\EnergyC(E)\geq2\sqrt{m^{\min}_{\tension,\lambda}|E|}$ holds for every planar competitor when $0<\lambda<\lambda_*(\tension)$. The planar completion lemma already applies for every $\lambda>-1$: filling exterior pockets does not increase the energy. The difficulty occurs afterwards, when a portion of the obstacle is filled to replace the wetted trace by a flat section.
	
	Write $F$ for a completed component, $\Gamma$ for its wetted trace, $D$ for the filling inside the obstacle, and $K=K_{\tension,\lambda}(n)$ for the associated model cap. In~\eqref{rev:eq:filling-identity}, the contact excess $\Delta_\tension(\Gamma)$ measures the excess anisotropic wetted measure over that of the flat section with the same total normal flux. Its contribution $-\lambda\,\Delta_\tension(\Gamma)$ is favorable for $\lambda\leq0$ and becomes a possible loss for $\lambda>0$. The filling bound still provides two compensating terms: $R_K(F)$ records the difference between the free-boundary densities $\tension$ and $h_K$, and the added volume $|D|$ increases the Wulff lower bound. Accordingly, it would suffice to prove
	\[
	\lambda\,\Delta_\tension(\Gamma)\leq R_K(F)
	+2|K|^{1/2}\bigl((|F|+|D|)^{1/2}-|F|^{1/2}\bigr).
	\]
	This inequality would compensate the contact loss for each completed component and close the existing comparison argument. It is a sufficient condition; its failure alone would not disprove the capillary inequality, since the argument has also discarded the Wulff deficit of $F\cup D$ and the energy gained during completion.
	
	Theorem~\ref{rev:thm:filling} already covers $\Delta_\tension(\Gamma)=0$. Lemma~\ref{rev:lem:filling} explains why this is broader than flat contact for crystalline tensions: a single vertex of $W$ can support several reflected contact normals, all lying in its normal cone, so the excess can vanish even when those normals vary. Polygonal obstacles and crystalline tensions therefore provide explicit geometries in which to examine the proposed compensation.
	\vspace{-1em}
	
	\paragraph{Higher dimensions.}
	For $d\geq3$ and even $\tension$, the question at zero adhesion is whether every finite-perimeter competitor with an arbitrary wetted trace satisfies
	\[
	\PerAni(E;\R^d\setminus\C)
	\geq d\left(|\wulff|/2\right)^{1/d}|E|^{(d-1)/d}.
	\]
	Central symmetry makes every hyperplane through the origin divide $W$ into equal volumes, which explains the factor $|W|/2$; the exponent $(d-1)/d$ is dictated by perimeter and volume scaling. The filling identity is already valid in every dimension for the prescribed geometries. What is missing from the present variational argument is a reduction of arbitrary wetted traces to fillings of the required form. In the plane, completion reduces contact to a boundary arc, closed by a chord, or to the entire obstacle boundary. Connectedness of a wetted patch in higher dimensions does not supply that same reduction to a hyperplane filling.
	
	The filling theorem, the comparison for two parallel wetted caps in Proposition~\ref{prop:cc-parallel-caps-zero}, and the ellipsoidal comparison establish the inequality in their respective settings. In particular, two separated contact regions are compatible with a sharp bound even when their exterior connection is unrestricted. The unresolved issue is the geometry of a general wetted trace and how to control its replacement. The failures in Section~\ref{rev:sec:obstruction} explain why the unrestricted transported estimate cannot provide that control for general anisotropies.
	\vspace{-1em}
	
	\paragraph{Other transports and rigidity.}
	One may also change the potential while keeping the Wulff body fixed. For a smooth tension strictly convex away from $0$ and a smooth strictly convex profile $F$, consider $\Phi_F=F(\tension)$, normalized by $F'(0)=0$ and $F'(1)=1$. The identity
	\[
	D\Phi_F(q)=F'(\tension(q))\,D\tension(q),\qquad q\neq0,
	\]
	shows what this choice changes. The first normalization makes the transport extend continuously by zero at the origin, and the second preserves its boundary map on $\{\tension=1\}$. Along a slope ray, $D\tension$ is constant, while $F'(\tension)$ determines the radial position of the image point. For example, when $\tension=|\cdot|$, the profile $F(t)=t^4/4$ sends radius $r$ to $r^3$, while fixing the unit sphere. Changing $F$ therefore changes the interior transport and its Jacobian, which enter the balance tested by the boundary images.
	
	Which pairs $(\tension,F)$ satisfy the two-point zero-adhesion estimate? Affine covariance holds for general profiles, as observed in Remark~\ref{rem:4.3}, but Theorem~\ref{thm:two-point-rigidity} also uses the two-homogeneity of $\tension^2/2$. Its classification cannot be transferred to a general $F$ from the boundary normalization alone. For the quadratic profile, classification without evenness or at a prescribed nonzero adhesion level, classification using only equal-height tests, and quantitative rigidity in terms of two-point deficits remain open. The perturbative construction also imposes no uniform bound on obstacle eccentricity; it does not settle estimates restricted to a class with a prescribed eccentricity bound.
	\vspace{-1em}
	
	\paragraph{Localization.}
	A small shape error leaves a separate question: where does the approximating cap sit relative to $\C$? Theorem~\ref{thm:unrestricted-planar-stability} permits free translation and does not ensure that the cap lies outside $\C$ with its facet on the wall. A localization estimate would have to connect proximity in measure to the geometry of the obstacle. For a fixed obstacle,~\eqref{eq:universal-profile-deficit} makes a further distinction: the universal deficit is the sum of the excess above the fixed-obstacle profile and the gap between that profile and the universal cap bound. If the second term is positive, approaching the fixed-obstacle minimum does not make the universal deficit vanish.
	
	Even for a half-plane, localization cannot be uniform over the full adhesion interval. To see the mechanism, take $\C=\{x_2\leq0\}$ and a unit disk centered at $(0,1+h)$, with $h>0$. The disk has no wetted boundary, so its energy is $2\pi$, independently of $h$ and $\lambda$. The reference cap has area
	\[
	m_{|\cdot|,\lambda}^{\min}
	=\int_\lambda^1 2\sqrt{1-t^2}\,dt
	=\arccos\lambda-\lambda\sqrt{1-\lambda^2}.
	\]
	Consequently, for this disk,
	\[
	\Deficit_\lambda=\sqrt{\frac{\pi}{\arccos\lambda-\lambda\sqrt{1-\lambda^2}}}-1
	\longrightarrow 0 \quad\text{ as }\quad \lambda\to(-1)^+. 
	\]
	At the prescribed area $\pi$, the model cap approaches a whole unit disk as its removed segment shrinks, while the competitor remains a distance $h$ from the wall. Since the deficit is independent of $h$, that distance may even diverge along the same limit. This is consistent with comparison to freely translated caps, but prevents the deficit alone from forcing proximity to the wall uniformly as $\lambda\to-1$. At the endpoint itself, zero extension gives $\EnergyC(E)=\PerAni(E;\R^2)$ and $K_{\tension,-1}(n)=W$: equality reduces to the full-space Wulff case, allowing admissible translated dilates without a wetted-face requirement.
	\vspace{-1em }
	
	\paragraph{Generalized minimizers.}
	At fixed adhesion, unbounded obstacles raise a different compactness issue: minimizing sequences may move to infinity. Recentring an escaping component also translates the obstacle, and the resulting limit geometry enters the generalized minimization problem. Nani--Vincini~\cite[Theorem~1.3]{nani2026generalized} recently established isotropic generalized existence through finitely many limiting regions; their Theorem~6.8 relates the profile and generalized minimizers to the obstacle's asymptotic geometry. For centered ellipsoidal Wulff bodies, the affine change of Section~\ref{rev:sec:ellipsoids} transfers these results because it multiplies energy and volume by the same fixed determinant factor. Extending this theory beyond ellipsoids remains a separate anisotropic problem. Their obstacle-variation theorem~\cite[Theorem~3.8]{nani2026generalized} changes $\C$; Corollary~\ref{cor:strict-deficit-persistence} instead varies $W$ while keeping a finite test fixed. These two continuity statements concern different objects and serve different purposes.
	
	\section*{Statements and Declarations}
	
	\paragraph{Funding.}
	LC was partially supported by ANID Chile through the Postdoctoral Fondecyt Grant No. 3240062. AQ was partially supported by FONDECYT grant 1231585 and by grant PID2023-148028NB-I00, funded by MICIU/AEI/10.13039/501100011033 and by ERDF/EU. \vspace{-.7em}
	
	\paragraph{Author contributions.} All authors contributed equally to this work.\vspace{-.7em}
	
	\paragraph{Data availability.}
	No datasets were generated or analyzed in this study. No external research dataset is required for these proofs. \vspace{-.7em}
	
	\paragraph{Code availability.}
	The figures were rendered in TikZ from code prepared with assistance from ChatGPT Astra 6.0 and reviewed by the authors. The Python script reproducing the numerical area ratio in Figure~\ref{fig:smooth-transport-failure} by evaluating the area integral specified in its caption is available from the authors upon request.\vspace{-.7em}
	
	\paragraph{Conflict of interest.}
	The authors have no competing interests to declare that are relevant to the content of this article. \vspace{-.7em}
	
	\paragraph{Use of artificial intelligence.}
	The authors used ChatGPT (OpenAI; Astra 6.0) and Claude (Anthropic; Opus 5.5) to improve the English of the manuscript and to assist with the development and exposition of Sections~\ref{sec:variat_comp_planar_stab}--\ref{sec:questions}, including the discussion of related literature and the geometric interpretation of the obstruction examples. The figures were rendered in LaTeX/TikZ from source code prepared with AI assistance, as disclosed in their captions. The authors reviewed and revised the figure code and checked the diagrams against the geometric constructions and formulas described in the manuscript. The authors critically reviewed and revised all AI-assisted material, verified the mathematical arguments and references, and take full responsibility for the final manuscript. \vspace{-1em}
	
	\paragraph{Acknowledgements.}
	Part of this work was carried out during a visit  to Departamento de Matem\'atica of Universidad T\'ecnica Federico Santa Mar\'ia, whose hospitality is gratefully acknowledged.

	\appendix
	\section{Perturbative failure in every dimension}\label{app:perturbation}
	We prove the perturbative construction used in Section~\ref{subsec:smooth-counterexamples}, including uniformity over compact adhesion intervals.
	
	\begin{proposition}[Perturbative failure in every dimension]\label{prop:quartic-obstruction}
		Let $d\geq2$ and $f(u)=u_1^4-6u_1^2u_2^2+u_2^4$ on $\mathbb S^{d-1}$. For sufficiently small $\varepsilon>0$, $h_\varepsilon=1+\varepsilon f$ is the support function of a smooth, origin-symmetric, uniformly convex body \(W_\varepsilon\).
		Its one-homogeneous extension \(\tension_\varepsilon\) satisfies uniform ellipticity bounds for \(D^2(\tension_\varepsilon^2/2)\).
		
		For every compact interval \(\mathcal I\subset\subset(-1,1)\), there exists \(\varepsilon_0=\varepsilon_0(d,\mathcal I)>0\) such that, whenever \(0<\varepsilon<\varepsilon_0\), the body \(W_\varepsilon\) fails the transported \(\lambda\)-ABP property for every \(\lambda\in \mathcal I\). For each \(\lambda\), two equal-height points on a smooth strictly convex ellipsoidal obstacle suffice.
	\end{proposition}
	
	\begin{proof}[Proof of Proposition~\ref{prop:quartic-obstruction}]
		We compare the first-order expansions of the minimum
		directional cap volume and the volume captured by a two-point configuration. A suitable orientation produces a negative gap. The final remainder estimate makes this gap uniform for adhesion parameters in a compact subinterval of $(-1,1)$. The polynomial
		$q_1^4-6q_1^2q_2^2+q_2^4$
		is homogeneous and harmonic of degree four. Since \(h_\varepsilon\to1\) in every \(C^k(\mathbb{S}^{d-1})\)-norm,
		\[
		D_S^2h_\varepsilon(u)[\tau,\tau]+h_\varepsilon(u)|\tau|^2>0,
		\quad\text{ for every }\; u\in\mathbb{S}^{d-1},\text{ and }\tau\in T_u\mathbb{S}^{d-1}\setminus\{0\},
		\]
		for sufficiently small \(\varepsilon\). Positivity and uniform ellipticity follow by perturbation of the Euclidean norm. Fix \(\lambda\in(-1,1)\), and write
		\[
		s=\sqrt{1-\lambda^2},
		\qquad
		m_0(\lambda)=|B_1\cap\{p: p_1>\lambda\}|.
		\]
		Let \(\omega_k\) denote the volume of the unit
		ball in \(\R^k\), with \(\omega_0=1\), and set
		\[
		a_d(\lambda)
		=
		\frac{(d+4)}{(d+1)}\,\omega_{d-1}\lambda s^{d+1}.
		\]
		
		\emph{Directional cap volumes.}
		The normal velocity of \(\partial W_\varepsilon\) at \(\varepsilon=0\) is \(f\), while the cutting plane in direction \(n\) moves with velocity \(\lambda f(n)\). Hence
		\[
		\left.\frac{d}{d\varepsilon}\right|_{\varepsilon=0}
		m_{\tension_\varepsilon,\lambda}(n)
		=
		\int_{\{u\in\mathbb S^{d-1}:\,u\cdot n>\lambda\}} f(u)\,d\calH^{d-1}(u) - \lambda\,\omega_{d-1}s^{d-1}f(n).
		\]
		For completeness, write $u=tn+\sqrt{1-t^2}\,v$, where $v$ ranges over the unit sphere in $n^\perp$. Let $\mathbb E$ denote averaging with respect to its normalized surface measure, and put $P=I-n\otimes n$. The second and fourth moments are
		\[
		\begin{aligned}
			\mathbb E[v_i v_j]
			&=\frac{P_{ij}}{d-1},\\
			\mathbb E[v_i v_j v_k v_l]
			&=\frac{P_{ij}P_{kl}+P_{ik}P_{jl}+P_{il}P_{jk}}
			{(d-1)(d+1)},
		\end{aligned}
		\]
		for $i,j,k,l\in\{1,\ldots,d\}$. Let $\widetilde f$ denote the homogeneous harmonic quartic polynomial whose restriction to $\mathbb S^{d-1}$ is $f$, and write $\Delta_{n^\perp}:=\Delta-(n\cdot\nabla)^2$.
		Harmonicity and homogeneity give
		\[
		\Delta_{n^\perp}\widetilde f(n)=-12f(n),
		\qquad
		\Delta_{n^\perp}^{\,2}\widetilde f(n)=24f(n).
		\]
		Expanding $\widetilde f(tn+\sqrt{1-t^2}\,v)$, the odd terms vanish upon averaging. The moment identities evaluate the quadratic and quartic terms through the two contractions above, yielding the normalized section average
		\[
		\frac{1}{(d-1)\omega_{d-1}}\int_{\mathbb S^{d-1}\cap n^\perp}f(tn+\sqrt{1-t^2}\,v)\,d\calH^{d-2}(v)=Q_d(t)f(n),
		\]
		\[
		Q_d(t)=t^4-\frac{6t^2(1-t^2)}{d-1}+\frac{3(1-t^2)^2}{(d-1)(d+1)}.
		\]
		For $d=2$, this is the average over the two points of $\mathbb S^0$. Multiplying by the spherical section density and integrating gives
		\[
		(d-1)\int_\lambda^1(1-t^2)^{(d-3)/2}Q_d(t)\,dt
		=\frac{\lambda(1-\lambda^2)^{(d-1)/2}}{d+1}
		\bigl((d+4)\lambda^2-3\bigr),
		\]
		as is checked by differentiation and the value at $\lambda=1$. Therefore
		\[
		\int_{\{u\in\mathbb S^{d-1}:\,u\cdot n>\lambda\}} f(u)\,d\calH^{d-1}(u)
		= \frac{\omega_{d-1}\lambda\,s^{d-1}}{(d+1)}\bigl((d+4)\lambda^2-3\bigr)f(n).
		\]
		Consequently,
		\begin{equation}\label{eq:quartic-cap-variation}
			m_{\tension_\varepsilon,\lambda}(n)
			= m_0(\lambda)-\varepsilon a_d(\lambda)f(n)+o(\varepsilon),
		\end{equation}
		uniformly for $n\in\mathbb{S}^{d-1}$ and $\lambda$ in compact subintervals of $(-1,1)$. Since \(-1\leq f\leq 1\), with both extreme values attained,
		\begin{equation}\label{eq:quartic-minimum-cap}
			m_{\tension_\varepsilon,\lambda}^{\min}
			= m_0(\lambda)-\varepsilon|a_d(\lambda)|+o(\varepsilon).
		\end{equation}
		The joint uniformity follows from the smooth support-function perturbation and the uniform transversality of the cutting hyperplanes with the unit sphere when $\lambda\in \mathcal I\subset\subset(-1,1)$.
		
		\emph{The two-point image.}
		Set
		\[
		n_\beta=(\cos\beta,\sin\beta,0,\ldots,0),
		\qquad
		a_\beta=(-\sin\beta,\cos\beta,0,\ldots,0).
		\]
		Consider the equal-height points
		\[
		0,
		\qquad
		v_\varepsilon=a_\beta-\varepsilon^2n_\beta,
		\]
		with supporting normals \(n_\beta\) and \(-n_\beta\), respectively. Their cells are
		\[
		\{q:q\cdot v_\varepsilon\le0\},
		\qquad
		\{q:q\cdot v_\varepsilon\ge0\}.
		\]
		Set $\Phi_\varepsilon:=\varphi_\varepsilon^2/2$. The set captured by this two-point configuration is
		\[
		\begin{aligned}
			\mathscr B_\varepsilon(\beta,\lambda)
			:= {}&
			\bigl\{D\Phi_\varepsilon(q)\in\operatorname{int}W_\varepsilon:\,
			q\cdot v_\varepsilon\le0,\quad
			D\Phi_\varepsilon(q)\cdot n_\beta
			>\lambda\,\tension_\varepsilon(-n_\beta)\bigr\}
			\\
			&{}\cup
			\bigl\{D\Phi_\varepsilon(q)\in\operatorname{int}W_\varepsilon:\,
			q\cdot v_\varepsilon\ge0,\quad
			D\Phi_\varepsilon(q)\cdot(-n_\beta)
			>\lambda\,\tension_\varepsilon(n_\beta)\bigr\}.
		\end{aligned}
		\]
		Write
		$M_\varepsilon(\beta,\lambda):=|\mathscr B_\varepsilon(\beta,\lambda)|$.
		
		Use coordinates \((x,y,z)\) along \(n_\beta,a_\beta\), and their orthogonal complement. On the unit sphere,
		\[
		f= \cos(4\beta)(x^4-6x^2y^2+y^4) - 4\sin(4\beta)xy(x^2-y^2).
		\]
		Let $y=Y_\varepsilon(x,z)$ be the exact graph of
		$D\Phi_\varepsilon(\{q:q\cdot v_\varepsilon=0\})$.
		Replacing $v_\varepsilon$ by $a_\beta$ changes this graph by $O(\varepsilon^2|(x,z)|)$ and the captured volume by $O(\varepsilon^2)$. Its first-order expansion can therefore be computed from the untilted plane $q\cdot a_\beta=0$:
		\[
		y=Y_\varepsilon(x,z)= -4\varepsilon\sin(4\beta)\frac{x^3}{x^2+|z|^2} +o(\varepsilon).
		\]
		Indeed, if \(P_4\) denotes the rotated homogeneous
		quartic polynomial, then
		\[
		\frac12\tension_\varepsilon(q)^2
		=
		\frac12|q|^2 + \varepsilon\frac{P_4(q)}{|q|^2} + O(\varepsilon^2|q|^2).
		\]
		Differentiating in the \(a_\beta\)-direction on \(q\cdot a_\beta=0\) gives the displayed interface velocity.
		
		For the cosine component, reflection in \(y=0\) shows that the outer-boundary and cap-plane contributions coincide with those of a full directional cap. Their sum is \(-a_d(\lambda)\cos(4\beta)\). Inside $W_\varepsilon$, away from the transported interface, the captured indicator is
		\[
		\chi_{\mathscr B_\varepsilon(\beta,\lambda)}
		=\chi_{\{y<Y_\varepsilon\}}\chi_{\{x>t_\varepsilon\}}
		+\chi_{\{y>Y_\varepsilon\}}\chi_{\{x<-t_\varepsilon\}},
		\qquad t_\varepsilon=\lambda h_\varepsilon(n_\beta).
		\]
		For either sign of $t$, almost everywhere in $x$,
		\begin{equation}\label{eq:cap-indicator-difference}
			\chi_{\{x>t\}}-\chi_{\{x<-t\}}
			=\operatorname{sgn}(x)\chi_{\{|x|>|t|\}}.
		\end{equation}
		Write $\ell=|\lambda|$ and set
		\[
		D_\ell:=\bigl\{(x,z)\in\R\times\R^{d-2}:\, x>\ell,\ x^2+|z|^2<1\bigr\}.
		\]        
		The first variation is linear in the rotated perturbation. For its sine component $-4xy(x^2-y^2)$, the cap-plane velocity is zero because this component vanishes at $\pm n_\beta$. On the lower unit hemisphere, $y=-\sqrt{1-x^2-|z|^2}$ and $d\calH^{d-1}=(1/|y|)dx\,dz$. Projection and central symmetry therefore give the outer-boundary coefficient
		\[
		\begin{aligned}
			&8\int_{\left\{\substack{x>\ell\\x^2+|z|^2<1}\right\}}
			x(2x^2+|z|^2-1)\,dx\,dz
			= \frac{8\omega_{d-2}}{d}\lambda^2s^d.
		\end{aligned}
		\]
		To evaluate the preceding integral, use the section identity
		\[
		\int_{|z|^2<1-x^2}(2x^2+|z|^2-1)\,dz
		=\omega_{d-2}(1-x^2)^{(d-2)/2}\left(2x^2-1+\frac{d-2}{d}(1-x^2)\right),
		\]
		and integrate $8x$ times its right-hand side from $\ell$ to $1$. Multiplying the interface velocity by~\eqref{eq:cap-indicator-difference} at $t=\lambda$ gives the transported-interface coefficient
		\[
		-8\int_{\left\{\substack{x>\ell\\x^2+|z|^2<1}\right\}}
		\frac{x^3}{x^2+|z|^2}\,dx\,dz.
		\]
		For $\lambda<0$, the two cap conditions coincide on the central strip and their difference vanishes there, explaining the threshold $\ell=|\lambda|$. Since
		\[
		\int_{\left\{\substack{x>\ell\\x^2+|z|^2<1}\right\}}
		x\,dx\,dz
		= \frac{\omega_{d-2}}{d}s^d,
		\]
		define
		\[
		b_d(\lambda)
		:=
		8\int_{\left\{\substack{x>\ell\\x^2+|z|^2<1}\right\}}
		x\left(\frac{x^2}{x^2+|z|^2}-\ell^2\right)\,dx\,dz.
		\]
		Combining the preceding formulas, the coefficient of $\sin(4\beta)$ in the first variation is
		\[
		\underbrace{8\int_{D_\ell}x(2x^2+|z|^2-1)\,dx\,dz}_{\text{outer boundary}}
		-
		\underbrace{8\int_{D_\ell}\frac{x^3}{x^2+|z|^2}\,dx\,dz}_{\text{transported interface}}
		=-b_d(\lambda).
		\]
		There is no cap-plane contribution to this coefficient, since the sine perturbation vanishes at $\pm n_\beta$. The factor $8$ combines the factor $4$ in the sine perturbation with the factor $2$ from the two centrally symmetric outer regions. For $\lambda<0$, the outer-boundary contribution from the central strip vanishes by oddness in $y$.
		
		For $d=2$, integration over $\R^0$ is understood to have unit mass. On $D_\ell$, we have $x>\ell\ge0$ and $x^2+|z|^2<1$, so
		\[
		\frac{x^2}{x^2+|z|^2}>x^2>\ell^2.
		\]
		Thus the integrand defining $b_d(\lambda)$ is strictly positive on a set of positive measure, and $b_d(\lambda)>0$ for every $\lambda\in(-1,1)$. We obtain
		\begin{equation}\label{eq:quartic-image-variation}
			M_\varepsilon(\beta,\lambda)
			= m_0(\lambda)-\varepsilon \bigl(a_d(\lambda)\cos4\beta + b_d(\lambda)\sin4\beta\bigr) +o(\varepsilon).
		\end{equation}
		Choose \(\beta\) so that
		\[
		\cos4\beta
		=\frac{a_d(\lambda)}{\sqrt{a_d(\lambda)^2+b_d(\lambda)^2}},
		\qquad
		\sin4\beta
		= \frac{b_d(\lambda)}
		{\sqrt{a_d(\lambda)^2+b_d(\lambda)^2}}.
		\]
		Subtracting~\eqref{eq:quartic-minimum-cap} from~\eqref{eq:quartic-image-variation} gives
		\[
		\begin{aligned}
			M_\varepsilon(\beta,\lambda)
			& - m_{\tension_\varepsilon,\lambda}^{\min}\\
			&\quad= -\varepsilon
			\left(\sqrt{a_d(\lambda)^2+b_d(\lambda)^2}
			- |a_d(\lambda)| \right)
			+o(\varepsilon)<0
		\end{aligned}
		\]
		for sufficiently small \(\varepsilon>0\).
		
		\emph{Uniform remainder.} Fix $\mathcal I\subset\subset(-1,1)$ and set
		\[
		\begin{aligned}
			L(\beta,\lambda)&=-a_d(\lambda)\cos4\beta-b_d(\lambda)\sin4\beta,\\
			R_\varepsilon(\beta,\lambda)&=M_\varepsilon(\beta,\lambda)-m_0(\lambda)-\varepsilon L(\beta,\lambda).
		\end{aligned}
		\]
		We prove $\sup_{(\beta,\lambda)\in[0,2\pi]\times \mathcal I}|R_\varepsilon(\beta,\lambda)|/\varepsilon\to0$. Since $|f|\leq1$, one has $B_{1-\varepsilon}\subset W_\varepsilon$. Fix $0<r<1/4$ and a smooth radial cutoff $0\leq\chi_r\leq1$, zero on $B_r$ and one outside $B_{2r}$. Split
		\[
		M_\varepsilon(\beta,\lambda)=M_\varepsilon^{\mathrm{out},r}(\beta,\lambda)+M_\varepsilon^{\mathrm{in},r}(\beta,\lambda),
		\]
		where
		\[
		\begin{aligned}
			M_\varepsilon^{\mathrm{out},r}(\beta,\lambda)&=\int_{\mathscr B_\varepsilon(\beta,\lambda)}\chi_r(p)\,dp,\\
			M_\varepsilon^{\mathrm{in},r}(\beta,\lambda)&=\int_{\mathscr B_\varepsilon(\beta,\lambda)}(1-\chi_r(p))\,dp.
		\end{aligned}
		\]
		For fixed $r$, the inverse transport is smooth away from zero. Treat the two weighted cell integrals separately. At $\varepsilon=0$, their cell plane is $y=0$ and their cap planes are respectively $x=\lambda$ and $x=-\lambda$. These planes and the unit sphere intersect transversally, uniformly for $\lambda\in \mathcal I$, including triple intersections when present. At $\lambda=0$ the cap planes coincide but occur in separate cell integrals. 
		
		For small $\varepsilon$, transversality persists, and the smooth weight adds no boundary. For fixed $r>0$, we may therefore straighten the active boundary hypersurfaces of each weighted cell integral in smooth local charts. Compactness of $[0,2\pi]\times\mathcal I$ permits a finite collection of such charts, with uniform bounds for small $\varepsilon$. After a partition of unity and a change of variables, the resulting integrals have fixed domains and common compact support, and their integrands have jointly continuous derivatives with respect to $\varepsilon$. Differentiation under the integral sign thus shows that $\partial_\varepsilon M_\varepsilon^{\mathrm{out},r}(\beta,\lambda)$ is jointly continuous in $(\varepsilon,\beta,\lambda)$. Writing $L_r(\beta,\lambda)=\left.\partial_\varepsilon M_\varepsilon^{\mathrm{out},r}(\beta,\lambda) \right|_{\varepsilon=0}$, uniform continuity on the compact parameter set gives
		\[
		M_\varepsilon^{\mathrm{out},r}(\beta,\lambda)
		=
		M_0^{\mathrm{out},r}(\beta,\lambda)
		+\varepsilon L_r(\beta,\lambda)+o_r(\varepsilon),
		\]
		uniformly for $(\beta,\lambda)\in[0,2\pi]\times\mathcal I$.
		
		The Wulff boundary avoids $B_{2r}$ for $\varepsilon<1/4$. The cell-interface velocity is bounded by $C|p|$; a cap plane meets $B_{2r}$ only for $|\lambda|\leq Cr$, when its velocity is bounded by $Cr$. Thus the omitted first-variation integrals satisfy
		\[
		\sup_{(\beta,\lambda)\in[0,2\pi]\times \mathcal I}|L_r(\beta,\lambda)-L(\beta,\lambda)|\leq C_{\mathcal I}r^d.
		\]
		For the inner part, homogeneity and smooth dependence on the sphere yield $D\Phi_\varepsilon^*(p)=p+\varepsilon V(p)+O(\varepsilon^2|p|)$, with one-homogeneous $V$, uniformly in angle. The $O(\varepsilon^2)$ cell tilt is absorbed in the remainder. Hence the exact interface $y=Y_\varepsilon(x,z)$ defined above satisfies, in $B_{2r}$, $|Y_\varepsilon|\leq C\varepsilon r$. The cap thresholds are $\pm t_\varepsilon$, where $t_\varepsilon=\lambda h_\varepsilon(n_\beta)$ and $|t_\varepsilon-\lambda|\leq C\varepsilon|\lambda|$.
		
		If $|\lambda|$ exceeds a sufficiently large multiple of $r$, the local captured set is unchanged: both cap conditions are false for $\lambda>0$ and true for $\lambda<0$. Otherwise $|\lambda|\leq Cr$. Equation~\eqref{eq:cap-indicator-difference} shows that the captured indicators can differ inside $B_{2r}$ only where $|y|\leq C\varepsilon r$ or $\bigl||x|-|\lambda|\bigr|\leq C\varepsilon r$. These strips have total volume at most $C_{\mathcal I}\varepsilon r^d$, so
		\[
		\sup_{(\beta,\lambda)\in[0,2\pi]\times \mathcal I}
		|M_\varepsilon^{\mathrm{in},r}(\beta,\lambda)-M_0^{\mathrm{in},r}(\beta,\lambda)|
		\leq C_{\mathcal I}\varepsilon r^d.
		\]
		Combining the inner and outer estimates gives
		\[
		\limsup_{\varepsilon\to0^+}\sup_{(\beta,\lambda)\in[0,2\pi]\times \mathcal I}\frac{|R_\varepsilon(\beta,\lambda)|}{\varepsilon}\leq C_{\mathcal I}r^d.
		\]
		Letting $r\to0^+$ proves uniformity. Since $b_d(\lambda)>0$, the continuous function $g_d(\lambda):=\sqrt{a_d(\lambda)^2+b_d(\lambda)^2}-|a_d(\lambda)|$ has a positive minimum on $\mathcal I$. Hence the strict deficit persists uniformly on $\mathcal I$ for sufficiently small $\varepsilon>0$.
		
		Finally, $n_\beta\cdot v_\varepsilon=-\varepsilon^2<0$, so Theorem~\ref{rev:thm:facet} realizes these points and normals on a smooth strictly convex obstacle with the same cells and captured image.
	\end{proof}
	
	\smallskip
	
	\begingroup\small
	\bibliographystyle{amsplain}
	\bibliography{citations}
	\endgroup
	
	\bigskip
	
	\begingroup
	\small
	\addressblock{Instituto de Ciencias de la Ingenier\'ia, Universidad de O'Higgins (UOH). Avenida del Libertador Bernardo O'Higgins 611, Rancagua, Chile.}
	\EmailB{lisbeth.carrero@postdoc.uoh.cl}
	
	\EmailB{andres.zuniga@uoh.cl}
	
	\bigskip
	
	\addressblock{Departamento de Matem\'atica, Universidad T\'ecnica Federico Santa Mar\'ia (USM), Avenida Espa\~na 1680, Valpara\'iso, Chile.}
	
	\EmailB{alexander.quaas@usm.cl} 
	\endgroup
	
\end{document}